\documentclass[11pt,reqno]{amsart}
\usepackage[margin=0.7in]{geometry}
\usepackage{amsmath,amssymb,amsthm,graphicx,amsxtra, setspace}
\usepackage[utf8]{inputenc}
\usepackage{mathrsfs}
\usepackage{hyperref}
\usepackage{upgreek}
\usepackage{mathtools}
\usepackage[dvipsnames]{xcolor}
\usepackage[mathcal]{euscript}
\usepackage{amsmath,units}
\usepackage{pgfplots}
\usepackage{tikz}
\usepackage{verbatim}
\allowdisplaybreaks

\usepackage{fontenc}
\usepackage{textcomp}
\usepackage{marvosym}
\usepackage{eurosym}

\usepackage[pagewise]{lineno}

\DeclareMathAlphabet{\mathpzc}{OT1}{pzc}{m}{it}

\usepackage[cyr]{aeguill}

\colorlet{darkblue}{blue!50!black}

\hypersetup{
	colorlinks,%
	citecolor=blue,%
	filecolor=Red,%
	linkcolor=red,%
	urlcolor=blue,%
	pdfnewwindow=true,%
	pdfstartview={FitH}
}

\newtheorem{theorem}{Theorem}[section]

\newtheorem{proposition}[theorem]{Proposition}

\newtheorem{definition}[theorem]{Definition}

\newtheorem{remark}[theorem]{Remark}

\newtheorem{hypothesis}[theorem]{Hypothesis}

\allowdisplaybreaks

\let\originalleft\left
\let\originalright\right
\renewcommand{\left}{\mathopen{}\mathclose\bgroup\originalleft}
\renewcommand{\right}{\aftergroup\egroup\originalright}

\let\emptyset\varnothing

\renewcommand{\d}{\/\mathrm{d}\/}

\def\w{\textbf{W}^{\varepsilon}_{{\theta}^{\varepsilon}}}

\def\L{\mathbb{L}}
\def\A{\mathrm{A}}
\def\I{\mathrm{I}}

\def\C{\mathrm{C}}
\def\f{\boldsymbol{f}}

\def\B{\mathrm{B}}
\def\D{\mathrm{D}}
\def\y{\boldsymbol{y}}
\def\q{\boldsymbol{q}}

\def\x{\boldsymbol{x}}

\def\g{\boldsymbol{g}}

\def\z{\boldsymbol{z}}
\def\v{\boldsymbol{v}}
\def\V{\mathbb{v}}
\def\w{\boldsymbol{w}}
\def\W{\mathrm{W}}

\def\Q{\mathrm{Q}}

\def\N{\mathbb{N}}

\def\V{\mathbb{V}}
\def\wi{\widetilde}
\def\Q{\mathrm{Q}}
\def\u{\mathrm{U}}
\def\P{\mathrm{P}}
\def\u{\boldsymbol{u}}
\def\H{\mathbb{H}}

\newcommand{\eps}{\varepsilon}

\newcommand{\R}{\mathbb{R}}

\renewcommand{\d}{\/\mathrm{d}\/}

\newcommand{\Addresses}{{
		\footnote{
			
			\noindent \textsuperscript{1,2,3}Department of Mathematics, Indian Institute of Technology Roorkee-IIT Roorkee,
			Haridwar Highway, Roorkee, Uttarakhand 247667, INDIA.\par\nopagebreak
			\noindent  \textit{e-mail:} \texttt{Manil T. Mohan: maniltmohan@ma.iitr.ac.in, maniltmohan@gmail.com.}
			
			\textit{e-mail:} \texttt{Kush Kinra: kkinra@ma.iitr.ac.in.}
			
			\textit{e-mail:} \texttt{Sagar Gautam: sagar\_g@ma.iitr.ac.in.}
			
			\noindent \textsuperscript{*}Corresponding author.
			
			\textit{Key words:}  Convective Brinkman-Forchheimer extended Darcy equations, torus, monotone operators, stabilization, feedback control.
			
			Mathematics Subject Classification (2020): Primary 49J20, 49N35, 93D15; Secondary 35Q35, 76D03.

}}}

\begin{document}
	
	\title[Nonlinear feedback controllers for $\mathrm{CBFeD}$ equations]{Feedback stabilization of Convective Brinkman-Forchheimer Extended Darcy equations
		\Addresses}
	\author[S. Gautam, K. Kinra and M. T. Mohan]
	{Sagar Gautam\textsuperscript{1}, Kush Kinra\textsuperscript{2} and Manil T. Mohan\textsuperscript{3*}}
	
	\maketitle
	
	\begin{abstract}
	In this article, the following controlled convective Brinkman-Forchheimer extended Darcy (CBFeD) system is considered in a $d$-dimensional torus:
	\[
	\frac{\partial\boldsymbol{y}}{\partial t}-\mu \Delta\boldsymbol{y}+(\boldsymbol{y}\cdot\nabla)\boldsymbol{y}+\alpha\y+\beta\vert \boldsymbol{y}\vert ^{r-1}\boldsymbol{y}+\gamma\vert \boldsymbol{y}\vert ^{q-1}\boldsymbol{y}+\nabla p=\boldsymbol{g}+\u,\ \nabla\cdot\boldsymbol{y}=0, 
	\]
	 where $d\in\{2,3\}$, $\mu,\alpha,\beta>0$, $\gamma\in\mathbb{R}$, $r,q\in[1,\infty)$ with $r>q\geq 1$. We prove the exponential stabilization of CBFeD system by finite- and infinite-dimensional feedback controllers. The solvability of the controlled problem is achieved by using  the abstract theory of $m$-accretive operators and density arguments. As an application of the above solvability result, by using infinite-dimensional feedback controllers, we demonstrate exponential  stability results such that the solution preserves an invariance condition for a given closed and convex set.  By utilizing the unique continuation property of controllability for finite-dimensional systems, we construct a finite-dimensional feedback controller which   exponentially stabilizes CBFeD system locally, where the control is localized in a smaller subdomain. Furthermore, we establish the local exponential stability of CBFeD system  via proportional controllers. 
	\end{abstract}

	\section{Introduction} 

The mathematical modeling and analysis of fluid flows through porous media is an active research field, as it covers some challenging problems in engineering and applied sciences. A porous medium or a porous material is defined as a solid (often called a matrix) permeated by an interconnected network of pores (voids) filled with a fluid (e.g., air or water). We consider a porous medium that is fully saturated with a fluid (cf. \cite{TCY}). Darcy's law is the subject of a significant amount of research in the porous media.  It is an empirical law which describes the linear relationship between fluid flow rate and pressure gradient. Mathematically, it says that  $\y_f=-\frac{k}{\nu}\nabla p,$ where $\y_f$ is the Darcy velocity, $k$ is the permeability of the fluid, $\nu>0$ is the dynamic viscosity and $p$ is the pressure in the fluid flow (cf. \cite{MTT}). Darcy's law is valid for the fluid flows with low Reynolds number (for example, oil reservoir and ground water flow). Nevertheless, nonlinearities appear with increasing Reynolds numbers, making it unreliable to describe fluid flow through porous media. For example fluid flow with high velocity, molecular and ionic effect, real fluids having non-Newtonian behavior (such as polymers and suspensions), compressible flow (such as high pressure gas reservoirs), etc.  Forchheimer proposed a modification for the fluid flows in porous media with high Reynolds number (for example, flow through packed beds) which suggests that Darcy's law still applicable but with an additional nonlinear term in order to account the pressure gradient. The Darcy's Forchheimer law describes the relationship between the pressure gradient and flow rate in a nonlinear manner at sufficiently high flow rates, that is, $\nabla p=-\frac{\nu}{k}\y_f-\gamma\rho_f|\y_f|\y_f,$ where $\y_f$ is the Forchheimer velocity, $\gamma>0$ is the Forchheimer coefficient and $\rho_f$ is the density (cf. \cite{MTT}). However, the quadratic nonlinearity in the Forchheimer equation can be further generalized to incorporate some additional nonlinear terms from the mathematical point of view.  Such a generalization, which takes into account the effects of both viscous resistance and inertial resistance, is called Brinkman-Forchheimer equations. These equations are used in several real world phenomena, for example, theory of non-Newtonian fluids (cf. \cite{avs}) and theory of tidal dynamics (cf. \cite{rgg,all,MT6}), etc. By adding the convective term $(\y\cdot\nabla)\y$, coming from the Navier-Stokes equations (NSE) to the Brinkman-Forchheimer equations, we obtain the convective-Brinkman Forchheimer equations (CBF). In this work, we consider the further generalization of CBF equations by taking into account the pumping term which also having similar nonlinearity as in the CBF equations but with opposite sign. 

\subsection{The Model}
For $\mathrm{L}>0$, let us consider a $d$-dimensional torus $\mathbb{T}^d=\left(\frac{\R}{\mathrm{L}\mathbb{Z}}\right)^d$ ($d=2,3$), which is identified with the $d$-dimensional cube $[0,\mathrm{L}]^d$ (see \cite[Chapter 3]{cf}). In this article, we discuss the stabilization of the mathematical generalizaion of CBF equations, known as \emph{convective Brinkman-Forchheimer extended Darcy (CBFeD) equations} (cf. \cite{MTT}) on $\mathbb{T}^d$ through feedback controllers. The CBFeD equations describe the motion of incompressible fluid flows in a saturated porous medium (\cite{MTT}). In this work, we consider the following controlled CBFeD equations on $\mathbb{T}^d$:
\begin{equation}\label{1}
	\left\{
	\begin{aligned}
		\frac{\partial \y}{\partial t}-\mu \Delta\y+(\y\cdot\nabla)\y+\alpha\y+\beta|\y|^{r-1}\y +\gamma|\y|^{q-1}\y+\nabla p&=\g+\u, \ &&\text{ in } \ \mathbb{T}^d\times(0,\infty), \\ \nabla\cdot\y&=0, \ &&\text{ in } \ \mathbb{T}^d\times(0,\infty), \\
		\y(0)&=\y_0 \ &&\text{ in } \ \mathbb{T}^d,
	\end{aligned}
	\right.
\end{equation}	
where $\y(x,t):\mathbb{T}^d\times[0,\infty)\to\R^d$, $p(x,t):\mathbb{T}^d\times[0,\infty)\to\R$ and $\g(x,t):\mathbb{T}^d\times[0,\infty)\to\R^d$ represent the velocity field, pressure field and external forcing, respectively, at time $t$ and position $x$. The distributed control acting on the system is denoted by $\u(\cdot)$. Moreover, $\y(\cdot,\cdot)$, $p(\cdot,\cdot)$ and $\g(\cdot,\cdot)$ satisfy the following periodic conditions:
\begin{align}\label{2}
	\y(x+\mathrm{L}e_{i},\cdot) = \y(x,\cdot), \ p(x+\mathrm{L}e_{i},\cdot) = p(x,\cdot) \ \text{and} \ \g(x+\mathrm{L}e_{i},\cdot) = \g(x,\cdot),
\end{align}
for every $x\in\R^{d}$ and $i\in\{1,\ldots,d\},$ where $\{e_{1},\dots,e_{d}\}$ is the canonical basis of $\R^{d}.$ The constant $\mu>0$ denotes the \emph{Brinkman coefficient} (effective viscosity). The constants $\alpha, \beta>0$ are due to Darcy-Forchheimer law and we call $\alpha$ and $\beta$ as \emph{Darcy} (permeability of the porous medium) and \emph{Forchheimer} (proportional to the porosity of the material) coefficients, respectively. The nonlinear term $\gamma|\y|^{q-1}\y$ appearing in \eqref{1} serves as damping for $\gamma>0$, and as pumping for $\gamma<0$. The parameters $r,q\in[1,\infty)$ with $r>q$ are known as absorption and pumping exponents, respectively. The exponent $r=3$ is known as critical exponent. Note that the system \eqref{1} has same scaling as NSE (see for instance \cite{KWH}) only when $\alpha=\gamma=0$ and $r=3$. It can be easily seen that, for $\alpha=\beta=\gamma=0$, the system \eqref{1} reduces to the classical NSE. Furthermore, we refer the case $r<3$ as \emph{subcritical} and $r>3$ as \emph{supercritical} (or fast growing nonlinearities). 
We refer to \cite{lim2,lim1} for a discussion on the formulation, validity, and limitations of Brinkman-Forchheimer extended Darcy's equations.

\subsection{Stabilization}
In the literature, the mathematical theory of stabilization of equilibrium solution to Newtonian fluid flows is extensively used as a principal tool to eliminate or attenuate the turbulence (cf. \cite{VB6}). 
	Let us consider the following stationary system corresponding to the CBFeD system \eqref{1}:
\begin{equation}\label{stability}
	\left\{
	\begin{aligned}
-\mu \Delta\y_e+(\y_e\cdot\nabla)\y_e+\alpha\y_e+\beta|\y_e|^{r-1}\y_e +\gamma|\y_e|^{q-1}\y+\nabla p_e&=\g_e,  &\text{ in } \ \mathbb{T}^d, \\ \nabla\cdot\y_e&=0,  &\text{ in } \ \mathbb{T}^d.
	\end{aligned}
	\right.
\end{equation}	
The solvability of the system \eqref{stability} has been discussed in  Appendix \ref{SP}. In this work, we are interested in the exponential stabilization of the stationary (or equilibrium) solution $\y_e$ to the system \eqref{1} by using infinite and finite-dimensional feedback controllers such that $\y(t)-\y_e$ satisfies some invariance condition for all $t\geq0$. We say that the controller $\u(\cdot)$ exponentially stabilizes the stationary solution $\y_e$ if there is a ``solution'' $\y(\cdot)$ of the system \eqref{1} such that it satisfies
\begin{align*}
	\lim\limits_{t\to\infty} e^{\delta t}\|\y(t)-\y_0\|_{\L^2({\mathbb{T}^d})}=0, \text{ for some } \delta>0.
\end{align*}
By the substitution $\y(\cdot)$ to $\y(\cdot)+\y_e$, one can reduce the problem of stabilization of $\y_e$ into the null stabilization  of the following system:
\begin{equation}\label{expstab}
	\left\{
	\begin{aligned}
		\frac{\partial \y}{\partial t}-\mu\Delta\y +((\y+\y_e)\cdot\nabla)(\y+\y_e)-(\y_e\cdot\nabla)\y_e+\alpha\y&+\beta(|\y+\y_e|^{r-1}(\y+\y_e)-|\y_e|^{r-1}\y_e) \\+\gamma(|\y+\y_e|^{q-1}(\y+\y_e)-|\y_e|^{q-1}\y_e)+\nabla(p-p_e)&=\g-\g_e+\u, \   \text{ in } \ \mathbb{T}^d\times(0,\infty), \\ \nabla\cdot\y&=0, \  \hspace{16mm}\text{ in }\ \mathbb{T}^d\times(0,\infty).
	\end{aligned}
	\right.
\end{equation}	
Keeping these goals in mind, we take control $\u(\cdot)\in\Xi+\partial\I_{\mathcal{K}}(\cdot)$ and consider the following inclusion problem associated with \eqref{expstab}:
	\begin{equation}\label{a}
		\left\{
		\begin{aligned}
				\frac{\partial \y}{\partial t}-\mu\Delta\y +((\y+\y_e)\cdot\nabla)(\y+\y_e)-(\y_e\cdot\nabla)\y_e+\alpha\y+&\beta(|\y+\y_e|^{r-1}(\y+\y_e)-|\y_e|^{r-1}\y_e) \\+\gamma(|\y+\y_e|^{q-1}(\y+\y_e)-|\y_e|^{q-1}\y_e)+\nabla (p-p_e)+\Xi\y+\partial\I_{\mathcal{K}}(\y)&\ni\g-\g_e, \ \text{ in } \ \mathbb{T}^d\times(0,\infty), \\ \nabla\cdot\y&=0, \ \hspace{8.5mm}\text{ in } \ \mathbb{T}^d\times(0,\infty),\\
			\y(0)&=\y_0, \hspace{8.1mm}\text{ in } \ \mathbb{T}^{d},
		\end{aligned}
		\right.
	\end{equation}
where $\I_{\mathcal{K}}$, the indicator function of a closed and convex set $\mathcal{K}$ satisfying  Hypothesis \ref{AssupK} below, is a lower semicontinuoius proper convex function. The subdifferential of the indicator function, that is, $\partial\I_{\mathcal{K}}(\cdot)$ is a maximal monotone operator and it is given by \eqref{subb}. Also, $\Xi:\mathbb{L}^2(\mathbb{T}^d)\to\mathbb{L}^2(\mathbb{T}^d)$ is a continuous linear operator, and the choice of $\Xi$ will be made precise later. The subdifferential of a lower semicontinuous proper convex function is a maximal monotone susbet of $\mathbb{L}^2(\mathbb{T}^d)\times\mathbb{L}^2(\mathbb{T}^d)$ and it connects the bridge between the theory of nonlinear maximal monotone operators and convex analysis. We refer to \cite{VB1,VB2}, for basic properties and concepts pertaining subdifferentials.

\subsection{Literature survey}
 Due to Leary and Hopf (cf. \cite{Hopf,Leray}), the existence of at least one weak solution of 3D NSE is known. But the uniqueness is still a challenging open problem for the mathematical community. Therefore,  several mathematicians came up with various modifications of 3D NSE (cf. \cite{SNA,SNA1,ZCQJ,KT2,ZZXW}, etc.). The authors in \cite{SNA,SNA1} introduced the NSE modified with an absorption term $\beta|\y|^{r-1}\y$ with $r>0$ and proved the existence of weak solution for any $d\geq2$ and uniqueness for $d=2$. In the literature, modification of the classical NSE with the damping term $\alpha\y+\beta|\y|^{r-1}\y$ is known as the \emph{convective Brinkman-Forchheimer equations} (cf. \cite{sKM,KWH,MT1}, etc.). In \cite{MTT}, the authors considered the classical 3D NSE with damping $\beta|\y|^{r-1}\y$ as well as pumping $\gamma|\y|^{q-1}\y$ (without linear damping $\alpha\y$) and established the existence of weak solutions for $r>q$ and uniqueness for $r>3$. Moreover, they demonstrated the global existence of a unique strong solution for $r>3$ with the initial data in $\mathbb{H}^1$. Similar to 3D NSE, the existence of a unique global (in time) weak solution of 3D CBFeD with $r\in[1,3)$ (for any $\beta,\mu>0$) and $r=3$ (for $2\beta\mu<1$) is also an open problem.

 The monotone operators  are a class of nonlinear operators which provide a broad framework for the analysis of infinite-dimensional problems. Of special interest are the so-called \emph{maximal monotone operators} (or \emph{$m$-accretive operators} on Hilbert spaces identified with its own dual), which exhibit remarkable surjectivity properties. Further, for inclusion problems, once the operator becomes maximal, then one can use the Yosida approximation scheme, which is an important regularization technique for solving evolution systems (see \cite{VB1,OPHB}, etc.).

Control theory plays a cruiciual role in  stabilization of a differential system around a specific orbit using a feedback control. One of the main problems in controllability is to find an appropriate control which steers from an arbitrary initial state to an arbitrary final (desired) state. Even though the controllabilty problem is quite satisfactory from the mathematical point of view, most of the control problems in applications require the control of closed-loop form, as they are stable under perturbations (cf. \cite[Part 3]{coron}). By a closed-loop control, we mean at each instant, the control is a function of the state. In the stability problems, our main concern is to show that the solution to a given evolution equation approaches to an equilibrium solution as the time approaches to infinity (cf. \cite{VG1,VB7,VGW}, etc.). One of the main results obtained in this direction is that the equilibrium solutions to NSE  can be exponentially stabilized by finite- or infinite-dimensional feedback controllers with support in the interior of the domain or on the boundary (see \cite{VB14,VB8,VB6}, etc.). The author in \cite{AIL2} established the global stabilization of 2D and 3D NSE with weak intial data by using a feedback control distributed on the whole space. Further they discussed the local stability results also by using a finite-dimensional feedback control distributed on the subdomain (see \cite{AIL,AIL4}). The local exponential stabilization for incompressible 2D NSE with Navier slip boundary conditions was established in \cite{CL}. They constructed a feedback control to stabilize the linearized system about an equilibrium point which helps to stabilize nonlinear system locally by using the solution of the Lyapunov equation. The similar problem in the context of 3D NSE was studied by V. Barbu in \cite{VB10}, where  a spectral decomposition method had been used. The feedback stabilization of 2D and 3D NSE in bounded domains by using a state decomposition technique  is discussed in \cite{VBT}. A local exponential stabilization for 3D NSE via finite-dimensional feedback controls by using observability inequality and the regularizing property for the corresponding linearized equation is examined in  \cite{VSSRA}. The authors in \cite{Mb2} provides a  general framework for the stabilization of a class of nonlinear equations by using finite-dimensional feedback controllers and its applications to stabilize the equilibrium solution of 3D NSE. 

The author in \cite{Mb} proved an exponential stabilizability  result for 2D and 3D NSE in bounded domain, where they construct the feedback controller by using Lyapunov functional method. The stability of 2D and 3D NSE with slip-curl boundary conditions is discussed in \cite{VBL} by using proportional controllers which is supported locally. We refer to \cite{av2} and references therein for the stabilization in unbounded domains for 2D NSE, where the authors use the technique of impulse control. For an extensive study on boundary stabilization, we refer the readers to the articles \cite{VBLT,VBLT1,VB11,NR,Jpr,Sr} and  references therein. A general approach for studying controllability of PDEs through finite-dimensional projections can be found in the articles \cite{Agh,Agh1}. We refer to \cite{BA, LPT1}, etc. for further reading.

\subsection{Difficulties and approaches}
The main advantage for considering the CBFeD system \eqref{1} in a $d$-dimensional torus is as follows: In the torus $\mathbb{T}^d$, the Helmholtz-Hodge projection $\mathcal{P}$ and $-\Delta$ commutes (cf. \cite[Theorem 2.22]{JCR}). So, the equality 
\begin{align}\label{3}
	&\int_{\mathbb{T}^d}(-\Delta \boldsymbol{y}(x))\cdot|\boldsymbol{y}(x)|^{r-1}\boldsymbol{y}(x)\d x=\int_{\mathbb{T}^d}|\nabla \boldsymbol{y}(x)|^2|\boldsymbol{y}(x)|^{r-1}\d x +4\left[\frac{r-1}{(r+1)^2}\right]\int_{\mathbb{T}^d}|\nabla|\boldsymbol{y}(x)|^{\frac{r+1}{2}}|^2\d x,
\end{align}
is quite useful in obtaining regularity results. It is also noticed in the literature that the above equality may not be useful in  domains other than the whole domain or a $d$-dimensional torus (see \cite{KT2,MT1},  etc. for a detailed discussion). 

In order to use the abstract theory of $m$-accretive operators, the authors  in \cite{VBSS,sKM,AIL2, AIL}, etc. use the perturbation result \cite[Chapter II, Theorem 3.5]{VB1}. However, for $r>3,$ due to the presence of fast-growing nonlinearities in \eqref{1}, we are able to prove $m$-accretivity without using the perturbation result (see Step IV, Proposition \ref{prop3.1}). In our previous work (\cite{sKM}), we established the well-posedness of the CBF model with strong initial data (for instance, $\y_0\in\H^2(\mathbb{T}^d)$) by using the abstract theory of $m$-accretive operators and we discussed its applications in control problems. In the present work, we prove the existence and uniqueness of the CBFeD systems \eqref{1} with weak initial data (for instance, $\y_0\in\L^2(\mathbb{T}^d)$).

For the stability of the stationary solution corresponding to the system \eqref{1p44}, one needs the existence of a global solution for the initial data $\y_0\in\H.$ Due to the presence of fast growing nonlinearities $\beta|\y|^{r-1}\y$ in \eqref{a}, we are able to prove the global existence of a unique weak solution for $r\geq3$ for $d=3$ and $d=r=3$ with $2\beta\mu>1$ (Sec. \ref{thm}). Obtaining the global existence of a unique weak solution for the multivalued system \eqref{1p44} is a challenging task, as we are dealing with a subdifferential operator. The main difficulty arises when we take the inner product with the subdifferential operator $\partial\I_{\mathcal{K}}(\cdot)$. The authors in \cite{sKM,AIL2,AIL,AIL4} addressed this difficulty by using more regularity on the initial data and exploiting the regularity of strong solutions. 

In the literature of CBF equations in bounded domains, the linear damping $\alpha\y$ poses no additional difficulty and therefore it is discarded from the analysis by taking $\alpha=0$.   In \cite{MTT}, the authors ignored the linear damping term  and considered a generalization of the Darcy-Forchheimer law by taking into account the pumping term $\gamma|\y|^{q-1}\y$ with $\gamma<0$ only. The linear damping term in our study cannot be avoided and must remain as it is due to a technical issue with the solvability of the stationary system corresponding to the CBFeD system \eqref{1p44} (non-zero average condition of velocity field, see Appendix \ref{SP}).  Let us also point out that we cannot utilize Poincar\'e inequality in this case since the average of $\y$ might not be zero. Instead, we must deal with the full $\H^1$-norm. Moreover, $0$ is an eigenvalue of the Stokes operator $\A=-\mathcal{P}\Delta$ on the torus $\mathbb{T}^d$ (see Subsection \ref{linope}) and hence $\A$ is not an invertible operator. Therefore we have to work with the operator $\I+\A$ rather than $\A$. 

The author in \cite{AIL2} uses a projection operator $\mathrm{P}_{\mathcal{K}}$ to handle the convective nonlinearity $(\y\cdot\nabla)\y$ (see \cite[Theorem 2.1]{AIL2}). In the present article, the fast-growing nonlinearities $\beta|\y|^{r-1}\y$ help us to avoid this assumption, and one can tackle it by using an integration trick (see the proof of the main result in Sec. \ref{thm}). Let us now explain an additional difficulty arising in the proof of the main result. Since the initial data $\y_0\in\H$ only, one cannot obtain a convergence of the approximation of the subdifferential operator $(\partial\I_{\mathcal{K}})(\y(\cdot))$ by using the standard energy estimates. Making use of the abstract result \cite[Theorem 1.6]{VB2} (in particular \eqref{4p1.1} below), we overcome this difficulty and obtain the convergence of the subdifferential operator for all $t \in [\tau, T],$  in $\H$, for any $\tau > 0$  (see more regularity estimates in Sec. \ref{thm}). Note that $\tau>0$ plays a crucial role in the existence of  weak solutions. The strong convergence \eqref{stc1} is helpful in ensuring the behavior of the solution $\y(\tau)$ as $\tau\to0^+$ (see Def. \ref{wdefn}).

\subsection{Novelties} In this work, we are interested in a problem of finding those feedback controllers of the form $\Psi(\y)$, where $\Psi$ is a multivalued map such that the solution $\y(t)$ of \eqref{a} reaches to the equilibrium solution $\y_e$  as $t\to\infty$ and it remains inside a closed and convex set $\mathcal{K}$. That is, the solution of the system preserves some invariance condition also. These are highly applicable real life problems where one wants to find those feedback controllers which preserve some specific properties of the state under some controlled flow. One such example is that the enstrophy $\int_{\mathbb{T}^d}|\nabla\times\y(x)|^2\d x$  in a specific spatial region has to be preserved within a bound (see \cite{VBSS} for NSE). These feedback controllers help us to find such a bound on the evolution in time for the controlled system (\cite{VBSS,sKM}). Furthermore, we  also look for the feedback stabilization of CBFeD system where the control takes values in a finite-dimensional space and localized in a smaller region. 

Most of the feedback stabilization problems for NSE are considered for bounded domains, with control distributed over the entire domain or a subset of the domain. The authors in \cite{VBL} studied the feedback stabilization for NSE by using proportional controllers with support in an open subset of a bounded domain. Following this, we prove the exponential stabilization for the CBFeD system \eqref{1}, with control situated on $\Omega\subset\mathbb{T}^d$ such that $\mathbb{T}^d\setminus\overline\Omega$ is sufficiently thin (see \eqref{prop1.1} in Appendix \ref{PC}). We use the classical Rayleigh-Faber-Krahn inequality (see \cite[Section 5.4]{VB7}) to achieve this, which gives a lower bound for the first eigenvalue of the operator $\mu\A+\alpha\I$ defined on $\mathbb{T}^d\setminus\overline\Omega$, which is strictly greater than $0$. Note that the Stokes operator $\A$ has $0$ as an eigenvalue in our setting, so we are not able to use the Rayleigh-Faber-Krahn inequality (see \eqref{RFK} in Appendix \ref{PC}) for this operator.

\subsection{Organization of the paper} 
The rest of the paper is organized as follows: In the following section, we introduce a functional setting for  the theory of CBFeD model and recall some definitions and properties of linear, bilinear and nonlinear operators as well as some well-known facts. In Sec. \ref{sec1}, we setup the abstract formulation \eqref{1p44} of \eqref{1} and state a result concerning the solvability of CBFeD equations (Theorem \ref{mainT}). We provide the proof of Theorem \ref{mainT} in Sec. \ref{thm} by using the abstract theory of $m$-accretive operators and density arguments. The Minty-Browder technique is exploited to pass the weak limit. Sec. \ref{stb} is devoted to the main results on the stabilization of CBFeD equations \eqref{1p44} by using infinite- and finite-dimensional feedback controllers (Theorem \ref{thm3} and \ref{findim}).  In Sec. \ref{sec2}, we discuss the quasi-$m$-accretivity of the single-valued operator $\mathcal{M}(\cdot)$ (Proposition \ref{prop3.1}) as well the quasi-$m$-accretivity of the multivalued operator $\mathfrak{G}(\cdot)$ (Proposition \ref{prop3.3}). The existence and uniqueness of solutions for the stationary CBFeD system are discussed in Appendix \ref{SP}. Exponential stabilization via proportional controllers is examined in Appendix \ref{PC}. 

\section{Functional settings and operators}\label{fn}\setcounter{equation}{0}
In this section, we define some necessary function spaces which are frequently used in the rest of the paper, and linear and nonlinear operators which help us to obtain the abstract formulation of the inclusion system \eqref{a}. For our analysis, we adapt the functional framework from the work \cite{JCR1}. 

\subsection{Function spaces} 
Let $\C_{\mathrm{p}}^{\infty}(\mathbb{T}^d;\R^d)$ denote the space of all infinitely differentiable  functions $\y$ satisfying periodic boundary conditions $\y(x+\mathrm{L}e_{i},\cdot) = \y(x,\cdot)$, for $x\in \R^d$. \emph{We are not assuming the zero mean condition for the velocity field unlike the case of NSE, since the absorption term $\beta|\y|^{r-1}\y$ does not preserve this property (see \cite{MTT}). Therefore, we cannot use the well-known Poincar\'e inequality and we have to deal with the  full $\H^1$-norm.} The Sobolev space  $\H_{\mathrm{p}}^s(\mathbb{T}^d):=\mathrm{H}_{\mathrm{p}}^s(\mathbb{T}^d;\mathbb{R}^d)$ is the completion of $\C_{\mathrm{p}}^{\infty}(\mathbb{T}^d;\R^d)$  with respect to the $\H^s$-norm and the norm on the space $\H_{\mathrm{p}}^s(\mathbb{T}^d)$ is given by $$\|\y\|_{{\H}^s_{\mathrm{p}}}:=\left(\sum_{0\leq|\boldsymbol\alpha|\leq s}\|\D^{\boldsymbol\alpha}\y\|_{\mathbb{L}^2(\mathbb{T}^d)}^2\right)^{1/2}.$$ 	
It is known from \cite{JCR1} that the Sobolev space of periodic functions $\H_{\mathrm{p}}^s(\mathbb{T}^d)$ is same as the space
   $$\H_{\mathrm{f}}^s(\mathbb{T}^d)=\left\{\y:\y=\sum_{k\in\mathbb{Z}^d}\y_{k}\mathrm{e}^{2\pi i k\cdot x /  \mathrm{L}},\ \overline{\y}_{k}=\y_{-k}, \  \|\y\|_{{\H}^s_\mathrm{f}}:=\left(\sum_{k\in\mathbb{Z}^d}(1+|k|^{2s})|\y_{k}|^2\right)^{1/2}<\infty\right\}.$$ We infer from \cite[Proposition 5.38]{JCR1} that the norms $\|\cdot\|_{{\H}^s_{\mathrm{p}}}$ and $\|\cdot\|_{{\H}^s_f}$ are equivalent. Let us define 
\begin{align*} 
	\mathcal{V}:=\{\y\in\C_{\mathrm{p}}^{\infty}(\mathbb{T}^d;\R^d):\nabla\cdot\y=0\}.
\end{align*}
We define the spaces $\H$ and $\widetilde{\L}^{p}$ as the closure of $\mathcal{V}$ in the Lebesgue spaces $\mathrm{L}^2(\mathbb{T}^d;\R^d)$ and $\mathrm{L}^p(\mathbb{T}^d;\R^d)$ for $p\in(2,\infty)$, respectively. We also define the space $\V$ as the closure of $\mathcal{V}$ in the Sobolev space $\mathrm{H}^1(\mathbb{T}^d;\R^d)$. Then, we characterize the spaces $\H$, $\widetilde{\L}^p$ and $\V$ with the norms  $$\|\y\|_{\H}^2:=\int_{\mathbb{T}^d}|\y(x)|^2\d x,\quad \|\y\|_{\widetilde{\L}^p}^p:=\int_{\mathbb{T}^d}|\y(x)|^p\d x\ \text{ and }\ \|\y\|_{\V}^2:=\int_{\mathbb{T}^d}(|\y(x)|^2+|\nabla\y(x)|^2)\d x,$$ respectively. 
 Let $(\cdot,\cdot)$ denote the inner product in the Hilbert space $\H$ and $\langle \cdot,\cdot\rangle $ represent the induced duality between the spaces $\V$  and its dual $\V'$ as well as $\widetilde{\L}^p$ and its dual $\widetilde{\L}^{p'}$, where $\frac{1}{p}+\frac{1}{p'}=1$. Note that $\H$ can be identified with its own dual $\H'$. From \cite[Subsection 2.1]{FKS}, we have that the sum space $\V'+\widetilde{\L}^{p'}$ is well defined and  is a Banach space with respect to the norm (see \cite{MT2} also)
\begin{align*}
	\|\y\|_{\V'+\widetilde{\L}^{p'}}&:=\inf\{\|\y_1\|_{\V'}+\|\y_2\|_{\wi\L^{p'}}:\y=\y_1+\y_2, \y_1\in\V' \ \text{and} \ \y_2\in\wi\L^{p'}\}\nonumber\\&=
	\sup\left\{\frac{|\langle\y_1+\y_2,\f\rangle|}{\|\f\|_{\V\cap\widetilde{\L}^p}}:\boldsymbol{0}\neq\f\in\V\cap\widetilde{\L}^p\right\},
\end{align*}
where $\|\cdot\|_{\V\cap\widetilde{\L}^p}:=\max\{\|\cdot\|_{\V}, \|\cdot\|_{\wi\L^p}\}$ is a norm on the Banach space $\V\cap\widetilde{\L}^p$. Also the norm $\max\{\|\y\|_{\V}, \|\y\|_{\wi\L^p}\}$ is equivalent to the norms  $\|\y\|_{\V}+\|\y\|_{\widetilde{\L}^{p}}$ and $\sqrt{\|\y\|_{\V}^2+\|\y\|_{\widetilde{\L}^{p}}^2}$ on the space $\V\cap\widetilde{\L}^p$. Furthermore, we have
$
(\V'+\widetilde{\L}^{p'})'=	\V\cap\widetilde{\L}^p \  \text{and} \ (\V\cap\widetilde{\L}^p)'=\V'+\widetilde{\L}^{p'},
$
Moreover, we have the continuous embeddings $$\V\cap\widetilde{\L}^p\hookrightarrow\V\hookrightarrow\H\cong\H'\hookrightarrow\V'\hookrightarrow\V'+\widetilde{\L}^{p'},$$ where the embedding $\V\hookrightarrow\H$ is compact. The following interpolation inequality is useful throughout the paper: Let $0\leq s_1\leq s\leq s_2\leq\infty$ and $0\leq\theta\leq1$ be such that $\frac{1}{s}=\frac{\theta}{s_1}+\frac{1-\theta}{s_2}$. Then for $\y\in\L^{s_2}(\mathbb{T}^d)$, we have
\begin{align*}
	\|\y\|_{\L^s}\leq\|\y\|_{\L^{s_1}}^{\theta}\|\y\|_{\L^{s_2}}^{1-\theta}.
\end{align*} 

\subsection{Linear operator}\label{linope}
Let $\mathcal{P}_p: \L^p(\mathbb{T}^d) \to\wi\L^p,$ $p\in[1,\infty)$ be the Helmholtz-Hodge (or Leray) projection operator (cf.  \cite{JBPCK,DFHM}, etc.).	Note that $\mathcal{P}_p$ is a bounded linear operator and for $p=2$,  $\mathcal{P}:=\mathcal{P}_2$ is an orthogonal projection (see \cite[Section 2.1]{JCR}). We define the Stokes operator 
\begin{equation*}
	\left\{
	\begin{aligned}
		\A\y&:=-\mathcal{P}\Delta\y=-\Delta\y,\;\y\in\D(\A),\\
	\D(\A)&:=\V\cap{\H}^{2}_\mathrm{p}(\mathbb{T}^d). 
	\end{aligned}
	\right.
\end{equation*}
For the Fourier expansion $\y=\sum\limits_{k\in\mathbb{Z}^d} e^{2\pi i k\cdot x} \y_{k} ,$ we calculate by using Parseval's identity
\begin{align*}
	\|\y\|_{\H}^2=\sum\limits_{k\in\mathbb{Z}^d} |\y_{k}|^2 \  \text{and} \ \|\A\y\|_{\H}^2=(2\pi)^4\sum_{k\in\mathbb{Z}^d}|k|^{4}|\y_{k}|^2.
\end{align*}
Therefore, we have 
\begin{align*}
	\|\y\|_{\H^2_\mathrm{p}}^2=\sum_{k\in\mathbb{Z}^d}|\y_{k}|^2+\sum_{k\in\mathbb{Z}^d}|k|^{4}| \y_{k}|^2=\|\y\|_{\H}^2+\frac{1}{(2\pi)^4}\|\A\y\|_{\H}^2\leq\|\y\|_{\H}^2+\|\A\y\|_{\H}^2.
\end{align*}
Moreover, by the definition of $\|\cdot\|_{\H^2_\mathrm{p}}$, we have $	\|\y\|_{\H^2_\mathrm{p}}^2\geq\|\y\|_{\H}^2+\|\A\y\|_{\H}^2$ and hence it is immediate that both the norms are equivalent and  $\D(\I+\A)=\H^2_\mathrm{p}(\mathbb{T}^d)$. 

Let us define a new operator $\mathscr{A}:\D(\mathscr{A})\subset \H\to\H$ by $\mathscr{A}:=\I+\A,$ where $\D(\mathscr{A})=\D(\A)$ and $\I$ is the identity operator on $\H$. It is clear that $\mathscr{A}$ is a non-negative self-adjoint and unbounded linear operator in $\H$. Then, we calculate
\begin{align*}
	\|\mathscr{A}^{\frac{1}{2}}\y\|_{\H}^2=(\mathscr{A}^{\frac{1}{2}}\y,\mathscr{A}^{\frac{1}{2}}\y)=(\y,\mathscr{A}\y)=
	\|\y\|_{\H}^2+\|\nabla\y\|_{\H}^2=\|\y\|_{\V}^2,
\end{align*}
which yields that $\D(\mathscr{A}^{\frac{1}{2}})=\V$. Moreover, the inverse $(\I+\A)^{-1}$ is also self-adjoint. The invertibilty of $\I+\A$ is follows from the fact that $\A=-\Delta$ is an $m$-accretive operator in $\H$ (see \cite[Chapter 1]{VB1}), which says that $\mathrm{R}(\I+\A)=\H$. This means for every $\f\in\H$, $\y\in\D(\A)$, the equation $\A\y+\y=\f$ 
has a unique solution. Also, we calculate
\begin{align*}
	\|\y\|_{\V}^2=\|\y\|_{\H}^2+\|\nabla\y\|_{\H}^2=\|\y\|_{\H}^2+\langle\A\y,\y\rangle=(\f,\y)\leq\|\f\|_{\H}\|\y\|_{\H}\leq\|\f\|_{\H}\|\y\|_{\V},
\end{align*}
which implies that $\|\y\|_{\V}\leq\|\f\|_{\H}$, that is, $\|(\I+\A)^{-1}\f\|_{\V}\leq\|\f\|_{\H}.$ 
Thus, $(\I+\A)^{-1}$ maps bounded subset of $\H$ into bounded subsets of $\V$. Since the embedding $\V\hookrightarrow\H$ is compact, therefore $(\I+\A)^{-1}$ is a compact operator on $\H$ and thus by the spectral mapping theorem, the spectrum of $\I+\A$ consists of an infinite sequence of eigenvalues $0<\lambda_1\leq\lambda_2\leq\ldots\lambda_k\leq\ldots,$ with $\lambda_k\to\infty$ as $k\to\infty$ such that  
\begin{align*}
	\mathscr{A}\w_k=\lambda_k\w_k, \  \ \text{for} \ k=1,2,\ldots,
\end{align*} 
where $\w_k$'s are the corresponding eigenfunctions in $\D(\mathscr{A})$ which forms an orthonormal basis for $\H$. We can express $\y$ in the Fourier expansion form as $\y=\sum\limits_{k\in\mathbb{Z}^d}\y_{k} \mathrm{e}^{2\pi i k\cdot x /  \mathrm{L}}$ or we can write
\begin{align*}
	\y=\sum\limits_{k\in\mathbb{Z}^d\setminus\{\boldsymbol{0}\}}\y_{k} \mathrm{e}^{2\pi i k\cdot\x /  \mathrm{L}} +\y_0 :=\wi{\y}+\y_0,
\end{align*}
where $\wi{\y}:=\sum\limits_{k\in\mathbb{Z}^d\setminus\{\boldsymbol{0}\}}\y_{k} \mathrm{e}^{2\pi i k\cdot x /  \mathrm{L}}$ and $\w_0$ is the element corresponding to $k=(0,\ldots,0)$. 
Then from \cite[Chapter 6]{JCR}, we calculate 
\begin{align*}
	\A\y=-\Delta\y=\left(\frac{4\pi^2}{\mathrm{L}^2}\right)\sum\limits_{k\in\mathbb{Z}^d}|k|^2\y_{k} \mathrm{e}^{2\pi i k\cdot x /  \mathrm{L}},
\end{align*}
and therefore $\mathscr{A}\y=(\I+\A)\y=\sum\limits_{k\in\mathbb{Z}^d}\left(1+\frac{4\pi^2}{\mathrm{L}^2}\right)|k|^2\y_{k} \mathrm{e}^{2\pi i k\cdot x /  \mathrm{L}}.$
It shows that $\y_k$ is an eigenfunction of $\mathscr{A}$ corresponding to the eigenvalue $1+\tilde{\lambda}_k$, where $\tilde{\lambda}_k=\frac{4\pi^2}{\mathrm{L}^2}|k|^2$. In particular, the operator $\mathscr{A}$ and the Stokes operator $\A$ have the same eigenfunctions corresponding to eigenvalues $1+\tilde{\lambda}_k$ and $\tilde{\lambda}_k$, respectively. For $k=\boldsymbol{0}$ and any constant vector $\w_0,$   we have $\mathscr{A}\w_0=(\I+\A)\w_0=\w_0=(1+0)\w_0$. Thus, $\w_0$ is an eigenfunction of $\mathscr{A}$ corresponding to the eigenvalue $\lambda_0=1$. Note that $0$ is an eigenvalue of the  Stokes operator $\A$.

From \cite[Chapter 6]{JCR}, the eigenfunctions of the operator $\mathscr{A}=\I+\A$ on the torus $\mathbb{T}^d$ are given by  
\begin{align}\label{ef}
	\w_k^{(s)}:=\sqrt{\frac{2}{\mathrm{L}^d}}\sin\left(2\pi k\cdot\frac{x}{\mathrm{L}}\right),  \ \w_k^{(c)}:=\sqrt{\frac{2}{\mathrm{L}^d}} \cos\left(2\pi k\cdot\frac{x}{\mathrm{L}}\right) \ \ \text{and} \ \w_k^{(0)}=\w_0,
\end{align}
with eigenvalues $\lambda_k=1+\frac{4\pi^2}{\mathrm{L}^2}|k|^2$. We can order these eigenvalues in a non-decreasing manner such that for every $\lambda_k $, $k\in\mathbb{Z}^d\setminus\{\boldsymbol{0}\}$, one can find a corresponding eigenvalue $\lambda_m$ for some appropriate $m\in\N$, with $\lambda_{m+1}\geq\lambda_{m}$ and the corresponding eigenfunction $\w_k=\w_m$ (cf. \cite[pp. 52]{corr}).

\subsection{Bilinear operator}
Let us define the \textit{trilinear form} $b(\cdot,\cdot,\cdot):\V\times\V\times\V\to\R$ by $$b(\y,\z,\w)=\int_{\mathbb{T}^d}(\y(x)\cdot\nabla)\z(x)\cdot\w(x)\d x=\sum_{i,j=1}^d\int_{\mathbb{T}^d}\y_i(x)\frac{\partial \z_j(x)}{\partial x_i}\w_j(x)\d x.$$ If $\y, \z$ are such that the linear map $b(\y, \z, \cdot) $ is continuous on $\V$, the corresponding element of $\V'$ is denoted by $\mathcal{B}(\y, \z)$. We also denote $\mathcal{B}(\y) = \mathcal{B}(\y, \y):=\mathcal{P}[(\y\cdot\nabla)\y]$.
An integration by parts yields 
\begin{equation*}
	\left\{
	\begin{aligned}
		b(\y,\z,\w) &=  -b(\y,\w,\z),\ &&\text{ for all }\ \y,\z,\w\in \V,\\
		b(\y,\z,\z) &= 0,\ &&\text{ for all }\ \y,\z \in\V.
	\end{aligned}
	\right.\end{equation*}
	
\subsection{Nonlinear operator}\label{nonlin}
Let us define an operator $\mathcal{C}_1(\y):=\mathcal{P}(|\y|^{r-1}\y)$ for $\y\in\V\cap\L^{r+1}$.  Since the projection operator $\mathcal{P}$ is bounded from $\H^1$ into itself (cf. \cite[Remark 1.6]{Te}), the operator $\mathcal{C}_1(\cdot):\V\cap\widetilde{\L}^{r+1}\to\V'+\widetilde{\L}^{\frac{r+1}{r}}$ is well-defined and we have $\langle\mathcal{C}_1(\y),\y\rangle =\|\y\|_{\widetilde{\L}^{r+1}}^{r+1}.$ Moreover, for all $\y\in\V\cap\L^{r+1}$, the map $\mathcal{C}_1(\cdot):\V\cap\widetilde{\L}^{r+1}\to\V'+\widetilde{\L}^{\frac{r+1}{r}}$ is Gateaux differentiable with Gateaux derivative given by 
\begin{align}\label{C}
	\mathcal{C}_1'(\y)\z&=\left\{\begin{array}{cl}\mathcal{P}(\z),&\text{ for }r=1,\\ \left\{\begin{array}{cc}\mathcal{P}(|\y|^{r-1}\z)+(r-1)\mathcal{P}\left(\frac{\y}{|\y|^{3-r}}(\y\cdot\z)\right),&\text{ if }\y\neq \mathbf{0},\\\mathbf{0},&\text{ if }\y=\mathbf{0},\end{array}\right.&\text{ for } 1<r<3,\\ \mathcal{P}(|\y|^{r-1}\z)+(r-1)\mathcal{P}(\y|\y|^{r-3}(\y\cdot\z)), &\text{ for }r\geq 3,\end{array}\right.
\end{align}
for all $\z\in\V\cap\widetilde{\L}^{r+1}$. For $r\geq3$, we also have the existence of second order Gateaux derivative given by (see \cite[Subsection 2.4]{MT3})
\begin{align}\label{C.}
	\mathcal{C}_1''(\y)(\z\otimes\w)&=\left\{\begin{array}{cl} \left\{\begin{array}{cc}(r-1)\mathcal{P}\{|\y|^{r-3}\left(\y\cdot\z)\z+(\y\cdot\z)\z+(\z\cdot\z)\y\right)\},&\text{ if }\y\neq \mathbf{0},\\\mathbf{0},&\text{ if }\y=\mathbf{0},\end{array}\right.&\text{ for } 3<r<5,\\ (r-1)\mathcal{P}\{|\y|^{r-3}\left(\y\cdot\z)\z+(\y\cdot\z)\z+(\z\cdot\z)\y\right)\}\\
		+(r-1)(r-3)\mathcal{P}\{|\y|^{r-5}(\y\cdot\z)(\y\cdot\z)\y\}, &\text{ for }r\geq5,\end{array}\right.
\end{align}
for all $\y,\z,\w\in\V\cap{\wi \L}^{r+1}$. For all $\y,\z\in\V\cap{\wi \L}^{r+1}$, we have (see \cite[Subsection 2.4]{MT4})
\begin{align}\label{C1}
	\langle\mathcal{C}_1(\y)-\mathcal{C}_1(\z),\y-\z\rangle&\geq \frac{1}{2}\||\y|^{\frac{r-1}{2}}(\y-\z)\|_{\H}^2+\frac{1}{2}\||\z|^{\frac{r-1}{2}}(\y-\z)\|_{\H}^2\geq \frac{1}{2^{r-1}}\|\y-\z\|_{\wi\L^{r+1}}^{r+1},
\end{align}
for all $r\in[1,\infty)$. Moreover, we define $\mathcal{C}_2(\y):=\mathcal{P}(|\y|^{q-1}\y)$ for $1\leq q<r$. The nonlinear operator $\mathcal{C}_2(\cdot)$ satisfies the similar properties as $\mathcal{C}_1(\cdot)$. 
\begin{remark}\label{C2}
	From \cite{MT2}, we have  for all $\y\in\D(\A)$ (see \cite[Subsection 3.5]{MT2}) 
	\begin{align*}
		\|\y\|_{\widetilde{\L}^{p(r+1)}}^{r+1}\leq C\int_{\mathbb{T}^d}|\nabla \y(x)|^2|\y(x)|^{r-1}\d x+C \int_{\mathbb{T}^d} |\y(x)|^{r+1}\d x,
	\end{align*}
	for all $p\in[2,\infty)$ for $d=2$ and $p=3$ for $d=3$. 
\end{remark}	

	\section{Solvability Result}\label{sec1}\setcounter{equation}{0} 
This work mainly investigates a feedback stabilization problem for the system \eqref{1} via finite- or infinite-dimensional feedback controllers. In order to achieve this, we first do the setting of a stabilization problem (see the inclusion problem \eqref{1p44}) and prove its well-posedness. We need the following assumption to obtain the global existence and uniqueness of weak solutions for the system \eqref{1p44}:
\begin{hypothesis}\label{AssupK}
	Let $\mathcal{K}\subset\H$  be a closed and convex set such that $\boldsymbol{0}\in\mathcal{K}$ and 
	\begin{align}\label{appl1.1}
		(\I+\lambda\A)^{-1}\mathcal{K}\subset\mathcal{K},  \  \text{for all} \  \lambda>0.
	\end{align}
\end{hypothesis}
We consider the indicator function $\I_{\mathcal{K}}:\H\to\overline{\R}$ (see \cite{VB2}) by 
\begin{align}\label{appl1.2}
	\I_{\mathcal{K}}(\x)=
	\begin{cases}
		{0},  &\text{if } \ \x\in\mathcal{K},\\
		+\infty, &\text{if } \  \x\notin\mathcal{K},
	\end{cases}
\end{align}
which is a lower-semicontinuous proper convex function, whose subdifferential is given by 
\begin{align}\label{subb}
	\partial\I_{\mathcal{K}}(\x)=
	\begin{cases}
		\emptyset, &\text{if } \ \x\notin\mathcal{K},\\
		\{\boldsymbol{0}\}, &\text{if } \  \x\in \mathrm{int}(\mathcal{K}),\\
		N_{\mathcal{K}}(\x)=\{\y\in\H:(\y,\x-\z)\geq0,\ \text{for all} \ \z\in\mathcal{K}\}, &\text{if } \ \x\in\mathrm{bdy}(\mathcal{K}),
	\end{cases}
\end{align}
where $\mathrm{int}(\mathcal{K})$ and $\mathrm{bdy}(\mathcal{K})$ denote the interior and boundary  of $\mathcal{K}$, respectively. Here $N_{\mathcal{K}}(\x)$ is the well-known Clark's normal cone to $\mathcal{K}$ at $\x$. It is readily observed that $\D(\partial\I_{\mathcal{K}})=\mathcal{K}$. From \cite[Theorem 2.1, pp. 62]{VB2}, we infer that the subdifferential operator $\partial\I_{\mathcal{K}}(\cdot)$ is a maximal monotone operator in $\H\times\H$. The regularization of $\I_{\mathcal{K}}$ is given by (see \cite[Chapter 2, Theorem 2.2]{VB2}) $$(\I_{\mathcal{K}})_{\lambda}(\x)=\frac{1}{2\lambda}\|\x-\P_{\mathcal{K}}(\x)\|_{\H}^2,$$
and its Gateaux derivative $(\partial\I_{\mathcal{K}})_{\lambda}(\x)=\frac{1}{\lambda}(\x-\P_{\mathcal{K}}(\x)),$ 
where $\P_{\mathcal{K}}:\L^2(\mathbb{T}^d)\to\mathcal{K}$ is the projection operator of $\x$ onto $\mathcal{K}$ which is equal to the resolvent  $(\I+\lambda\partial\I_{\mathcal{K}})^{-1}.$ It implies that the above derivative is equal to the Yosida approximation of $\partial\I_{\mathcal{K}},$ that is,       
\begin{align}\label{subdiff}
	(\partial\I_{\mathcal{K}})_{\lambda}(\x)=\frac{1}{\lambda}(\x-(\I+\lambda\partial\I_{\mathcal{K}})^{-1}(\x)),  \  \text{for all} \  \x\in\H.
\end{align}
Furthermore, $\|\partial\I_{\mathcal{K}}(\boldsymbol{0})\|_{\H}=0.$ Indeed, for $\boldsymbol{0}\in\mathrm{bdy}(\mathcal{K})$, we know that
\begin{align*}
	\|\partial\I_{\mathcal{K}}(\boldsymbol{0})\|_{\H}:=\inf\{\|\y\|_{\H}:\y\in\partial\I_{\mathcal{K}}(\boldsymbol{0})\}&=\inf\{\|\y\|_{\H}:\y\in N_{\mathcal{K}}(\boldsymbol{0})\}=0,
\end{align*} 
where we have used the fact that $\boldsymbol{0}\in N_{\mathcal{K}}(\boldsymbol{0})$. On the other hand, for $\boldsymbol{0}\in\mathrm{int}(\mathcal{K})$,  $\|\partial\I_{\mathcal{K}}(\boldsymbol{0})\|_{\H}=0$ is obvious.

\subsection{Abstract formulation and weak solution} We first write our problem in the abstract form. For this, let us define the continuous operators $\wi{\mathcal{B}}(\cdot):\V\to\V'$ and $\wi{\mathcal{C}_i}(\cdot):\V\cap\widetilde{\L}^{r_i+1}\to\V'+\widetilde{\L}^{\frac{r_i+1}{r_i}}$, where $r_1=r$ and $r_2=q$, by
\begin{align*}
	\wi{\mathcal{B}}(\y):=\mathcal{B}(\y+\y_e)-\mathcal{B}(\y_e) \text{ and } \wi{\mathcal{C}}_i(\y):=\mathcal{C}_i(\y+\y_e)-\mathcal{C}_i(\y_e), \  \text{for} \ \ i=1,2,
\end{align*}
where $\y_e\in\D(\A)$ is fixed. It is observed that the operators $\wi{\mathcal{B}}(\cdot)$ and $\wi{\mathcal{C}_i}(\cdot)$ take values from $\D(\A)$ into $\H$ also.  We now consider the CBFeD system  \eqref{a} in an abstract form perturbed by a subdifferential:
	\begin{equation}\label{1p44}
	\left\{
	\begin{aligned}
		\frac{\d \y(t)}{\d t}+\mu\A\y(t)+\wi{\mathcal{B}}(\y(t))+\alpha\y(t)+ \beta\wi{\mathcal{C}}_1(\y(t))&+\gamma\wi{\mathcal{C}}_2(\y(t))+ \mathfrak{F}\y(t)+\partial\I_{\mathcal{K}}(\y(t))\ni \f(t), \\
		\y(0)&=\y_0, 
	\end{aligned}
	\right.
\end{equation}
for  \mbox{a.e. $ t\in[0,T]$},  where $\mathfrak{F}\y:=\mathcal{P}\Xi\y$, the initial data $\y_0\in\H$ and forcing term $\f:=\mathcal{P}(\g-\g_e)\in\mathrm{L}^2(0,T;\H)$. Since $\mathfrak{F}$ is a continuous linear operator on $\H$, there exists a constant $M$ such that 
\begin{align}\label{F_Oper}
	\|\mathfrak{F}\y\|_{\H}\leq M \|\y\|_{\H},
\end{align}
for all $\y\in\H$. We are now in a position to provide the definition of solution for the system \eqref{1p44} in the weak sense. 

\begin{definition}\label{wdefn}
	A function $\y(\cdot)$ is said to be a \emph{weak solution} of the system \eqref{1p44} on the time interval $[0,T]$,  with initial data $\y_0\in\H$ if 
	\begin{align*}
		\y\in\mathrm{C}([0,T];\H)\cap\mathrm{L}^2(0,T;\V)\cap\mathrm{L}^{r+1}(0,T;\wi\L^{r+1}),
	\end{align*}
	with $\frac{\d\y}{\d t}\in\mathrm{L}^{\frac{r+1}{r}}(\tau,T;\H),$ for any $\tau>0$, and there exists $\xi(t)\in\partial\I_{\mathcal{K}}(\y(t))$ for a.e. $t\in[\tau,T]$  such that following equality holds:
	\begin{align*}
		-&\int_\tau^t (\y(s),\v_t(s))\d s+\mu\int_\tau^t (\nabla(\y(s)+\y_e),\nabla\v(s))\d s+\alpha\int_\tau^t (\y(s),\v(s))\d s\nonumber\\&+ \int_\tau^t \left(((\y(s)+\y_e)\cdot\nabla)(\y(s)+\y_e),\v(s)\right)\d s +\beta\int_\tau^t \left(|\y(s)+\y_e|^{r-1}(\y(s)+\y_e),\v(s)\right)\d s\nonumber\\& +\gamma\int_\tau^t \left(|\y(s)+\y_e|^{q-1}(\y(s)+\y_e),\v(s)\right)\d s+ \int_\tau^t (\mathfrak{F}(\y(s)),\v(s))\d s+\int_\tau^t (\xi(s),\v(s))\d s\nonumber\\ &\quad=-(\y(t)+\y_e,\v(t))+ (\y(\tau)+\y_e,\v(\tau))+\int_\tau^t (\f(s),\v(s))\d s+\int_\tau^t ((\y_e\cdot\nabla)\y_e,\v(s))
		\d s \nonumber\\&\quad\quad+\beta\int_\tau^t \left(|\y_e|^{r-1}\y_e,\v(s)\right)\d s+\gamma\int_\tau^t \left(|\y_e|^{q-1}\y_e,\v(s)\right)\d s ,
	\end{align*}
	for all $t\in[\tau,T]$ and for all functions $\v$ in the space of divergence free test functions 
	\begin{align*} 
\mathcal{V}_T&:=\{\v\in\C_{\mathrm{p}}^{\infty}(\mathbb{T}^d\times[\tau,T];\R^d):\nabla\cdot\v=0\},
	\end{align*}
	and the initial data is satisfied in the following sense:
	\begin{align*}
		\lim_{\tau\to 0^{+}}\|\y(\tau)-\y_0\|_{\H}=0. 
	\end{align*}
\end{definition}

Let us now state the main result of this paper regarding the solvability of the system \eqref{1p44}.
\begin{theorem}\label{mainT}
	Let $\mathcal{K}\subset\H$ satisfiy  Hypothesis \ref{AssupK}. Let $\y_0\in\mathcal{K}=\overline{\D(\A)\cap\mathcal{K}}^{\|\cdot\|_{\H}}$ and $\f\in\mathrm{L}^2(0,T;\H)$. Then, for $d=2,3$ with $r\geq3$ and $d=r=3$ with $2\beta\mu>1$, there exists a weak solution $\y(\cdot)$  of the system \eqref{1p44} in the sense of Definition \ref{wdefn}. Furthermore, such a  weak solution is unique.
\end{theorem} 
We prove Theorem \ref{mainT} in Sec. \ref{thm}. 

\begin{remark}\label{inv}
	The multivalued operator $\mathfrak{G}(\cdot) = \mu\mathrm{A} +\wi{\mathcal{B}}(\cdot)+\alpha\I+\beta\wi{\mathcal{C}}_1(\cdot)+\gamma\wi{\mathcal{C}}_2(\cdot)+\partial\I_{\mathcal{K}}(\cdot)$ is a maximal monotone operator with $\D(\mathfrak{G})=\D(\A)\cap\mathcal{K}$ (see Proposition \ref{prop3.3} below). Since $\D(\A)$ is dense subspace of $\H$ and $\mathcal{K}$ is a closed subset of $\H$, therefore from \cite[Remark and Theorem 11, pp. 110]{hb}, we have $\overline{\D(\A)\cap\mathcal{K}}^{\|\cdot\|_{\H}} =\overline{\D(\A)}^{\|\cdot\|_{\H}}\cap\mathcal{K}=\H\cap\mathcal{K}=\mathcal{K}$.
\end{remark}

\section{Proof of Theorem \ref{mainT}}\label{thm}\setcounter{equation}{0}
In this section, we prove the solvability result (Theorem \ref{mainT}) of the system \eqref{1p44}, that is, the existence and uniqueness of weak solutions of the system \eqref{1p44}. The proof is based on the existence of  unique strong solutions of some Yosida approximated system (see the system \eqref{4p1} below).

Let $\y_0\in\mathcal{K}$ and $\f\in\mathrm{L}^2(0,T;\H).$ Since the embeddings $\D(\A)\cap\mathcal{K}\hookrightarrow\H$ and $\W^{1,1}(0,T;\H)\hookrightarrow\mathrm{L}^2(0,T;\H)$ are dense, we obtain two sequences  $\{\y_0^j\}_{j\geq1}$ and $\{\f^j\}_{j\geq1}$ in $\D(\A)\cap\mathcal{K}$ and $\mathrm{W}^{1,1}(0,T;\H)$, respectively, such that
\begin{align}\label{density}
	\y_0^j\to\y_0 \  \text{ in } \ \H \  \text{and} \ \f^j\to\f \  \text{ in } \  \mathrm{L}^2(0,T;\H)\ \mbox{ as } \ j\to\infty.
\end{align}
Let us consider the following Yosida approximated system:
	\begin{equation}\label{4p1}
	\left\{
	\begin{aligned}
	\frac{\d \y^j(t)}{\d t}+\mu\A\y^j(t)+\alpha\y^j(t) +\wi{\mathcal{B}}(\y^j(t))+\beta\wi{\mathcal{C}}_1(\y^j(t))+ &\gamma\wi{\mathcal{C}}_2(\y^j(t))+\mathfrak{F}\y^j(t) +(\partial\I_{\mathcal{K}})_{\frac{1}{j}}(\y^j(t))\\ &= \f^j(t), \ \text{ a.e. } \ t\in[0,T], \\
		\y^j(0)&=\y_0^j. 
	\end{aligned}
	\right.
\end{equation}
where $(\partial\I_{\mathcal{K}})_{\frac{1}{j}}(\cdot)$ is the Yosida approximation $(\partial\I_{\mathcal{K}})_\lambda(\cdot)$ of the subdifferential $\partial\I_{\mathcal{K}}(\cdot)$ for $\lambda=\frac{1}{j}$. Form Proposition \ref{prop3.1}, we have shown that the operator $\mu\A+\wi{\mathcal{B}}(\cdot)+\beta\wi{\mathcal{C}_1}(\cdot)+\gamma\wi{\mathcal{C}_2}(\cdot)+(\alpha+\kappa)\I$ is maximal monotone in $\H$ for sufficiently large $\kappa>0$, while the operator $\mathfrak{F}(\cdot)$ is hemicontinous since it is linear and continuous from $\H$ to $\H$. Thus, from \cite[Corollary 1.4, pp 44]{VB2}, it implies that the operator $ \mu\mathrm{A} +\wi{\mathcal{B}}(\cdot)+\beta\wi{\mathcal{C}}_1(\cdot)+\gamma\wi{\mathcal{C}}_2(\cdot)+\mathfrak{F}(\cdot)+(\alpha+\kappa)\I$ is maximal monotone in $\H.$ Moreover, since the Yosida approximation $(\partial\I_{\mathcal{K}})_{\frac{1}{j}}(\cdot)$ is a demicontinuous and monotone operator,  from \cite[Corollary 1.5, pp 56]{VB2}, we infer that 
the operator $ \mu\mathrm{A}+\wi{\mathcal{B}}(\cdot)+\beta\wi{\mathcal{C}}_1(\cdot)+ \gamma\wi{\mathcal{C}}_2(\cdot)+\mathfrak{F}(\cdot)+(\partial\I_{\mathcal{K}})_{\frac{1}{j}}(\cdot)+(\alpha+\kappa)\I$ is maximal monotone in $\H\times\H.$ Hence, we have following result for \eqref{4p1} (cf. \cite[Proposition 3.4]{sKM}):
\begin{align}\label{regularity}
	\y^j\in\W^{1,\infty}(0,T;\H)\cap\mathrm{L}^{\infty}(0,T;\D(\A))\cap\C([0,T];\V),
\end{align}
$\y^j$ is right differentiable, $\frac{\d^+\y^j}{\d t}$ is right continuous and 
\begin{align}\label{4p1.1}
	&\frac{\d^+\y^j(t)}{\d t}+\mu\A\y^j(t)+\alpha\y^j(t) +\wi{\mathcal{B}}(\y^j(t))+\beta\wi{\mathcal{C}}_1(\y^j(t))+ \gamma\wi{\mathcal{C}}_2(\y^j(t))+\mathfrak{F}\y^j(t) +(\partial\I_{\mathcal{K}})_{\frac{1}{j}}(\y^j(t)) \nonumber\\&= \f^j(t), \ \text{ for all } \ t\in(0,T]. 
\end{align}
\vskip 2mm
\noindent
\textbf{\emph{Energy estimates.}} Taking the inner product in the first equation of the system \eqref{4p1} with $\y^j(\cdot)$ and using a similar calculation as in  \eqref{eyy}, we obtain  for a.e. $t\in[0,T]$
\begin{align*}
\frac{1}{2}\frac{\d}{\d t}\|\y^j(t)\|_{\H}^{2}+\frac{\mu}{2}\|\nabla\y^j(t)\|_{\H}^2 + \alpha\|\y^j(t)\|_{\H}^2+\frac{\beta}{2^r}\|\y^j(t)\|_{\wi\L^{r+1}}^{r+1}\leq\frac{1}{4}\|\f^j(t)\|_{\H}^2+k\|\y^j(t)\|_{\H}^2,
\end{align*}
where $k:=M+1+\varrho_{\frac{1}{2}}+\rho_{1}+\rho_{\frac{1}{2}}.$
Integrating above over $(0,t)$ and applying Gronwall's inequality, we find 
\begin{align*}
&\|\y^j(t)\|_{\H}^2+\mu\int_0^t \|\nabla\y^j(s)\|_{\H}^2 \d s+\alpha\int_0^t \|\y^j(s)\|_{\H}^2\d s + \frac{\beta}{2^{r-1}}\int_0^t \|\y^j(s)\|_{\wi\L^{r+1}}^{r+1} \d s\nonumber\\&\leq \left(\|\y_0^j\|_{\H}^2+\frac{1}{4}\int_0^t \|\f^j(s)\|_{\H}^2 \d s\right)e^{2k T},
\end{align*}
for all $t\in[0,T]$. Therefore, we have 
\begin{align}\label{4p2}
	\sup\limits_{j>0}\bigg(\sup_{t\in[0,T]}\|\y^j(t)\|_{\H}^2+\mu\int_0^T\|\nabla\y^j(s)\|_{\H}^2 \d s+\alpha\int_0^T \|\y^j(s)\|_{\H}^2\d s+\frac{\beta}{2^{r-1}}\int_0^T \|\y^j(s)\|_{\wi\L^{r+1}}^{r+1} \d s\bigg)\leq C,	
\end{align}
where $C=C(\|\y_0\|_{\H},\|\f\|_{\mathrm{L}^2(0,T;{\H})},T, k)>0$ is a constant independent of $j$. From the above energy estimate, we have following weak convergences:
\begin{equation}\label{wc1}
	\left\{
	\begin{aligned}
		\y^j&\xrightharpoonup{w^{\ast}}\ \y \ &&\text{ in } \  \mathrm{L}^{\infty}(0,T;\H),\  \	\mathfrak{F}\y^j \xrightharpoonup{w^{\ast}}\mathfrak{F}\y \ \text{ in } \  \mathrm{L}^{\infty}(0,T;\H), \\
		\y^j&\xrightharpoonup{w}\ \y \ &&\text{ in }  \  \mathrm{L}^2(0,T;\V)\cap\mathrm{L}^{r+1}(0,T;\wi\L^{r+1}). 
	\end{aligned}\right.
\end{equation}
 We rewrite \eqref{4p1} as 
\begin{align*}
	\frac{\d \y^j}{\d t}+\mu\A(\y^j+\y_e)+\mathcal{B}(\y^j+\y_e )+\alpha(\y^j+\y_e)&+ \beta\mathcal{C}_1(\y^j+\y_e)+\gamma\mathcal{C}_2(\y^j+\y_e)+\mathfrak{F}\y^j\\+(\partial\I_{\mathcal{K}})_{\frac{1}{j}}(\y^j)&=\f^j+\mathscr{G}(\y_e), 
\end{align*}
where $\mathscr{G}(\cdot):=\mu\A +\mathcal{B}(\cdot)+\alpha\I+\beta\mathcal{C}_1(\cdot)+\gamma\mathcal{C}_2(\cdot)$. Let us now take the inner product  with $\A\y^j(\cdot)$ in \eqref{4p1} and rearranging terms, we get for a.e. $t\in[0,T]$
\begin{align*}
	&\frac{1}{2}\frac{\d}{\d t}\|\nabla\y^j(t)\|_{\H}^2+\mu\|\A(\y^j(t)+\y_e)\|_{\H}^2 +(\mathcal{B}(\y^j(t)+\y_e),\A(\y^j(t)+\y_e))+\alpha\|\nabla(\y^j(t)+\y_e)\|_{\H}^2\nonumber\\& +((\partial\I_{\mathcal{K}})_{\frac{1}{j}}(\y^j(t)),\A(\y^j(t)))+\beta(\mathcal{C}_1(\y^j(t)+\y_e),\A(\y^j(t)+\y_e))\nonumber\\& +\alpha(\mathcal{C}_2(\y^j(t)+\y_e),\A(\y^j(t)+\y_e))+(\mathfrak{F}\y^j(t),\A\y^j(t))\nonumber\\&=(\f^j(t),\A\y^j(t)) +(\mathscr{G}(\y_e),\A\y^j(t))+\mu(\A(\y^j(t)+\y_e),\A\y_e)+(\mathcal{B}(\y^j(t)+\y_e),\A\y_e)\nonumber\\&\quad+\beta(\mathcal{C}_1(\y^j(t)+\y_e),\A\y_e)+\gamma(\mathcal{C}_2(\y^j(t)+\y_e),\A\y_e).
\end{align*}
We consider the following cases:
\vskip 2mm
\noindent
\textbf{Case I:} \emph{$r>3$ for $d=2,3$.} By using the Cauchy Schwarz and Young's inequalities and calculations similar to \eqref{3.2.7.0}-\eqref{e6} yields 
\begin{align*}
	&\frac{1}{2}\frac{\d}{\d t}\|\nabla\y^j\|_{\H}^2+\frac{\mu}{4}\|\A(\y^j+\y_e)\|_{\H}^2 +\frac{\beta}{2}\||\y^j+\y_e|^{\frac{r-1}{2}}\nabla(\y^j+\y_e)\|_{\H}^2+\alpha\|\nabla(\y^j+\y_e)\|_{\H}^2\nonumber\\&\leq\frac{2}{\mu}\|\f^j\|_{\H}^2+ \frac{2}{\mu}\|\mathscr{G}(\y_e)\|_{\H}^2+\frac{\mu}{4}\|\A\y_e\|_{\H}^2+(\varrho_{\frac{1}{4}}+\eta_1+\eta_2)\|\nabla(\y^j+\y_e)\|_{\H}^2+\frac{4M^2}{\mu}\|\y^j\|_{\H}^2\nonumber\\&\quad+C\|\y^j+\y_e\|_{\H}^{\frac{3(r+1)-2q}{3(r+1)-3q+1}}+C\|\y^j+\y_e\|_{\wi\L^{r+1}}^{\frac{(r+1)(r+3)}{r+7}}.
\end{align*}
Integrating the above expression over $(\tau,t)$, for $\tau>0$, we obtain 
\begin{align}\label{4p2.1}
	&\|\nabla\y^j(t)\|_{\H}^2+\frac{\mu}{2}\int_\tau^{t} \|\A(\y^j(s)+\y_e)\|_{\H}^2 \d s+
	\beta\int_\tau^{t}\||\y^j(s)+\y_e|^{\frac{r-1}{2}}\nabla(\y^j(s)+\y_e)\|_{\H}^2 \d s \nonumber\\& +\alpha\int_{\tau}^t \|\nabla(\y^j(s)+\y_e)\|_{\H}^2\d s\nonumber\\&\leq
	\|\nabla\y^j(\tau)\|_{\H}^2+\frac{4}{\mu}\int_\tau^{t}\|\f^j(s)\|_{\H}^2\d s+ \frac{4C(t-\tau)}{\mu}\|\A\y_e\|_{\H}^2+\frac{\mu(t-\tau)}{4}\|\A\y_e\|_{\H}^2\nonumber\\&\quad+2(\varrho_{\frac{1}{4}}+\eta_1+\eta_2) \int_\tau^{t} \|\nabla(\y^j(s)+\y_e)\|_{\H}^2 \d s+\frac{4M^2}{\mu}\int_\tau^t \|\y^j(s)\|_{\H}^2\d s\nonumber\\&\quad+C\int_\tau^t \|\y^j(s)+\y_e\|_{\H}^{\frac{3(r+1)-2q}{3(r+1)-3q+1}} \d s
+C\left(\int_\tau^t \|\y^j(s)+\y_e\|_{\wi\L^{r+1}}^{r+1}\d s\right)^{\frac{r+3}{r+7}}(t-\tau)^{\frac{4}{r+7}},
\end{align}
for all $t\in[\tau,T].$ Using the energy estimate \eqref{4p2}, we find
\begin{align}\label{eeA}
	&\|\nabla\y^j(t)\|_{\H}^2+\frac{\mu}{2}\int_\tau^{t} \|\A(\y^j(s)+\y_e)\|_{\H}^2 \d s+
	\beta\int_\tau^{t}\||\y^j(s)+\y_e|^{\frac{r-1}{2}}\nabla(\y^j(s)+\y_e)\|_{\H}^2 \d s \nonumber\\&\leq
	\|\nabla\y^j(\tau)\|_{\H}^2+\frac{4}{\mu}\int_\tau^{t}\|\f^j(s)\|_{\H}^2\d s+ C(t-\tau)\|\A\y_e\|_{\H}^2+C+C(t-\tau)^{\frac{4}{r+7}}
	\nonumber\\& \leq C\left[\|\nabla\y^j(\tau)\|_{\H}^2+(t-\tau)+(t-\tau)^{\frac{4}{r+7}}+1\right],
\end{align}
where $C=C(\|\y_0\|_{\H},\|\f\|_{\mathrm{L}^2(0,T;{\H})},\|\A\y_e\|_{\H},T,M)>0$ is a constant. 
  Now integrating \eqref{eeA} with respect to $\tau$ over $(0,t)$, we achieve
\begin{align*}
	t\|\nabla\y^j(t)\|_{\H}^2\leq C\left[\int_0^{t}\|\nabla\y^j(s)\|_{\H}^2\d s+t+ t^2 + t^{\frac{r+11}{r+7}}\right],
\end{align*}
which along with \eqref{4p2} gives
\begin{align}\label{4p2.2}
	\|\nabla\y^j(t)\|_{\H}^2&\leq C\left[1+\frac{1}{t}+t+t^{\frac{4}{r+7}}\right],
\end{align}
for all $t\in[\tau,T]$, where the constant $C=C(\|\y_0\|_{\H},\|\f\|_{\mathrm{L}^2(0,T;{\H})},\|\A\y_e\|_{\H},T,M,\mu)>0$ is independent of $j$. Therefore, from \eqref{4p2.1}-\eqref{4p2.2}, one can conclude that for all $t\in[\tau,T]$
	\begin{align}\label{4p2.b}
		&\|\nabla\y^j(t)\|_{\H}^2+\frac{\mu}{2}\int_\tau^{t} \|\A(\y^j(s)+\y_e)\|_{\H}^2 \d s+
		\beta\int_\tau^{t}\||\y^j(s)+\y_e|^{\frac{r-1}{2}}\nabla(\y^j(s)+\y_e)\|_{\H}^2 \d s \nonumber\\&\leq C\left[1+\frac{1}{\tau}+T+T^{\frac{4}{r+7}}\right].
	\end{align}
\vskip 2mm
\noindent
\textbf{Case II:} \emph{$d=r=3$ with $2\beta\mu>1$.} 
One can perform the similar calculations as we have done in \eqref{3.2.8.1}-\eqref{3.2.8.3} and then integrate over $(\tau,t)$ for $\tau>0$ to obtain 
\begin{align*}
	&\|\nabla\y^j(t)\|_{\H}^{2}+\mu(1-\theta)\int_\tau^{t} \|\A(\y^j(s)+\y_e)\|_{\H}^2 \d s +\frac{1}{2} \left(\beta-\frac{1}{2\theta\mu}\right)\int_\tau^{t}\||\y^j(s)+\y_e|\nabla(\y^j(s)+\y_e)\|_{\H}^2 \d s \nonumber\\&\leq C+\|\nabla\y^j(\tau)\|_{\H}^{2}+C(t-\tau)+C(t-\tau)^{\frac{3}{5}}\leq C\left[1+\frac{1}{\tau}+T+T^{\frac{2}{5}}\right],
\end{align*}
for all $t\in[\tau,T]$ and $0<\theta<1$, where we have used the energy estimates \eqref{4p2} for $r=3$. 
\vskip 2mm
\noindent
\textbf{\emph{Higher regularity estimates.}} Let us write the first equation of the system \eqref{4p1} for $t+h$ with $h>0$ as  
\begin{align*}
	&\frac{\d \y^j(t+h)}{\d t}+\mu\A\y^j(t+h)+\wi{\mathcal{B}}(\y^j(t+h))+\alpha\y^j(t+h)+ \beta\wi{\mathcal{C}}_1(\y^j(t+h))+\gamma\wi{\mathcal{C}}_2(\y^j(t+h))\nonumber\\&  +\mathfrak{F}\y^j(t+h)+(\partial\I_{\mathcal{K}})_{\frac{1}{j}}(\y^j(t+h))\nonumber\\&= \f^j(t+h), 
\end{align*}
for a.e. $t\in[0,T].$ Subtracting the above equation form the first equation of the system \eqref{4p1}, taking the inner product with 
$\y^j(\cdot+h)-\y^j(\cdot)$ and using the similar calculations as in \eqref{3.4}-\eqref{c2.2}, we derive  
\begin{align}\label{4pr}
	&\frac{1}{2}\frac{\d}{\d t}\|\y^j(t+h)-\y^j(t)\|_{\H}^2+\frac{\mu}{2}\|\nabla(\y^j(t+h)-\y^j(t))\|_{\H}^2 +\alpha\|\y^j(t+h)-\y^j(t)\|_{\H}^2\nonumber\\&\leq\|\f^j(t+h)-\f^j(t)\|_{\H}\|\y^j(t+h)-\y^j(t)\|_{\H} +\mathfrak{K}\|\y^j(t+h)-\y^j(t)\|_{\H}^2,
\end{align}
for a.e. $t\in[0,T]$, where $\mathfrak{K}:=M+\varrho_{\frac{1}{2}}+\rho_{1}+\rho_{\frac{1}{2}}$ and the constants $M$, $\varrho_{\frac{1}{2}}$, $\rho_{1}$, $\rho_{\frac{1}{2}}$ are same as given in \eqref{F_Oper}, \eqref{3.4}, \eqref{c2.2}, respectively. 
It is immediate from \eqref{4pr} that 
\begin{align*}
	\frac{\d}{\d t}\|\y^j(t+h)-\y^j(t)\|_{\H}\leq \|\f^j(t+h)-\f^j(t)\|_{\H}+\mathfrak{K}\|\y^j(t+h)-\y^j(t)\|_{\H},
\end{align*}
for a.e. $t\in[0,T].$ Using the Gronwall inequality, we obtain 
\begin{align}\label{5pr}
	\|\y^j(t+h)-\y^j(t)\|_{\H}\leq e^{\mathfrak{K}(t-\tau)}\left(\|\y^j(\tau+h)-\y^j(\tau)\|_{\H}+
	\int_\tau^t \|\f^j(s+h)-\f^j(s)\|_{\H} \d s\right),
\end{align}
for all $t\in[\tau,T].$ Dividing \eqref{5pr} by $h$ and taking limit as $h\to0$, we find
\begin{align}\label{4pr1}
	\left\|\frac{\d^+\y^j(t)}{\d t}\right\|_{\H}&\leq e^{\mathfrak{K}T}\left(\left\|\frac{\d^+\y^j(\tau)}{\d t}\right\|_{\H}+\int_0^T\left\|\frac{\d\f^j(s)}{\d t}\right\|_{\H}\d s\right). 
\end{align}
for all $t\in[\tau,T].$ Now, integrating \eqref{4pr1} with respect to $\tau$ on $(0,t)$, we derive
\begin{align*}
	\left\|\frac{\d^+\y^j(t)}{\d t}\right\|_{\H}\leq\frac{e^{\mathfrak{K}T}}{t} \left(\int_0^t\left\|\frac{\d^+\y^j(\tau)}{\d t}\right\|_{\H}\d\tau\right)+e^{\mathfrak{K}T} \int_0^T \left\|\frac{\d\f^j(s)}{\d t}\right\|_{\H}\d s,
\end{align*}
which implies  
\begin{align*}
	\left\|\frac{\d^+\y^j(t)}{\d t}\right\|_{\H}\leq\frac{e^{\mathfrak{K}T}}{\tau} \left(\int_0^t\left\|\frac{\d^+\y^j(\tau)}{\d t}\right\|_{\H}\d\tau\right)+e^{\mathfrak{K}T} \int_0^T \left\|\frac{\d\f^j(s)}{\d t}\right\|_{\H}\d s.
\end{align*}
Again we apply the Gronwall inequality to obtain
\begin{align}\label{4pr1.1}
	\left\|\frac{\d^+\y^j(t)}{\d t}\right\|_{\H}\leq\left(e^{\mathfrak{K}T} \int_0^T \left\|\frac{\d\f^j(s)}{\d s}\right\|_{\H}\d s\right)\exp \left(\frac{Te^{\mathfrak{K}T}}{\tau}\right),
\end{align}
for all $t\in[\tau,T].$ For $r=3$ with $2\beta\mu>1$, from Remark \ref{RK}, one obtains for a.e. $t\in[\tau,T]$
\begin{align*}
	&\frac{\d}{\d t}\|\y^j(t+h)-\y^j(t)\|_{\H}^2+\alpha\|\y^j(t+h)-\y^j(t)\|_{\H}^2\nonumber\\&+ \left(\beta-\frac{1}{2\mu}\right)\bigg[\||\y^j(t+h)+\y_e|(\y^j(t+h)-\y^j(t))\|_{\H}^2+\frac{1}{2}\||\y^j(t)+\y_e|(\y^j(t+h)-\y^j(t))\|_{\H}^2\bigg]\nonumber\\&\leq
	\|\f^j(t+h)-\f^j(t)\|_{\H}\|\y^j(t+h)-\y^j(t)\|_{\H}+\mathfrak{K}_1\|\y^j(t+h)-\y^j(t)\|_{\H}^2,
\end{align*}
where $\mathfrak{K}_1=M+\wi\rho_1+\wi\rho_2$. Then by performing similar calculations as we have done for $r>3$, we obtain
\begin{align}\label{4pr.11}
		\left\|\frac{\d^+\y^j(t)}{\d t}\right\|_{\H}\leq\left(e^{\mathfrak{K}_1T} \int_0^T \left\|\frac{\d\f^j(s)}{\d t}\right\|_{\H}\d s\right)\exp \left(\frac{Te^{\mathfrak{K}_1T}}{\tau}\right),
\end{align}
for all $t\in[\tau,T].$
\vskip 2mm
\noindent
\textbf{\emph{Uniform bounds for operators.}} 
 Using Agmon's and Young's inequalities, we calculate
\begin{align*}
\|\mathcal{B}(\y^j+\y_e)\|_{\H}^2&\leq C\left[	\|\y^j+\y_e\|_{\H}^{1-\frac{d}{4}}\|\nabla(\y^j+\y_e)\|_{\H}\|\y^j+\y_e\|_{\H^2_\mathrm{p}}^{\frac{d}{4}}\right]^2\nonumber\\&\leq C\left[
	\|\y^j+\y_e\|_{\H}^{1-\frac{d}{4}}\|\nabla(\y^j+\y_e)\|_{\H}(\|\A(\y^j+\y_e)\|_{\H}^{\frac{d}{4}}+\|\y^j+\y_e\|_{\H}^{\frac{d}{4}})\right]^2\nonumber\\&\leq C\left[
	\|\y^j+\y_e\|_{\H}^{2-\frac{d}{2}}\|\nabla(\y^j+\y_e)\|_{\H}^2\|\A(\y^j+\y_e)\|_{\H}^{\frac{d}{2}}
	+\|\y^j+\y_e\|_{\H}^2\|\nabla(\y^j+\y_e)\|_{\H}^2\right]\nonumber\\&\leq C\left[\|\nabla(\y^j+\y_e)\|_{\H}^2(\|\y^j+\y_e\|_{\H}^2+\|\A(\y^j+\y_e)\|_{\H}^2)\right].
\end{align*}
In view of \eqref{4p2} and \eqref{4p2.b}, we obtain
\begin{align}\label{4p4.3}
	\int_\tau^{t} \|\mathcal{B}(\y^j(s)+\y_e)\|_{\H}^2\d s&\leq
	C\left(1+\frac{1}{\tau}+t+t^{\frac{4}{r+7}}\right)^2,
\end{align}
for all $t\in[\tau,T].$ Moreover, by interpolation inequality, we infer
\begin{align*}
	\|\mathcal{C}_1(\y^j+\y_e)\|_{\H}^{\frac{r+1}{r}}&\leq\|\y^j+\y_e\|_{\wi\L^{2r}}^{r+1}\leq
	\|\y^j+\y_e\|_{\H}^{\frac{(r+3)(r+1)}{r(3r+1)}}\|\y^j+\y_e\|_{\wi\L^{3(r+1)}}^{\frac{3(r-1)(r+1)}{r(3r+1)}}.
\end{align*}
Applying H\"older's inequality, and using the energy estimate \eqref{4p2} and \eqref{4p2.b}, we find
\begin{align}\label{4p4.4}
	\int_\tau^t \|\mathcal{C}_1(\y^j(s)+\y_e)\|_{\H}^{\frac{r+1}{r}}\d s\leq C\left(\int_\tau^t \|\y^j(s)+\y_e\|_{\wi\L^{3(r+1)}}^{r+1}\d s\right)^ {\frac{3(r-1)}{r(3r+1)}}\leq C
	\left(1+\frac{1}{\tau}+t+t^{\frac{4}{r+7}}\right),
\end{align}
for all $t\in[\tau,T],$ where we have used the fact that $\frac{3(r-1)}{r(3r+1)}<1.$ Again, by interpolation inequality, we calculate
\begin{align*}
	\|\mathcal{C}_2(\y^j+\y_e)\|_{\H}^{\frac{q+1}{q}}&\leq\|\y^j+\y_e\|_{\wi\L^{2q}}^{q+1}\leq\|\y^j+\y_e\|_{\H}^{\frac{(q+1)(3(r+1)-2q)}{q(3(r+1)-2)}} \|\y^j+\y_e\|_{\widetilde{\L}^{3(r+1)}}^{\frac{3(r+1)(q^2-1)}{q(3(r+1)-2)}}.
\end{align*} 
Therefore, similar to \eqref{4p4.4}, we find
\begin{align}\label{4p4.5}
\int_\tau^t \|\mathcal{C}_2(\y^j(s)+\y_e)\|_{\H}^{\frac{q+1}{q}}\d s\leq C\left(1+\frac{1}{\tau}+t+t^{\frac{4}{r+7}}\right),
\end{align}
for all $t\in[\tau,T]$, where we have used $\frac{3(q^2-1)}{q(3r+1)}<1.$ Thus by using \eqref{4pr1.1} (for $r>3$), \eqref{4pr.11} (for $d=r=3$ with $2\beta\mu>1$) and \eqref{4p4.3}-\eqref{4p4.5}, we finally conclude from \eqref{4p1.1} that 
\begin{align}\label{4p4.6}
	\int_{\tau}^t \|(\partial\I_{\mathcal{K}})_{\frac{1}{j}}(\y^j(s))\|_{\H}^{\frac{r+1}{r}}\d s\leq C\left(1+\frac{1}{\tau}+t+t^{\frac{4}{r+7}}\right)^2,
\end{align}
for all $t\in[\tau,T].$ Let us rewrite the system \eqref{4p1} as
\begin{equation}\label{mb1}
	\left\{
	\begin{aligned}
		\frac{\d \y^j(t)}{\d t}+\mathscr{F}(\y^j(t))+(\partial\I_{\mathcal{K}})_{\frac{1}{j}}(\y^j(t))&= \f^j(t), \ \text{ a.e. } \ t\in[0,T],  \\
		\y^j(0)&=\y_0^j. 
	\end{aligned}
	\right.
\end{equation} 
where $\mathscr{F}(\cdot):=\mu\A+\wi{\mathcal{B}}+\alpha\I+\beta\wi{\mathcal{C}}_1(\cdot)+\gamma\wi{\mathcal{C}}_2(\cdot)+\mathfrak{F}$. From \eqref{4p4.3}-\eqref{mb1}, we find
\begin{align}\label{mb6}
	\int_\tau^t \|\mathscr{F}(\y^j(s))\|_{\H}^{\frac{r+1}{r}}\d s\leq C.
\end{align}
From \eqref{4p4.6}-\eqref{mb6}, making use of the Banach-Alaoglu theorem, we have  the following weak convergences as $j\to\infty$:
\begin{equation}\label{wc2}
	\left\{
	\begin{aligned}
		\mathscr{F}(\y^j)&\xrightharpoonup{w}\mathscr{F}_0 \  &&\text{in} \ \mathrm{L}^{\frac{r+1}{r}}(\tau,T;\H)\\
	(\partial\I_{\mathcal{K}})_{\frac{1}{j}}(\y^j)&\xrightharpoonup{w}\phi \ &&\text{ in } \ \mathrm{L}^\frac{r+1}{r}(\tau,T;\H),\\
	\frac{\d\y^j}{\d t}&\xrightharpoonup{w} \ \frac{\d\y}{\d t} \ &&\text{ in } \ 
	\mathrm{L}^{\frac{r+1}{r}}(\tau,T;\H).
	\end{aligned}\right.
\end{equation}
\vskip 2mm
\noindent
\textbf{\emph{Strong convergences}.} 
Let $\{\y^{j_k}_0\}_{k\geq1}$ and $\{\y^{j_l}_0\}_{l\geq1}$ be two subsequences of $\{\y^j_0\}_{j\geq1}$ such that \eqref{4p1} holds for $\y^{j_k}(\cdot)$ and $\y^{j_l}(\cdot)$ with the initial data $\y^{j_k}(0)=\y^{j_k}_0$ and $\y^{j_l}(0)=\y^{j_l}_0$, respectively and external forcing $\f^{j_k}(\cdot)$ and $\f^{j_l}(\cdot)$, respectively, where $\{\f^{j_k}(\cdot)\}_{k\geq1}$ and $\{\f^{j_l}(\cdot)\}_{l\geq1}$ are two subsequences of $\{\f^j(\cdot)\}_{j\geq1}$ such that $\f^{j_k}\to\f$ and $\f^{j_k}\to\f$ in $\mathrm{L}^2(0,T;\H).$ Note that both $\y^{j_k}$ and $\y^{j_l}$  have the regularity given in \eqref{regularity}. Then from \eqref{density}, we have
\begin{align}\label{density1}
	\y^{j_k}_0\to\y_0 \ \text{ and } \  \y^{j_l}_0\to\y_0 \ \text{ in } \ \H.
\end{align}
Moreover, for a.e. $t\in[0,T]$, we have 
\begin{align*}
	&\frac{\d}{\d t}(\y^{j_k}(t)-\y^{j_l}(t))+ \mu(\A(\y^{j_k}(t)+\y_e)-\A(\y^{j_l}(t)+\y_e))+
	\alpha(\y^{j_k}(t)-\y^{j_l}(t))\nonumber\\& +\beta(\mathcal{C}_1(\y^{j_k}(t)+\y_e)-\mathcal {C}_1(\y^{j_l}(t)+\y_e))+((\partial\I_{\mathcal{K}})_{\frac{1}{j_k}}(\y^{j_k}(t))-(\partial\I_{\mathcal{K}})_{\frac{1}{j_l}}(\y^{j_l}(t))\nonumber\\&= \f^{j_k}(t)-\f^{j_l}(t)-(\mathcal{B}(\y^{j_k}(t)+\y_e) -\mathcal{B}(\y^{j_l}(t)+\y_e))\\&\nonumber\quad-\gamma(\mathcal{C}_2(\y^{j_k}(t)+\y_e)-\mathcal{C}_2(\y^{j_l}(t)+\y_e))-\mathfrak{F}(\y^{j_k}(t)-\y^{j_l}(t)).
\end{align*} 
Taking the inner product with $\y^{j_k}(\cdot)-\y^{j_l}(\cdot)$, using the monotonicity of $(\partial\I_{\mathcal{K}})(\cdot)$ and performing the calculations as we have done in \eqref{3.4}-\eqref{c2.2}  leads to
\begin{align}\label{4p4.13}
	&\frac{1}{2}\frac{\d}{\d t}\|\y^{j_k}-\y^{j_l}\|_{\H}^2 +\frac{\mu}{2}\|\nabla(\y^{j_k}-\y^{j_l})\|_{\H}^2
	+\alpha\|\y^{j_k}-\y^{j_l}\|_{\H}^2+
	\frac{\beta}{4}\||\y^{j_k}+\y_e|^{\frac{r-1}{2}}(\y^{j_k}-\y^{j_l})\|_{\H}^2\nonumber\\&+\frac{\beta}{4}\||\y^{j_l}+\y_e|^{\frac{r-1}{2}}(\y^{j_k}-\y^{j_l})\|_{\H}^2 \nonumber\\ &\leq
	\frac{1}{2}\|\f^{j_k}-\f^{j_l}\|_{\H}^2+\frak{a}\|\y^{j_k}-\y^{j_l}\|_{\H}^2,
\end{align}
where  $\frak{a}>0$ is a constant independent of $j_k$ and $j_l$. Using \eqref{C1}, we find from \eqref{4p4.13} that 
\begin{align}\label{4p4.14}
	&\frac{1}{2}\frac{\d}{\d t}\|\y^{j_k}(t)-\y^{j_l}(t)\|_{\H}^2+ \frac{\mu}{2}\|\nabla(\y^{j_k}(t)-\y^{j_l}(t))\|_{\H}^2+ \frac{\beta}{2^{r}}\|\y^{j_k}(t)-\y^{j_l}(t)\|_{\wi\L^{r+1}}^{r+1}\nonumber\\&\leq
	\frac{1}{2}\|\f^{j_k}(t)-\f^{j_l}(t)\|_{\H}^2+\frak{a}\|\y^{j_k}(t)-\y^{j_l}(t)\|_{\H}^2,	
\end{align}
for a.e. $t\in[0,T]$. Now integrating \eqref{4p4.14} over $(0,t)$,  we obtain 
\begin{align*}
	& \|\y^{j_k}(t)-\y^{j_l}(t)\|_{\H}^2+\mu\int_0^t\|\nabla(\y^{j_k}(s)-\y^{j_l}(s))\|_{\H}^2\d s +\frac{\beta}{2^{r-1}} \int_0^t \|\y^{j_k}(s)-\y^{j_l}(s)\|_{\wi\L^{r+1}}^{r+1} \d s \nonumber\\&\leq\|\y^{j_k}(0)-\y^{j_l}(0)\|_{\H}^2+
	\int_0^t \|\f^{j_k}(s)-\f^{j_l}(s)\|_{\H}^2\d s+2\mathfrak{a}\int_0^t \|\y^{j_k}(s)-\y^{j_l}(s)\|_{\H}^2\d s,
\end{align*}
for all $t\in[0,T]$. Applying the Gronwall inequality, we obtain 
\begin{align}\label{Cauchy}
	& \|\y^{j_k}(t)-\y^{j_l}(t)\|_{\H}^2+\mu\int_0^t\|\nabla(\y^{j_k}(s)-\y^{j_l}(s))\|_{\H}^2\d s +\frac{\beta}{2^{r-1}} \int_0^t \|\y^{j_k}(s)-\y^{j_l}(s)\|_{\wi\L^{r+1}}^{r+1} \d s\nonumber\\&\leq
	C\left(\|\y^{j_k}(0)-\y^{j_l}(0)\|_{\H}^2+\int_0^t \|\f^{j_k}(s)-\f^{j_l}(s)\|_{\H}^2\d s\right),
\end{align}
for all $t\in[0,T]$ and for all $j,k\in\N.$  This shows that the sequence $\{\y^{j_k}(t)\}$ for all $t\in[0,T]$ is a uniformly Cauchy sequence in $\H$ and since $\H$ is complete, we obtain the following uniform convergence:
\begin{align}\label{stc1}
	\y^{j_k}(t)\to\y (t)\   \text{ in } \  \H \  \mbox{ for all }\ t\in[0,T]. 
\end{align}
Since $\y^{j_k}\in\C([0,T];\H)$, the uniform convergence implies $\y\in\C([0,T];\H)$. By similar reasoning as above, we obtain from \eqref{Cauchy} that
\begin{align*}
\y^{j_k}\to\y \  \ \text{in} \  \mathrm{L}^2(0,T;\V) \ \text{ and } \  \y^{j_k}\to\y \  \ \text{in} \  \mathrm{L}^{r+1}(0,T;\wi\L^{r+1}).
\end{align*}
Note that from \eqref{stc1}, by continuity,  we have following results:
\begin{align*}
	\y^{j_k}(0)\to\y(0) \  \text{ in } \ \H \ \text{ and } \ \y^{j_k}(T)\to\y(T) \  \text{ in } \ \H.
\end{align*}

Let us now  prove that $\phi(t)\in\partial\I_{\mathcal{K}}(\y(t))$ for a.e. $t\in[\tau,T]$, for any $\tau>0$. 
From \cite[Proposition 1.4, part(i), pp. 52]{VB2}, we know that $\left(\I+\frac{1}{j}\partial\I_{\mathcal{K}}\right)^{-1}$ is nonexpansive, that is, Lipschitz continuous with the Lipschitz constant $1$ and from \cite[Proposition 1.3, part (iii), pp. 49]{VB2}, we have
\begin{align*}
	\int_0^T \bigg\|\left(\I+\frac{1}{j}\partial\I_{\mathcal{K}}\right)^{-1}(\y^j(t))-\y(t)\bigg\|_{\H}^2\d t&\leq 2\int_0^T\bigg\|\left(\I+\frac{1}{j}\partial\I_{\mathcal{K}}\right)^{-1}(\y^j(t))-\left(\I+\frac{1}{j}\partial\I_{\mathcal{K}}\right)^{-1}(\y(t))\bigg\|_{\H}^2\d t\nonumber\\&\quad+ 2\int_0^T\bigg\|\left(\I+\frac{1}{j}\partial\I_{\mathcal{K}}\right)^{-1}(\y(t))-\y(t)\bigg\|_{\H}^2\d t\nonumber\\&\leq 2\int_0^T\|\y^j(t)-\y(t)\|_{\H}^2\d t+ 2\int_0^T\bigg\|\left(\I+\frac{1}{j}\partial\I_{\mathcal{K}}\right)^{-1}(\y(t))-\y(t)\bigg\|_{\H}^2\d t\nonumber\\&\to 0\ \text{ as }\ j\to\infty,
\end{align*}
so  that $\left(\I+\frac{1}{j}\partial\I_{\mathcal{K}}\right)^{-1}(\y^j)\to \y$ in $\mathrm{L}^2(0,T;\H)$ and $\left(\I+\frac{1}{j}\partial\I_{\mathcal{K}}\right)^{-1}(\y^j(t))\to \y(t)$, for a.e. $t\in[0,T]$ in $\H$ (along a subsequence). Moreover, the maximal monotone operator $\partial\I_{\mathcal{K}}(\cdot)$ is weak-strong and strong-weak closed in $\H\times\H$ (see \cite[Proposition 1.1, part (i), pp. 37]{VB2} and  \cite[Proposition 1.7]{JPSS}), which means that if 
\begin{align*}
	(\partial\I_{\mathcal{K}})_{\frac{1}{j}}(\y^j)&\in\partial\I_{\mathcal{K}}\left(\left(\I+\frac{1}{j}\partial\I_{\mathcal{K}}\right)^{-1}(\y^j)\right),\\
	\left(\I+\frac{1}{j}\partial\I_{\mathcal{K}}\right)^{-1}(\y^j)&\to\y \ \text{ in }\  \mathrm{L}^2(0,T;\H)\ \text{ and }
	(\partial\I_{\mathcal{K}})_{\frac{1}{j}}(\y^j)\xrightharpoonup{w}\ \phi  \ \text{ in } \ \mathrm{L}^{\frac{r+1}{r}}(\tau,T;\H), 
\end{align*}
then $\phi\in\partial\I_\mathcal{K}(\y)$ for a.e. $t\in[\tau,T]$  in $\H$. 

\vskip 2mm
\noindent
\textbf{\emph{Passing limit via Minty-Browder trick}.} We pass weak limit in \eqref{mb1} to obtain
\begin{equation*}
	\left\{
	\begin{aligned}
		\frac{\d \y}{\d t}+\mathscr{F}_0+\phi(t)&= \f(t), \ \text{ in } \ \mathrm{L}^{\frac{r+1}{r}}(\tau,T;\H)  \\
		\y(0)&=\y_0,
	\end{aligned}
	\right.
\end{equation*} 
for some $\phi\in\partial\I_\mathcal{K}(\y)$. Our aim is to now show that $\mathscr{F}_0(t)=\mathscr{F}(\y(t))$ for a.e. $t\in[\tau,T]$. Using  \eqref{wc2}, \eqref{stc1} and the fact that $\partial\I_{\mathcal{K}}$ is weak-strong and strong-weak closed, we calculate  (along a subsequence)
\begin{align*}
	&\left|\int_\tau^t ((\partial\I_{\mathcal{K}})_{\frac{1}{j}}(\y^j(s)),\y^j(s))\d s-\int_\tau^t (\phi(s),\y(s))\d s\right|
	\nonumber\\&\leq\left|\int_\tau^t ((\partial\I_{\mathcal{K}})_{\frac{1}{j}}(\y^j(s))-\phi(s),\y^j(s)-\y(s))\d s \right|+\left|\int_\tau^t ((\partial\I_{\mathcal{K}})_{\frac{1}{j}}(\y^j(s))-\phi(s),\y(s))\d s \right|\nonumber\\&\quad+\left|\int_\tau^t (\phi(s),\y^j(s)-\y(s))\d s\right|\nonumber\\&\to0  \ \ \text{ as } \  j\to\infty,
\end{align*}
for all $t\in[\tau,T]$ and $\phi\in\partial\I_{\mathcal{K}}(\y)$. Since the map $t\mapsto\|\y(t)\|_{\H}^2$ is absolutely continuous (cf. \eqref{wc1} and \eqref{wc2}) on the interval $[\tau,T]$ with $\tau>0$, we have the following energy equality:
		\begin{align}\label{engeq1}
				\|\y(t)\|_{\H}^2+2\int_\tau^t(\mathscr{F}_0(s)-\f(s),\y(s))\d s+
				2\int_\tau^t(\phi(s),\y(s))\d s=\|\y(\tau)\|_{\H}^2,
		\end{align}
for all $t\in[\tau,T]$ with $\tau>0$.	Taking the inner product with  $\y^j(\cdot)$ in the first equation of the system \eqref{mb1}, we derive the following energy equality: 
	\begin{align}\label{mb5}
		\|\y^j(t)\|_{\H}^2+2\int_\tau^t(\mathscr{F}(\y^j(s))-\f^j(s),\y^j(s))\d s+
		2\int_\tau^t((\partial\I_{\mathcal{K}})_{\frac{1}{j}}(\y^j(s)),\y^j(s))\d s=\|\y^j(\tau)\|_{\H}^2,
	\end{align}
	for all $t\in[\tau,T]$ with $\tau>0$. 
		Taking the limit supremum in \eqref{mb5}, we obtain
		\begin{align}\label{mb7}
			&\limsup\limits_{j\to\infty}\int_\tau^t(\mathscr{F}(\y^j(s)),\y^j(s))\d s\nonumber\\&= \limsup\limits_{j\to\infty}\left[\frac{1}{2}\|\y^j(\tau)\|_{\H}^2-\frac{1}{2}\|\y^j(t)\|_{\H}^2+\int_\tau^t(\f^j(s),\y^j(s))\d s-\int_\tau^t((\partial\I_{\mathcal{K}})_{\frac{1}{j}}(\y^j(s)),\y^j(s))\d s\right] \nonumber\\&=
			\frac{1}{2}\|\y(\tau)\|_{\H}^2-\frac{1}{2}\|\y(t)\|_{\H}^2+\limsup\limits_{j\to\infty}\int_\tau^t(\f^j(s),\y^j(s))\d s -\liminf\limits_{j\to\infty} \int_\tau^t((\partial\I_{\mathcal{K}})_{\frac{1}{j}}(\y^j(s)),\y^j(s))\d s\nonumber\\&=\frac{1}{2}\|\y(\tau)\|_{\H}^2-\frac{1}{2}\|\y(t)\|_{\H}^2+\int_\tau^t(\f(s),\y(s))\d s - \int_\tau^t(\phi(s),\y(s))\d s=\int_\tau^t(\mathscr{F}_0(s),\y(s))\d s,
		\end{align}
		for all $t\in[\tau,T]$, where we have used \eqref{engeq1} in the final equality. One can easily see from Proposition \ref{prop3.1} that the operator  $\mathscr{F}(\cdot)$ is quasi-$m$-accretive. Therefore, we write
		\begin{align}\label{531}
			\int_\tau^T (\mathscr{F}(\v(t))-\mathscr{F}(\y^j(t)),\v(t)-\y^j(t))\d t+\mathfrak{K}\int_\tau^T \|\v(t)-\y^j(t)\|_{\H}^2\d t\geq0,
		\end{align}
	where $\mathfrak{K}$ is same as in \eqref{4pr}. On taking the limit supremum in \eqref{531}  and using \eqref{mb7}, we obtain 
		\begin{align}\label{532}
			\int_\tau^T (\mathscr{F}(\v(t))-\mathscr{F}_0(t),\v(t)-\y(t))\d t+\mathfrak{K}\int_\tau^T \|\v(t)-\y(t)\|_{\H}^2\d t\geq0,
		\end{align}
 for all $\v\in\mathrm{L}^2(0,T;\H)$. Substituting $\v=\y+\lambda\w$ (for $\w\in\mathrm{L}^2(0,T;\H)$) in \eqref{532} and dividing by $\lambda$ result to
\begin{align*}
	\int_\tau^T (\mathscr{F}(\y(t)+\lambda\w(t))-\mathscr{F}_0(t),\w(t))\d t\geq0.
\end{align*}
 Since $\mathscr{F}(\cdot)$ is hemicontinuous, we conclude that $\mathscr{F}_0(t)=\mathscr{F}(\y(t))$, for a.e. $t\in[\tau,T]$ in $\H$. 

\vskip 2mm
\noindent
\textbf{\emph{Initial data}.} 
We want to show that \begin{align}\label{ID}
	\lim\limits_{\tau\to0^+}\|\y(\tau)-\y_0\|_{\H}=0.
\end{align} Let $\epsilon>0$ be given.  Note that from \eqref{stc1}, one can conclude by the continuity that $\y^{j_k}(0)\to\y(0)$ in $\H$. This means, there is a sufficiently large $k_1\in\N$ such that 
\begin{align*}
	\|\y^{j_k}(0)-\y(0)\|_{\H}<\frac{\epsilon}{3} \ \text{ for all } \  k\geq k_1.
\end{align*}
Similarly, from \eqref{density1}, there exists sufficiently large $k_2\in\N$ such that 
\begin{align*}
    \|\y^{j_k}(0)-\y_0\|_{\H}<\frac{\epsilon}{3} \  \text{ for all } \ k\geq k_2.
\end{align*}
Using the triangle inequality, one can calculate
\begin{align*}
		\|\y(\tau)-\y_0\|_{\H}\leq\|\y(\tau)-\y(0)\|_{\H}+\|\y^{j_k}(0)-\y(0)\|_{\H}+\|\y^{j_k}(0)-\y_0\|_{\H}.
\end{align*}
Since $\y\in\C([0,T];\H)$ (cf. \eqref{stc1}), it is clear that the first term in the right hand side of above inequality approaches to $0$ as $\tau\to0$ from right. Therefore, for any $\epsilon>0$, there exists some $\delta>0$ such that for all $k\geq k_0:=\max\{k_1,k_2\}$, we have following implication from above inequality  
\begin{align*}
	\|\y(\tau)-\y_0\|_{\H}<\epsilon \ \text{ if } \  0<\tau<\delta,
\end{align*}
which completes the proof of \eqref{ID}. 
\vskip 2mm
\noindent
\textbf{\emph{Uniqueness}.} Assume that $\y_1(\cdot)$ and $\y_2(\cdot)$ be any two solutions of the system \eqref{4p1} satisfying the initial data $\y_1(0)=\y_2(0)=\y_0$, respectively.
 Then one can obtain
\begin{align*}
	&\frac{\d}{\d t}(\y_1(t)-\y_2(t))+\mu(\A\y_1(t)-\A\y_2(t))+\alpha(\y_1(t)-\y_2(t)) \nonumber\\& +
	\beta(\mathcal{C}_1(\y_1(t)+\y_e)-\mathcal {C}_1(\y_2(t)+\y_e)) +\xi(\y_1(t))-\xi(\y_2(t))\nonumber\\&=(\mathcal{B}(\y_1(t)+\y_e) -\mathcal{B}(\y_2(t)+\y_e))-\gamma(\mathcal{C}_2(\y_1(t)+\y_e)-\mathcal{C}_2(\y_2(t)+\y_e))-\mathfrak{F}(\y_1(t)-\y_2(t)), 
\end{align*}
for a.e. $t\in[\tau,T]$. Taking the inner product with $\y_1(\cdot)-\y_2(\cdot)$ and doing similar calculations as in \eqref{4p4.13}-\eqref{4p4.14} imply that 
\begin{align*}
&\frac12 \frac{\d}{\d t}\|\y_1(t)-\y_2(t)\|^2_{\H}+\mu \|\nabla(\y_1(t)-\y_2(t))\|_{\H}^2+\alpha\|\y_1(t)-\y_2(t)\|_{\H}^2 +\frac{\beta}{2^{r}}\|\y_1(t)-\y_2(t)\|_{\wi\L^{r+1}}^{r+1}\nonumber\\& \leq\mathfrak{a}\|\y_1(t)-\y_2(t)\|_{\H}^2,
\end{align*}
for a.e. $t\in[\tau,T].$
Thus by the Gronwall inequality, we obtain
\begin{align*}
\|\y_1(t)-\y_2(t)\|_{\H}\leq\|\y_1(\tau)-\y_2(\tau)\|_{\H}e^{\mathfrak{a}(t-\tau)}\leq
(\|\y_1(\tau)-\y_0\|_{\H}+\|\y_2(\tau)-\y_0\|_{\H})e^{\mathfrak{a}(t-\tau)}.
\end{align*}
In view of \eqref{ID}, we complete the proof of uniqueness.

\section{Stability}\label{stb}\setcounter{equation}{0}
In this section, we discuss stabilization of CBFeD system by means of finite- and infinite-dimensional feedback controllers. 
\subsection{\textbf{\emph{Infinite-dimensional feedback controllers}}}\label{infinite}
Let us consider the following controlled CBFeD system for a.e. $t>0$:
\begin{equation}\label{appl1}
	\left\{
	\begin{aligned}
		\frac{\d \y(t)}{\d t}+\mu\A\y(t)+ \mathcal{B}(\y(t))+\alpha\y(t) +\beta\mathcal{C}_1(\y(t))+\gamma\mathcal{C}_2(\y(t))&= \f_e+\mathbf{U}(t),\\ 
		\y(0)&=\y_0,
	\end{aligned}
	\right.
\end{equation}	
where $\mathbf{U}(\cdot)$ is a distributed control acting on the whole torus. Let $\y_e\in\D(\A)$ be the steady state (equilibrium) solution of the stationary equation
\begin{equation}\label{appl2}
	\mu\A\y_e+\mathcal{B}(\y_e)+\alpha\y(t)+\beta\mathcal{C}_1(\y_e)+\gamma\mathcal{C}_2(\y_e)= \f_e \  \text{ in } \ \mathbb{T}^d,
\end{equation}	
whose solvability is discussed in Appendix \ref{SP} (see Theorem \ref{STT} below). Let $\mathcal{K}\subset\H$ satisfy  Hypothesis \ref{AssupK}. Our aim is to find  those feedback controllers which stabilize the equlibrium solution $\y_e$ of \eqref{appl1} exponentially, under some invariance condition. So, we would like to find a feedback controller, say $\mathbf{U}(\cdot)$, such that following conditions meet:
\begin{align*}
\lim\limits_{t\to\infty}e^{\delta t}\|\y(t)-\y_e\|_{\H}=0, \ \text{for some}\ \delta>0, 
\end{align*}
and
\begin{align*}
	\y(t)-\y_e\in\mathcal{K} \  \text{for all} \  t\in[0,T].
\end{align*}

Let us set $\z(\cdot)=\y(\cdot)-\y_e$ and we take $\mathbf{U}(t)\in-\theta\z(t)-(\partial\I_{\mathcal{K}})(\z(t))$, for some $\theta>0$. Then from \eqref{appl1} and \eqref{appl2}, we  consider the following inclusion problem for \eqref{appl1} for a.e. $t>0$:
\begin{equation}\label{appl3}
	\left\{
	\begin{aligned}
		\frac{\d \z(t)}{\d t}+\mu\A\z(t)+\wi{\mathcal{B}}(\z(t))+\alpha\z(t)+\beta\wi{\mathcal{C}}_1(\z(t))+ \gamma\wi{\mathcal{C}}_2(\z(t))+\theta\z(t)+(\partial\I_{\mathcal{K}})(\z(t))&\ni\boldsymbol{0}, \\ 
		\z(0)&=\y_0-\y_e.
	\end{aligned}
	\right.
\end{equation}	

Let us now prove the following global stability result:
\begin{theorem}\label{thm3}
Let $\mathcal{K}$ satisfy Hypothesis \ref{AssupK}. Let us suppose $\f_e\in\H$ is given and $\y_e$ satisfies \eqref{appl2}. Moreover, we also assume that $\z(0)\in\mathcal{K}$. Then for sufficiently large $\theta>0$ (to be specified later), the problem \eqref{appl3} has a unique weak solution for $r\geq3$ satisfying 
\begin{align}\label{reg}
	\z\in\mathrm{C}([0,T];\H)\cap\mathrm{L}^2(0,T;\V)\cap\mathrm{L}^{r+1}(0,T;\wi\L^{r+1}) \ \text{ with } \ \frac{\d\z}{\d t}\in\mathrm{L}^{\frac{r+1}{r}}(\tau,T;\H),
\end{align}
 for any  $0<\tau<T$. Furthermore, there exists an $\delta=\delta(\theta)>0$ such that the following global exponential stability condition holds:
\begin{align*}
	\lim\limits_{t\to\infty}e^{\delta t}\|\z(t)\|_{\H}=0.
\end{align*}
Also, $\z(t)\in\mathcal{K}$ for all $t\in[0,T].$
\end{theorem}

\begin{proof}
For the time being, we assume that $\theta>0$ be a fixed constant. One can apply main Theorem \ref{mainT} for $\f=\boldsymbol{0}$ and $\mathfrak{F}=\theta\I$ to ensure the existence of a unique weak solution of the system \eqref{appl3} having the regularity \eqref{reg}. 

Let us now prove exponential stability. For this purpose, we take the inner product with $\z(\cdot)$ in \eqref{appl3} and use previous calculations (see Proposition \ref{prop3.1}) to get for a.e. $t\geq 0$,
\begin{align*}
&\frac{1}{2}\frac{\d}{\d t}\|\z(t)\|_{\H}^2+\frac{\mu}{2}\|\nabla\z(t)\|_{\H}^2+\beta\left(\frac{1}{4}-\frac{\wi\varepsilon}{4}\right) \||\z+\y_e|^{\frac{r-1}{2}}\z\|_{\H}^2+\beta\left(\frac{1}{4}-\frac{\varepsilon}{2}\right)\||\y_e|^{\frac{r-1}{2}}\z\|_{\H}^2 \leq -(\theta+\alpha-\mathfrak{C})\|\z(t)\|_{\H}^2,
\end{align*}
where $\mathfrak{C}:=\varrho_{\eps}+\rho_{\wi{\eps}}+\rho_{\eps}$ and we have used the monotonicity of $\partial\I_{\mathcal{K}}(\cdot)$ with $\boldsymbol{0}\in\mathcal{K}.$ This implies  
\begin{align*}
	\frac{\d}{\d t}\|\z(t)\|_{\H}^2\leq-2(\theta+\alpha-\mathfrak{C})\|\z(t)\|_{\H}^2,
\end{align*}
for a.e. $t\in[0,T]$. Let us choose $\theta>0$ be sufficiently large such that $\delta_1=\theta+\alpha-\mathfrak{C}>0.$ Using the Gronwall inequality, we have for all $t\geq 0$
\begin{align*}
\|\z(t)\|_{\H}^2\leq e^{-2\delta_1 t}\|\z(0)\|_{\H}^2.
\end{align*}
We choose $\delta\in(0,\delta_1)$ such that 
\begin{align*}
	\lim\limits_{t\to\infty}e^{\delta t}\|\z(t)\|_{\H}=0,
\end{align*}
which proves that the solutions of \eqref{appl3} are exponentially stabilizable. In view of Hypothesis \ref{AssupK}, Remark \ref{inv} and \cite[Corollary 2.54, pp. 56]{VBs}, we conclude that $\z(t)\in\mathcal{K}$ for all $t\in[0,T].$
\end{proof}
                   
\subsection{\textbf{\emph{Finite-dimensional compactly supported stabilizing feedback controllers}}}\label{finite}
 Let us now discuss the stability problem via finite-dimensional controllers which are compactly supported.  Here the controllers are situated in smaller parts of the domain and taking values  in some finite-dimensional subspace spanned by the eigenfunctions of the Stokes operator. Since the second order derivatives of the nonlinear operators are involved, we restrict ourselves to $r=3$ with $\gamma=0,$ and $r>3$ and $q\in[3,r)$ with $\gamma<0$ in this subsection.

Let $\{\w_k\}_{k\in\N}$ be the eigenfunctions of the Stokes operator $\A$. We consider 
\begin{align}\label{cvx}
	\mathcal{K}:=\mathrm{linspan}\{\w_k:k=1,\ldots, n\}\subset\H.
\end{align}
It is clear that $\boldsymbol{0}\in\mathcal{K}$ and $\mathcal{K}$ is a closed and convex set. Also, since $(\I+\lambda\A)$ is invertible and $\mathcal{K}$ is finite-dimensional, therefore, it satisfies the condition \eqref{appl1.1} of Hypothesis \ref{AssupK}. Moreover, we re-arrange the sequence $\{\w_k\}_{k\in\N\cup\{0\}}$ of eigenfunctions of the Stokes operator $\A$ in such a way that the first $k$ eigenfunctions are exactly in the same position as in the generators of $\mathcal{K}.$ Then, for any $\y\in\H$, we can write $\y=\sum\limits_{k=0}^{\infty}y_k\w_k,$ where $y_k=(\y,\w_k)\in\R$. We denote $\mathrm{P}_{\mathcal{K}},$ the orthogonal projection (finite-dimensional) in $\H$ onto $\mathcal{K}.$ Then one may write $\mathrm{P}_{\mathcal{K}}\y=\sum\limits_{k=1}^{n}y_k\w_k\in\mathcal{K}.$ Note that $\text{int}(\mathcal{K})=\varnothing$, as $\mathcal{K}$ is a proper subspace of $\H$.

 We consider a non-empty open subset $\omega\subset\subset\mathbb{T}^d$. Our aim is now to construct a finite-dimensional feedback controller $\mathbf{U}(\cdot)$ which is supported in the cylinder $\omega\times(0,\infty)$ such that the equilibrium solution $\y_e$ of the system  \eqref{appl1} are locally exponential stabilizable. Given $\omega\subset\subset\mathbb{T}^d$, the application $m:\L^2(\omega)\to\L^2(\mathbb{T}^d)$ such that
 	\begin{align*}
 		(m\y)(x):=\begin{cases}
 			\y(x), \ &x\in\omega,\\
 			0, \ &x\in\mathbb{T}^d\setminus\omega,
 		\end{cases}
 	\end{align*} 
 extends the functions in the subdomain $\L^2(\omega)$ to the whole space $\L^2(\mathbb{T}^d)$. Let $\mathrm{R}:\mathcal{K}\to\L^2(\omega)$ be a linear and continuous operator. 

Similar to \eqref{appl3}, we choose $\mathbf{U}(\cdot)\in\mathcal{P}m\mathrm{R}(\mathrm{P}_{\mathcal{K}}\z(\cdot))-\partial\I_{\mathcal{K}}(\z(\cdot))$ and we consider the following problem for a.e. $t>0$:
\begin{equation}\label{appl4}
\hspace{-4mm}	\left\{
	\begin{aligned}
		\frac{\d \z(t)}{\d t}+\mu\A\z(t)+\wi{\mathcal{B}}(\z(t))+\alpha\z(t)+ \beta\wi{\mathcal{C}}_1(\z(t))+\gamma\wi{\mathcal{C}}_2(\z(t)) +\partial\I_{\mathcal{K}}(\z(t))-\mathcal{P}m\mathrm{R}(\mathrm{P}_{\mathcal{K}}\z(t))&\ni\boldsymbol{0},\\ 
		\z(0)&=\y_0-\y_e,
	\end{aligned}
	\right.
\end{equation}	
where $\y_e\in\D(\A)$ be the solution of \eqref{appl2}. We now state the following local stability result:
\begin{theorem}\label{findim}
For $\f_e\in\H$ and $\mathcal{K}$ be the set given by \eqref{cvx}. Also, assume that $\z(0)\in\mathcal{K}$ with $\|\z(0)\|_{\H}<\rho,$ for some $\rho>0$. Then for every $\delta>0$, there exists a linear continuous operator $\mathrm{R}:\mathcal{K}\to\L^2(\omega)$ so that the problem \eqref{appl4} has a unique weak solution $\z(\cdot)$ for $r>3$ and $d=r=3$ with $2\beta\mu>1$, satisfying 
	\begin{align}\label{reg1}
		\z\in\mathrm{C}([0,T];\H)\cap\mathrm{L}^2(0,T;\V)\cap\mathrm{L}^{r+1}(0,T;\wi\L^{r+1})\ \ \text{ with }\ \ \frac{\d\z}{\d t}\in\mathrm{L}^{\frac{r+1}{r}}(\tau,T;\V'),
	\end{align}
 for any $0<\tau<T$. Furthermore, the following local exponential stability condition holds:
	\begin{align*}
		\lim\limits_{t\to\infty}e^{\delta t}\|\z(t)\|_{\H}=0.
	\end{align*}
	Also, $\z(t)\in\mathcal{K},$ for all $t\in[0,T].$
\end{theorem}

\begin{proof}
	We prove the theorem in the following steps:
	\vskip 2mm
	\noindent \textbf{Step 1.} \emph{The finite-dimensional system:}
	For any linear continuous operator $\mathrm{R}$, the operator $\mathcal{P}m\mathrm{R}\mathrm{P}_{\mathcal{K}}$ is  linear continuous  as it is the composition of linear continuous operators. Thus from Theorem \ref{mainT}, with $\mathfrak{F}=-\mathcal{P}m\mathrm{R}\mathrm{P}_{\mathcal{K}}$, one can conclude that the system \eqref{appl4} has a unique weak solution with required regularity given in \eqref{reg1}. Moreover, we have $\z(t)\in\mathcal{K}$ for all $t\in[0,T]$. Now we only need to show the existence of such a \emph{stabilizing operator} $\mathrm{R}$. For this, let us first consider the following system for a.e. $t>0$:
\begin{equation}\label{appl4.0}
	\left\{
	\begin{aligned}
		\frac{\d \z(t)}{\d t}+\mu\A\z(t)+\wi{\mathcal{B}}(\z(t))+\alpha\z(t)+ \beta\wi{\mathcal{C}}_1(\z(t))+\gamma\wi{\mathcal{C}}_2(\z(t))+\partial\I_{\mathcal{K}}(\z(t))-\mathcal{P}m\mathbf{\overline{U}}(t)&\ni\boldsymbol{0},\\ 
		\z(0)&=\y_0-\y_e,
	\end{aligned}
	\right. 
\end{equation}	
where we assume $\mathbf{\overline{U}}(\cdot)$ is a distributed control in $\mathrm{L}^2(0,T;\L^2(\omega))$. Let us take the inner product with $\w_k, \  k=0,\ldots,n$ in the first equation of the system \eqref{appl4.0} to get
\begin{align}\label{appl4.1}
&\left(\frac{\d \z(t)}{\d t},\w_k\right)+\mu(\z(t),\A\w_k)+ \alpha(\z(t),\w_k)\nonumber\\&=-(\wi{\mathcal{B}}(\z(t)),\w_k)-\beta(\wi{\mathcal{C}}_1(\z(t)),\w_k)- \gamma(\wi{\mathcal{C}}_2(\z(t)),\w_k)+(m\mathbf{\overline{U}}(t),\w_k),
\end{align}
where we have used the fact that $\partial \I_{\mathcal{K}}=\mathcal{K}^{\perp}$ in the elements of $\mathcal{K}=\partial\mathcal{K}$. We have 
\begin{align}\label{appl4.2}
	(\wi{\mathcal{B}}(\z),\w_k)=(\mathcal{B}(\z+\y_e),\w_k) -(\mathcal{B}(\y_e),\w_k)=b(\z,\z,\w_k)+ b(\z,\y_e,\w_k) + b(\y_e,\z,\w_k),
\end{align}
and, by Taylor's formula (\cite[Theorem 7.9.1]{PGC})
\begin{align*}
	\wi{\mathcal{C}}_1(\z)=\mathcal{C}_1(\z+\y_e)-\mathcal{C}_1(\y_e)
	=\mathcal{C}_1'(\y_e)\z+\int_0^1 \mathcal{C}_1''(\y_e+\theta\z)(\z\otimes\z)\d\theta.
\end{align*}
It implies that 
\begin{align}\label{appl4.3}
	(\wi{\mathcal{C}}_1(\z),\w_k)&=(\mathcal{C}_1'(\y_e)\z,\w_k)+\left(\int_0^1 \mathcal{C}_1''(\y_e+\theta\z)(\z\otimes\z)\d\theta,\w_k\right)\nonumber\\&=
	\sum\limits_{i=1}^{k}z_i(\mathcal{C}_1'(\y_e)\w_i,\w_k)+\sum\limits_{i,j=1}^{k}z_iz_j
	\left(\int_0^1 \mathcal{C}_1''(\y_e+\theta\z)(\w_i\otimes \w_j)\d\theta,\w_k\right),
\end{align}
where $z_k(t):=(\z(t),\w_k)$ for $k=0,\ldots,n$. Then the equation \eqref{appl4.1} together with \eqref{appl4.2}-\eqref{appl4.3} yields
\begin{align}\label{fd}
	\left\{
	\begin{aligned}
	z_k'(t)+\mu\lambda_kz_k(t)+\alpha z_k(t)+\sum\limits_{i=1}^{n}h_{ik}^1z_i(t)& +\sum\limits_{i=1}^{n}h_{ik}^2z_i(t)+\sum\limits_{i,j=1}^{n}g_{ijk}^1z_i(t)z_j(t)+\sum\limits_{i,j=1}^{n}g_{ijk}^2(\z)z_i(t)z_j(t)\\&=(m\mathbf{\overline{U}}(t),\w_k),\\
		z_k(0)&=(\y_0-\y_e,\w_k),
	\end{aligned}
	\right.
\end{align}
where
\begin{align*}
	g_{ijk}^1&=b(\w_i,\w_j,\w_k),\nonumber\\
	g_{ijk}^2(\z)&=\beta\left(\int_0^1 \mathcal{C}_1''(\y_e+\theta\z)(\w_i\otimes\w_j)\d\theta,\w_k\right)+\gamma\left(\int_0^1 \mathcal{C}_2''(\y_e+\theta\z)(\w_i\otimes\w_j)\d\theta,\w_k\right),\nonumber\\
	h_{ik}^1&=b(\w_i,\y_e,\w_k)+b(\y_e,\w_i,\w_k) \text{ and }
	h_{ik}^2=\beta(\mathcal{C}_1'(\y_e)\w_i,\w_k)+\gamma(\mathcal{C}_2'(\y_e)\w_i,\w_k),
\end{align*}
for $i,j,k=0,\ldots,n$. Note that \eqref{fd} represents a system of  $n$ nonlinear ordinary differential equations with unknowns $z_k(\cdot)$, for $k=0,\ldots,n$. We define the linear operator $\mathcal{L}:\R^n\to\R^n$, the quadratic form $\mathcal{Q}(\cdot):\R^n\to\R^n$ and the nonlinear operator $\mathcal{N}(\cdot):\R^n\to\R^n$, which are given by 
\begin{align}\label{vectors}
	&\mathcal{L}\v:=\left(\mu\lambda_kz_k+\alpha z_k+
	\sum\limits_{i=1}^{n}(h_{ik}^1+h_{ik}^2)z_i\right)_{k=1}^{n}, \ 
	\mathcal{Q}(\v):=\left(\sum\limits_{i,j=1}^{n}g_{ijk}^1z_iz_j\right)_{k=1}^n\nonumber\\&
	\text{ and }  \mathcal{N}(\v):=\left(\sum\limits_{i,j=1}^{n}g_{ijk}^2(\z)z_iz_j\right)_{k=1}^n. 
\end{align}
where $\v(\cdot)=\left(z_k(\cdot)\right)_{k=1}^n$. Also, we define $\mathcal{H}:\wi\L^2(\omega)\to\R^n$ by $\mathcal{H}\overline{\mathbf{U}}=\left((m\mathbf{\overline{U}},\w_k)\right)_{k=1}^n$.
We finally rewrite the system \eqref{fd} in the following matrix form:
	\begin{equation}\label{appl5}
	\left\{
	\begin{aligned}
		\v'(t)+\mathcal{L}\v(t)+\mathcal{Q}(\v(t))+\mathcal{N}(\v(t))&=\mathcal{H}\overline{\mathbf{U}}(t), \ \text{ for a.e. } \ t>0,\\
	\v(0)&=\v_0,
	\end{aligned}
	\right.
\end{equation}
where $\v_0=\left(z_k(0)\right)_{k=1}^n.$

\vskip 2mm
\noindent \textbf{Step 2.} \emph{The linear system and stabilization:} 
We linearize the above system \eqref{appl5} about $\boldsymbol{0}$ as
	\begin{equation}\label{appl6}
	\left\{
	\begin{aligned}
		\v'(t)+\mathcal{L}\v(t)&=\mathcal{H}\overline{\mathbf{U}}(t), \ \text{ for a.e. } \ t>0,\\
	\v(0)&=\v_0.
	\end{aligned}
	\right.
\end{equation}
Moreover, we write the corresponding adjoint system as 
\begin{align}\label{appl7}
	\q'(t)-\mathcal{L^*}\q(t)=0, \ \text{ for a.e. } \ t>0,
\end{align}
where $\mathcal{L^*}$ is adjoint to $\mathcal{L}.$ We will now prove the unique continuation property of the adjoint system \eqref{appl7}. In order to do this, we need to show that $\mathcal{H}^*\q(t)=0$ for all $t>0$ implies   $\q(t)=0$, for all $t>0$,  where $\mathcal{H}^*(\cdot)$ is adjoint to $\mathcal{H}(\cdot)$ (see \cite[pp. 762]{RAM}). Then, for any $\v\in\R^n$, we calculate 
	\begin{align*}
		(\mathbf{\overline{U}},\mathcal{H}^*\v)_{\L^2(\omega)}&=(\mathcal{H}\mathbf{\overline{U}},\v)
		_{\R^n}=\left(\mathcal{H}\mathbf{\overline{U}}\right)^{\top}\v=
		(m\mathbf{\overline{U}},\w_1)z_1+\cdots+(m\mathbf{\overline{U}},\w_n)z_n\nonumber\\&=
		(m\mathbf{\overline{U}},\w_1z_1+\cdots+\w_nz_n)=
		\bigg(m\mathbf{\overline{U}},\sum\limits_{k=1}^n\w_kz_k\bigg)_{\L^2(\mathbb{T}^d)}=\bigg(\mathbf{\overline{U}},\sum\limits_{k=1}^nz_k\w_k\bigg)_{\L^2(\omega)},
	\end{align*}
for all $\mathbf{\overline{U}}\in\L^2(\omega)$. 	This gives that $\mathcal{H}^*\v=\sum\limits_{k=1}^nz_k\w_k|_{\omega}.$ So, if $\q=(q_1,\ldots,q_n),$ then 
	\begin{align*}
		\mathcal{H}^*\q(t)=0 \  \text{ if and only if } \  \sum\limits_{k=1}^nq_k(t)\w_k|_{\omega}=0 \ \text{in} \ \L^2(\omega).
	\end{align*}
	Since $\{\w_k\}_{k\in\N}$ is an orthonormal basis in $\L^2(\mathbb{T}^d)$, we know that  $\{\w_k\}_{k\in\N}\in\L^2(\omega)$ is linearly independent, which shows that $q_k(t)=0$ for all $t>0$ and $k=1,2,\ldots,n$ (see \cite[pp. 1458]{VBT}). This proves the unique continuation property for the adjoint system \eqref{appl7} and hence the system \eqref{appl6} is exactly controllable (see \cite[Chapter 2]{fb}). Thus by the well known \emph{Wonham theorem} \cite[Chapter 2, pp. 35]{jZ}, one can conclude that \eqref{appl6} is \emph{completely stabilzable}. Therefore by definition, it asserts that there exists a linear and continuous feedback controller $\mathrm{R}:\R^n\to\L^2(\omega)$ and a constant $\hat{M}>0$ such that for every $\sigma>0,$ the solution of \eqref{appl6} satisfies $|\v(t)|_{\mathbb{R}^n}\leq\hat{M}e^{-\sigma t}|\v_0|_{\mathbb{R}^n},$ for $t>0$, which can also be written as 
\begin{align}\label{stbl}
	|e^{-(\mathcal{L}-\mathcal{H}\mathrm{R})t}|_{\mathscr{L}(\mathbb{R}^n)}\leq\hat{M}e^{-\sigma t}, \ \text{for} \ t>0,
\end{align}
where $\mathscr{L}(\mathbb{R}^n)$ is the space of all bounded linear operators on $\mathbb{R}^n$. Moreover from \eqref{appl6}, we infer that the controller $\mathbf{\overline{U}}(t)=\mathrm{R}\v(t)$ stabilzes the following system asymptotically:
	\begin{equation*}
	\left\{
	\begin{aligned}
		\v'(t)+(\mathcal{L}-\mathcal{H}\mathrm{R})\v(t)&=\boldsymbol{0}, \ \text{ for a.e. } \ t>0,\\
	\v(0)&=\v_0.
	\end{aligned}
	\right.
\end{equation*}
\vskip 2mm
\noindent \textbf{Step 3.} \emph{Stabilization of the nonlinear system:} 
 Let us now consider the following nonlinear system corresponding to \eqref{appl5}:
	\begin{equation}\label{appl9}
	\left\{
	\begin{aligned}
		\v'(t)+(\mathcal{L}-\mathcal{H}\mathrm{R})\v(t)+\mathcal{Q}(\v(t))+\mathcal{N}(\v(t))&=\boldsymbol{0}, \ \text{ for a.e. } \ t>0,\\
	\v(0)&=\v_0.
	\end{aligned}
	\right.
\end{equation}
Our aim is to now show that the controller $\mathbf{\overline{U}}(t)=\mathrm{R}\v(t)$ stabilizes the nonlinear system \eqref{appl9} exponentially. The solution of \eqref{appl9} is given by the variation of constant formula 
\begin{align}\label{stbl1}
	\v(t)=e^{-(\mathcal{L}-\mathcal{H}\mathrm{R})t}\v_0-\int_0^t e^{-(\mathcal{L}-\mathcal{H}\mathrm{R})(t-s)}\mathcal{Q}(\v(s))\d s-\int_0^t e^{-(\mathcal{L}-\mathcal{H}\mathrm{R})(t-s)}\mathcal{N}(\v(s))\d s.
\end{align}
Using \eqref{stbl} in \eqref{stbl1} to obtain 
\begin{align}\label{stbl2}
	|\v(t)|_{\mathbb{R}^n}\leq \hat{M}e^{-\sigma t}|\v_0|_{\mathbb{R}^n}+\hat{M}\int_0^t e^{-\sigma (t-s)}|\mathcal{Q}(\v)|_{\mathbb{R}^n}\d s+\hat{M}\int_0^t e^{-\sigma (t-s)}|\mathcal{N}(\v)|_{\mathbb{R}^n}\d s.
\end{align}
 From \eqref{vectors}, we have 
\begin{align}\label{appl5.0}
	|\mathcal{Q}(\v)|^2_{\mathbb{R}^n}=\sum\limits_{k=1}^n\left(\sum\limits_{i,j=1}^{n}g_{ijk}^1z_iz_j\right)^2.
\end{align}
By using \eqref{ef}, we calculate
\begin{align*}
|g_{ijk}^1|=|b(\w_i,\w_j,\w_k)|\leq\int_{\mathbb{T}^d} |\w_i(x)||\nabla\w_j(x)||\w_k(x)|\d x\leq
\left(\frac{2}{\mathrm{L}^d}\right)^{\frac{3}{2}}\frac{2\pi}{\mathrm{L}}|\mathbb{T}^d|,
\end{align*}
which gives $|g_{ijk}^1|^2\leq\frac{32\pi^2}{\mathrm{L}^{d+2}}.$ On employing discrete H\"older's inequality in \eqref{appl5.0}, we find
\begin{align*}
|\mathcal{Q}(\v)|^2_{\mathbb{R}^n}\leq\frac{32\pi^2}{\mathrm{L}^{d+2}}\sum\limits_{k=1}^n \left(\sum\limits_{i,j=1}^{n}z_iz_j\right)^2\leq\frac{32\pi^2}{\mathrm{L}^{d+2}}\sum\limits_{k=1}^n\left(\sum\limits_{i=1}^{n}z_i^2\right)^2\leq\frac{32\pi^2}{\mathrm{L}^{d+2}}|\v|^4,
\end{align*}
which implies that 
\begin{align}\label{appl5.0.1}
	|\mathcal{Q}(\v)|_{\mathbb{R}^n}\leq\gamma_0|\v|^2_{\mathbb{R}^n},
\end{align} 
where the constant $\gamma_0:=\frac{4\pi}{\mathrm{L}}\left(\frac{2}{\mathrm{L}^d}\right)^{\frac{1}{2}}.$ 
Now we calculate $|\mathcal{N}(\v)|^2_{\mathbb{R}^n}$ separately for  $r=3$ with $\gamma=0,$ and $r>3$ and $q\in[3,r)$.

\vskip 2mm
\noindent \textbf{Step 4.} \emph{The case $r=3$ with $\gamma=0.$} Using discrete H\"older's inequality, we obtain 
 \begin{align}\label{appl5.10}
|\mathcal{N}(\v)|^2_{\mathbb{R}^n}=\sum\limits_{k=1}^n\left(\sum\limits_{i,j=1}^{n}g_{ijk}^2(\z)z_iz_j\right)^2
&\leq \beta^2\sum\limits_{i,j,k=1}^n\left(\int_0^1\mathcal{C}_1''(\y_e+\theta\z)(\w_i\otimes\w_j)\d\theta,\w_k\right)^2\left(\sum\limits_{i=1}^nz_i^2\right)^2\nonumber\\&\leq\beta^2|\v|^4
\sum\limits_{i,j,k=1}^n\left(\int_0^1\mathcal{C}_1''(\y_e+\theta\z)(\w_i\otimes\w_j)\d\theta,\w_k\right)^2.
\end{align}
We now calculate
\begin{align}\label{appl5.11}
	\left|\left(\int_0^1 \mathcal{C}_1''(\y_e+\theta\z)(\w_i\otimes\w_j)\d\theta,\w_k\right)\right|=
	\left|\int_{\mathbb{T}^d}\left(\int_0^1 \mathcal{C}_1''(\y_e+\theta\z)(\w_i\otimes\w_j)\d\theta\right)\w_k(x)\d x\right|.
\end{align}
Using \eqref{ef} and \eqref{C} (for $r=3$), we obtain 
\begin{align*}
	\left|\int_0^1 \mathcal{C}_1''(\y_e+\theta\z)(\w_i\otimes\w_j)\d\theta\right|&= 2
	\left|\int_0^1\left(((\y_e+\theta\z)\cdot\w_i)\w_j +((\y_e+\theta\z)\cdot\w_j) \w_i+(\y_e+\theta\z)(\w_i\cdot\w_j)\right)\d\theta\right|\nonumber\\&\leq
	6|\y_e+\theta\z| |\w_i||\w_j|.
\end{align*}
Thus from \eqref{appl5.11}  and \eqref{ef}, we infer
\begin{align}\label{appl5.12}
\left|\left(\int_0^1 \mathcal{C}_1''(\y_e+\theta\z)(\w_i\otimes\w_j)\d\theta,\w_k\right)\right|&\leq 6\left(\frac{2}{\mathrm{L}^d}\right)^{\frac{3}{2}}\int_{\mathbb{T}^d}(|\y_e(x)|+|\z(x)|)\d x.
\end{align}
Let us calculate
\begin{align*}
\int_{\mathbb{T}^d} |\z(x)|\d x=\int_{\mathbb{T}^d} \left|\sum\limits_{k=1}^n z_k \w_k(x) \right|\d x \leq\int_{\mathbb{T}^d}\left(\sum\limits_{k=1}^n|z_k||\w_k(x)|\right)\d x\leq|\v|_{\mathbb{R}^n}
\left(\sum\limits_{k=1}^n\left(\int_{\mathbb{T}^d}|\w_k(x)|\d x\right)^{2}\right)^{1/2}\leq \sqrt{n} \mathrm{L}^{\frac{d}{2}}|\v|_{\mathbb{R}^n}.
\end{align*}
Thus by the above estimate, \eqref{appl5.12} implies that 
\begin{align}\label{appl5.13}
\left|\left(\int_0^1 \mathcal{C}_1''(\y_e+\theta\z)(\w_i\otimes\w_j)\d\theta,\w_k\right)\right|&\leq 
6\left(\frac{2}{\mathrm{L}^d}\right)^{\frac{3}{2}}\left(C_1+\sqrt{n}\mathrm{L}^{\frac{d}{2}}|\v|_{\mathbb{R}^n}\right),
\end{align}
where $\|\y_{e}\|_{\widetilde{\mathbb{L}}^1}\leq C_1$ (see \eqref{eg} below). Using \eqref{appl5.13}, we obtain from \eqref{appl5.10} that
\begin{align}\label{appl5.15}
|\mathcal{N}(\v)|_{\mathbb{R}^n}&\leq\gamma_1(1+|\v|_{\mathbb{R}^n})|\v|^2_{\mathbb{R}^n},
\end{align}
where $\gamma_1:=6\beta\left(\frac{2n}{\mathrm{L}^d}\right)^{\frac{3}{2}}\max\left\{C_1, \sqrt{n}\mathrm{L}^{\frac{d}{2}}\right\}$. Combining \eqref{appl5.0.1} and \eqref{appl5.15}, and using it in \eqref{stbl2}, we find 
\begin{align*}
|\v(t)|_{\mathbb{R}^n}\leq \hat{M}e^{-\sigma t}|\v_0|_{\mathbb{R}^n}+\hat{M}\gamma_0'\int_0^t e^{-\sigma (t-s)}|\v(s)|_{\mathbb{R}^n}^2\d s+\hat{M}\gamma_1 \int_0^t e^{-\sigma (t-s)}|\v(s)|_{\mathbb{R}^n}^3\d s,	
\end{align*}
where $\gamma_0'=\gamma_0+\gamma_1.$ Using the Gronwall lemma, the above inequality implies 
\begin{align}\label{appl5.16}
|\v(t)|_{\mathbb{R}^n}\leq \hat{M}|\v_0|_{\mathbb{R}^n}\exp\left(-\sigma t+\gamma_0'\int_0^t |\v(s)|_{\mathbb{R}^n}\d s+\gamma_1\int_0^t |\v(s)|^2_{\mathbb{R}^n}\d s\right),
\end{align}
for $t>0.$ Let us suppose that the initial data $\v_0$ be such that $\hat{M}|\v_0|_{\mathbb{R}^n}<\rho$ for some constant $\rho>0,$ which will be specified later. We want to show that $|\v(t)|_{\mathbb{R}^n}<\rho$ for all $t>0$. Suppose this does not hold true. Then by the continuity of $\v(\cdot)$, there exists a time $T_1>0$ such that 
\begin{align}\label{appl5.17}
	|\v(t)|_{\mathbb{R}^n}<\rho \  \text{for} \ t\in[0,T_1) \ \text{and} \ |\v(T_1)|_{\mathbb{R}^n}=\rho.
\end{align}
For $t=T_1$, \eqref{appl5.16} implies that 
\begin{align*}
|\v(T_1)|_{\mathbb{R}^n}\leq \hat{M}|\v_0|_{\mathbb{R}^n}\exp\left\{\left(-\sigma+\gamma_0'\rho+\gamma_1\rho^2\right)T_1\right\}.
\end{align*}
It can be easily seen that the quadratic equation $-\sigma+\gamma_0'\rho+\gamma_1\rho^2=0$ has positive discriminant, so it has real roots. Even, one can conclude that for $\rho<\rho_1:=-\frac{\gamma_0'}{2\gamma_1}+\sqrt{\frac{\sigma}{\gamma_1}+\left(\frac{\gamma_0'}{2\gamma_1}\right)^2}$, where $\rho_1>0$, we have $|\v(T_1)|_{\mathbb{R}^n}<\rho,$ which is a contradiction to \eqref{appl5.17}. Thus, we have 
\begin{align}\label{appl5.18}
	|\v(t)|_{\mathbb{R}^n}<\rho \  \text{for every} \ t\in[0,\infty).
\end{align}
Hence, from \eqref{appl5.16} and \eqref{appl5.18}, we conclude that $|\v(t)|_{\mathbb{R}^n}\leq \hat{M}|\v_0|_{\mathbb{R}^n}e^{-\delta_1t},$   for every $t\geq0,$ where $\delta_1=\sigma-\gamma_0'\rho-\gamma_1\rho^2>0$ , as we know $-\sigma+\gamma_0'\rho+\gamma_1\rho^2<0$ for all $\rho\in(0,\rho_1)$.  Then for $\delta\in(0,\delta_1)$, the above estimate  implies 
\begin{align}\label{appl5.19}
	\lim\limits_{t\to\infty} e^{\delta t}|\v(t)|_{\mathbb{R}^n}=0.
\end{align}
But $\delta$ is arbitrary in $(0,\sigma),$ if  $\rho$ is small enough. Since $\sigma>0$ is also arbitrary, then  $\delta$ is arbitrary in $(0,\infty).$ 
Hence, for the feedback controller $\mathbf{\overline{U}}(t)=\mathrm{R}\v(t)$ and sufficiently small initial data, the solution $\v(t)$ of \eqref{appl5} is defined for all $t\geq0$  and the stability condition \eqref{appl5.19} holds. Furthermore, as $\z(t)\in\mathcal{K}$ for $t\in[0,\infty)$, the solution $\z(t)$ of \eqref{appl4.0} is defined on $[0,\infty)$, for the feedback controller $\mathbf{\overline{U}}(t)=\mathrm{R}\mathcal{P}_{\mathcal{K}}\z(t)$, and the exponential stability condition $\lim\limits_{t\to\infty} e^{\delta t}|\z(t)|_{\mathbb{R}^n}=0$ holds.
\vskip 2mm
\noindent \textbf{Step 5.} \emph{The case $r>3$ and $q\in[3,r)$\  \text{and} \ $\gamma<0$.} In this case, we find 
\begin{align}\label{appl5.1}
	|\mathcal{N}(\v)|^2_{\mathbb{R}^n}&=\sum\limits_{k=1}^n\left(\sum\limits_{i,j=1}^{n}g_{ijk}^2(\z)z_iz_j\right)^2\nonumber\\&=\sum\limits_{k=1}^n\bigg[\beta\sum\limits_{i,j=1}^{n}\left(\int_0^1 \mathcal{C}_1''(\y_e+\theta\z)(\w_i\otimes\w_j)\d\theta,\w_k\right)z_iz_j+\gamma\sum\limits_{i,j=1}^{n}\left(\int_0^1 \mathcal{C}_2''(\y_e+\theta\z)(\w_i\otimes\w_j)\d\theta,\w_k\right)z_iz_j\bigg]^2\nonumber\\&\leq
	2\sum\limits_{k=1}^{n}\left(\beta\sum\limits_{i,j=1}^{n}\left(\int_0^1 \mathcal{C}_1''(\y_e+\theta\z)(\w_i\otimes\w_j)\d\theta,\w_k\right)z_iz_j\right)^2\nonumber\\&
	\quad+2\sum\limits_{k=1}^n\left(\gamma\sum\limits_{i,j=1}^{n}\left(\int_0^1 \mathcal{C}_2''(\y_e+\theta\z)(\w_i\otimes\w_j)\d\theta,\w_k\right)z_iz_j\right)^2=:
	2\beta^2\sum\limits_{k=1}^n E_k^1+2\gamma^2\sum\limits_{k=1}^n E_k^2.
\end{align}
We calculate $E_k^1$ by using discrete H\"older's inequality as follows:
\begin{align*}
	E_k^1&
	=\left(\sum\limits_{i=1}^n\left(\sum\limits_{j=1}^n\left(\int_0^1 \mathcal{C}_1''(\y_e+\theta\z)(\w_i\otimes\w_j)\d\theta,\w_k\right)z_j\right)z_i\right)^2\nonumber\\&\leq
	\sum\limits_{i,j=1}^n\left(\int_0^1 \mathcal{C}_1''(\y_e+\theta\z)(\w_i\otimes\w_j)\d\theta,\w_k\right)^2\left(\sum\limits_{i=1}^nz_i^2\right)
	\left(\sum\limits_{j=1}^nz_j^2\right).
\end{align*}
In a similar way, we can estimate $E_k^2$ also. Using the above estimates, we obtain from \eqref{appl5.1} that
\begin{align}\label{appl5.2}
	|\mathcal{N}(\v)|^2_{\mathbb{R}^n}&\leq2\beta^2|\v|^4_{\mathbb{R}^n}\sum\limits_{i,j,k=1}^n\left(\int_0^1 \mathcal{C}_1''(\y_e+\theta\z)(\w_i\otimes\w_j)\d\theta,\w_k\right)^2\nonumber\\&\quad
	+2\gamma^2|\v|^4_{\mathbb{R}^n}\sum\limits_{i,j,k=1}^n\left(\int_0^1 \mathcal{C}_2''(\y_e+\theta\z) (\w_i\otimes\w_j)\d\theta,\w_k\right)^2.
\end{align}
Moreover, we calculate
\begin{align}\label{appl5.3}
	\left|\left(\int_0^1 \mathcal{C}_1''(\y_e+\theta\z)(\w_i\otimes\w_j)\d\theta,\w_k\right)\right|=
	\left|\int_{\mathbb{T}^d}\left(\int_0^1 \mathcal{C}_1''(\y_e+\theta\z)(\w_i\otimes\w_j)\d\theta\right)\w_k(x)\d x\right|.
\end{align}
From \eqref{C.}, we have
\begin{align}\label{appl5.4}
	&\left|\int_0^1 \mathcal{C}_1''(\y_e+\theta\z)(\w_i\otimes\w_j)\d\theta\right|\nonumber\\&\leq(r-1)
	\left|\int_0^1 \hspace{-2mm} |\y_e+\theta\z|^{r-3} \left(((\y_e+\theta\z)\cdot\w_i)\w_j +((\y_e+\theta\z)\cdot\w_j) \w_i+(\y_e+\theta\z)(\w_i\cdot\w_j)\right)\d\theta\right|\nonumber\\&\quad+(r-1)(r-3)
	\left|\int_0^1 |\y_e+\theta\z|^{r-5} ((\y_e+\theta\z)\cdot\w_i)((\y_e+\theta\z)\cdot\w_j) (\y_e+\theta\z)\d\theta\right|\nonumber\\&\leq3(r-1)|\y_e+\theta\z|^{r-2} |\w_i||\w_j|+(r-1)(r-3)|\y_e+\theta\z|^{r-2}|\w_i||\w_j|\nonumber\\&\leq 2^{r-3}r(r-1)\left(|\y_e|^{r-2}+|\z|^{r-2}\right)|\w_i||\w_j|.
\end{align}
Using \eqref{appl5.4} and \eqref{ef} in \eqref{appl5.3}, we get
\begin{align}\label{appl5.4.0}
	\left|\left(\int_0^1 \mathcal{C}_1''(\y_e+\theta\z)(\w_i\otimes\w_j)\d\theta,\w_k\right)\right|\leq 2^{r-3}r(r-1)\left(\frac{2}{\mathrm{L}^d}\right)^{\frac{3}{2}}\int_{\mathbb{T}^d}
	\left(|\y_e|^{r-2}+|\z|^{r-2}\right)\d x.
\end{align} 
Now, we calculate 
\begin{align*}
\int_{\mathbb{T}^d} |\z(x)|^{r-2}\d x=\int_{\mathbb{T}^d}\bigg|\sum\limits_{k=1}^n z_k \w_k \bigg|^{r-2}\d x\leq |\v|^{r-2}\int_{\mathbb{T}^d}\bigg(\sum\limits_{k=1}^n |\w_k(x)|^2\bigg)^{\frac{r-2}{2}}\d x\leq 
|\v|^{r-2}\bigg(\frac{2n}{\mathrm{L}^d}\bigg)^{\frac{r-2}{2}}\mathrm{L}^d.
\end{align*}
Thus by the Sobolev embedding and the above inequality, we obtain from \eqref{appl5.4.0}
\begin{align}\label{appl5.6}
	\left|\left(\int_0^1 \mathcal{C}_1''(\y_e+\theta\z) (\w_i\otimes\w_j)\d\theta,\w_k\right)\right|\leq 2^{r-3}r(r-1) \left(\frac{2}{\mathrm{L}^d}\right)^{\frac{3}{2}} \bigg(C_2 +|\v|^{r-2}\bigg(\frac{2n}{\mathrm{L}^d}\bigg)^{\frac{r-2}{2}}\mathrm{L}^d \bigg),
\end{align}
where $\|\y_{e}\|_{\widetilde{\mathbb{L}}^{r-2}}^{r-2}\leq C_2$ (see \eqref{eg} below). Similarly, one can determine
\begin{align}\label{appl5.7}
\left|\left(\int_0^1 \mathcal{C}_2''(\y_e+\theta\z) (\w_i\otimes\w_j)\d\theta,\w_k\right)\right|\leq 2^{r-3}r(r-1) \left(\frac{2}{\mathrm{L}^d}\right)^{\frac{3}{2}} 
\bigg(C_3+|\v|^{q-2}\bigg(\frac{2n}{\mathrm{L}^d}\bigg)^{\frac{q-2}{2}}\bigg),
\end{align}
where $\|\y_{e}\|_{\widetilde{\mathbb{L}}^{q-2}}^{r-2}\leq C_3$ (see \eqref{eg} below).  Using \eqref{appl5.6}-\eqref{appl5.7}, we deduce from \eqref{appl5.2} that
\begin{align*}
|\mathcal{N}(\v)|_{\mathbb{R}^n}&\leq2^{r-3}r(r-1) \left(\frac{4n}{\mathrm{L}^d}\right)^{\frac{3}{2}}\beta \bigg(C_2 +|\v|^{r-2}\bigg(\frac{2n}{\mathrm{L}^d}\bigg)^{\frac{r-2}{2}}\mathrm{L}^d \bigg)|\v|^2_{\mathbb{R}^n}
\nonumber\\&+2^{r-3}r(r-1) \left(\frac{4n}{\mathrm{L}^d}\right)^{\frac{3}{2}}|\gamma|\bigg(C_3+|\v|^{q-2}\bigg(\frac{2n}{\mathrm{L}^d}\bigg)^{\frac{q-2}{2}}\bigg)|\v|^2_{\mathbb{R}^n},
\end{align*}
which implies 
\begin{align*}
	|\mathcal{N}(\v)|_{\mathbb{R}^n}\leq\gamma_2\left(|\v|^2_{\mathbb{R}^n}+|\v|^{r}_{\mathbb{R}^n}+|\v|^{q}_{\mathbb{R}^n}\right),
\end{align*}
where $\gamma_2:=2^{r-2}r(r-1) \left(\frac{4n}{\mathrm{L}^d}\right)^{\frac{3}{2}} \max\bigg\{\beta C_2+|\gamma|C_3,\beta\bigg(\frac{2n}{\mathrm{L}^d}\bigg)^{\frac{r-2}{2}}, |\gamma|\bigg(\frac{2n}{\mathrm{L}^d}\bigg)^{\frac{r-2}{2}}\bigg\}.$ Therefore, from \eqref{stbl} and \eqref{appl9}, we obtain
\begin{align}\label{appl5.9}
|\v(t)|_{\mathbb{R}^n}\leq \hat{M}e^{-\sigma t}|\v_0|_{\mathbb{R}^n}+\hat{M}\gamma_0'\int_0^t e^{-\sigma (t-s)}|\v(s)|^2_{\mathbb{R}^n}\d s+\hat{M}\gamma_2\int_0^t e^{-\sigma (t-s)}(|\v(s)|^{r}_{\mathbb{R}^n}+|\v(s)|^{q}_{\mathbb{R}^n})\d s,
\end{align}
where $\gamma_0'=\gamma_0+\gamma_2.$ We calculate by using Young's inequality as 
\begin{align*}
|\v|^{q}_{\mathbb{R}^n}=(|\v|^{q}_{\mathbb{R}^n})^{\frac{r-q}{(r-1)q}}(|\v|^{q}_{\mathbb{R}^n})^{\frac{r(q-1)}{q(r-1)}} \leq C_4|\v|^{r}_{\mathbb{R}^n}+\frac{\sigma}{4\gamma_2}|\v|_{\mathbb{R}^n},
\end{align*}
where $C_4:=\left[\frac{4\gamma_2}{\sigma}\left(\frac{r-q}{r-1}\right)\right]^{\frac{r-q}{q-1}}\left(\frac{q-1}{r-1}\right).$ Similarly, we calculate
\begin{align*}
	|\v|^2_{\mathbb{R}^n}=(|\v|^2_{\mathbb{R}^n})^{\frac{r-3}{2(r-1)}}(|\v|^2_{\mathbb{R}^n})^{\frac{r+1}{2(r-1)}}\leq
C_5|\v|^{\frac{r+1}{2}}_{\mathbb{R}^n}+\frac{\sigma}{4\gamma_0'}|\v|_{\mathbb{R}^n},
\end{align*}
where $C_5:=\left[\frac{2\gamma_0'}{\sigma}\frac{(r-3)}{(r-1)}\right]^{\frac{r-3}{2}}\left(\frac{2}{r-1}\right)$. Using the above estimates in \eqref{appl5.9}, we get
\begin{align*}
 |\v(t)|_{\mathbb{R}^n}&\leq \hat{M}e^{-\sigma t}|\v_0|_{\mathbb{R}^n}+\hat{M}\gamma_0'C_5\int_0^t e^{-\sigma (t-s)}|\v(s)|^{\frac{r+1}{2}}_{\mathbb{R}^n}\d s \nonumber\\&\quad +\hat{M}(1+C_4)\gamma_2\int_0^t e^{-\sigma (t-s)}|\v(s)|^{r}_{\mathbb{R}^n}\d s+\frac{\hat{M}\sigma}{2}\int_0^t e^{-\sigma(t-s)}|\v(s)|_{\mathbb{R}^n}\d s.
\end{align*}
Now by applying the Gronwall inequality, we obtain
\begin{align}\label{appl5.9.1}
	|\v(t)|_{\mathbb{R}^n}\leq \hat{M}|\v_0|_{\mathbb{R}^n}\exp\left(-\frac{\sigma}{2}t+(1+C_4)\gamma_2\int_0^t |\v(s)|^{r-1}_{\mathbb{R}^n}\d s+\gamma_0'C_5\int_0^t |\v(s)|^{\frac{r-1}{2}}_{\mathbb{R}^n}\d s\right),
\end{align}
for all $t>0$. Similar to the case $r=3$ with $\gamma=0$, let us suppose that the initial data $\v_0$ is such that $M|\v_0|_{\mathbb{R}^n}<\rho$ for some $\rho>0$. We want to then show that $|\v(t)|_{\mathbb{R}^n}<\rho$ for all $t>0.$ Suppose it is not true. Then as $\v(\cdot)$ is continuous function, we can find some $T_1>0$ such that 
\begin{align}\label{appl5.9.2}
	|\v(t)|_{\mathbb{R}^n}<\rho \  \text{ for } \ t\in(0,T_1) \ \text{ and } \ |\v(T_1)|_{\mathbb{R}^n}=\rho.
\end{align}
Then from \eqref{appl5.9.1}, we calculate
\begin{align*}
|\v(T_1)|_{\mathbb{R}^n}\leq \hat{M}|\v_0|_{\mathbb{R}^n} \exp\left\{\left(-\frac{\sigma}{2}+(1+C_4)\gamma_2\rho^{r-1} +\gamma_0'C_5\rho^{\frac{r-1}{2}}\right)t\right\}.
\end{align*}
We see that the quadratic equation $(1+C_4)\gamma_2\left(\rho^{\frac{r-1}{2}}\right)^2+\gamma_0'C_5\rho^{\frac{r-1}{2}}-\frac{\sigma}{2}=0$ has a positive discriminant, so it has real roots. Thus for 
\begin{align*}
	\rho<\rho_1:=\left(\frac{-\gamma_0'C_5+\sqrt{(\gamma_0'C_5)^2+2(1+C_4)\gamma_2\sigma}}{2(1+C_4)\gamma_2}\right)^{\frac{2}{r-1}},
\end{align*}
we have $|\v(T_1)|_{\mathbb{R}^n}<\rho$, which contradicts \eqref{appl5.9.2}. Thus, we conclude that $|\v(t)|_{\mathbb{R}^n}<\rho$ on $t\in(0,\infty).$ Proceeding similar to the case $r=3$ with $\gamma=0$, we can find the existence of $\delta>0$, arbitrarily chosen on $(0,\infty),$ such that $\lim\limits_{t\to\infty} e^{\delta t}|\v(t)|_{\mathbb{R}^n}=0.$ Furthermore, the solution $\z(t)$ of \eqref{appl4.0} is defined on $[0,\infty)$ for the feedback controller $\mathbf{\overline{U}}(t)=\mathrm{R}\mathcal{P}_{\mathcal{K}}\z(t)$, and the exponential stability condition $\lim\limits_{t\to\infty} e^{\delta t}|\z(t)|_{\mathbb{R}^n}=0$ holds.
\end{proof}

\begin{appendix}
	\renewcommand{\thesection}{\Alph{section}}
	\numberwithin{equation}{section}\section{Supporting results}\label{sec2}
The purpose of this section is to introduce some basic results regarding the  nonlinear operators  appearing in the system  \eqref{1p44}. 
\begin{proposition}\label{prop3.1}
	Define the operator $\mathcal{M}(\cdot):\mathrm{D}(\mathcal{M})\to\H$ by
	\begin{align*}
		\mathcal{M}(\cdot)=\mu\A+\wi{\mathcal{B}}(\cdot)+\alpha\I+\beta\wi{\mathcal{C}_1}(\cdot)+\gamma\wi{\mathcal{C}_2}(\cdot), 
	\end{align*}
	where $\mathrm{D}(\mathcal{M})=\{\y\in\V\cap\wi\L^{r+1}:\A\y\in\H\}.$ Then, for $d=2,3$ with $r>3$, the operator $\mathcal{M}+\kappa\mathrm{I}$ is $m$-accretive in $\H$ with $\D(\mathcal{M}+\kappa\mathrm{I})=\D(\A)$ for some sufficiently large $\kappa>0$. 
\end{proposition}

\begin{proof}
	We shall first show that $\mathcal{M}+\kappa\mathrm{I}$ is a monotone operator for some sufficiently large $\kappa>0$. Then we prove that $\mathcal{M}+\kappa\mathrm{I}$ is coercive and demicontinuous, which imply  the $m$-accretivity of the operator $\mathcal{M}+\kappa\I$. Finally, we demonstrate that $\D(\mathcal{M}+\kappa\mathrm{I})=\D(\A)$, for some sufficiently large $\kappa>0$. The proof is divided into the following steps:
	
	\vskip 2mm
	\noindent
	\textbf{Step I:} \textsl{The operator $\mathcal{M}+\kappa\mathrm{I}$ is monotone for some sufficiently large $\kappa>0$.} Let us define $\hat{\y}:=\y_1-\y_2$. We estimate $	\langle\A\y_1-\A\y_2,\hat{\y}\rangle $ by	using an integration by parts as
	\begin{align}\label{3.1}
		\langle\A\y_1-\A\y_2,\hat{\y}\rangle =\|\nabla\hat{\y}\|^2_{\H}.
	\end{align}
	Note that $\langle\mathcal{B}(\y_1+\y_e,\hat{\y}),\hat{\y}\rangle=0$, which  along with H\"older's and Young's inequalities implies
	\begin{align}\label{3.2}
		|\langle\wi{\mathcal{B}}(\y_1)-\wi{\mathcal{B}}(\y_2),\hat{\y}\rangle|=	|\langle\mathcal{B}(\y_1+\y_e)-\mathcal{B}(\y_2+\y_e),\hat{\y}\rangle| 
		\leq \frac{\mu }{2}\|\nabla\hat{\y}\|_{\H}^2+\frac{1}{2\mu }\||\y_2+\y_e|\hat{\y}\|_{\H}^2.
	\end{align} 
	Using H\"older's and Young's inequalities, we estimate the term $\||\y_2+\y_e|\hat{\y}\|_{\H}^2$ as
	\begin{align}\label{3.3}
		\int_{\mathbb{T}^d}|\y_2(x)+\y_e(x)|^2|\hat{\y}(x)|^2\d x &=\int_{\mathbb{T}^d}|\y_2(x)+\y_e(x)|^2|\hat{\y}(x)|^{\frac{4}{r-1}}|\hat{\y}(x)|^{\frac{2(r-3)}{r-1}}\d x\nonumber\\&\leq\frac{\varepsilon\beta\mu }{2}\||\y_2+\y_e|^{\frac{r-1}{2}}\hat{\y}\|_{\H}^2+\frac{r-3}{r-1}\left[\frac{4}{\varepsilon\beta\mu (r-1)}\right]^{\frac{2}{r-3}}\|\hat{\y}\|_{\H}^2,
	\end{align}
	for $r>3$ and $\varepsilon>0$. Using \eqref{3.3} in \eqref{3.2}, we find 
	\begin{align}\label{3.4}
		|\langle\wi{\mathcal{B}}(\y_1)-\wi{\mathcal{B}}(\y_2),\hat{\y}\rangle|\leq
		\frac{\mu }{2}\|\nabla\hat{\y}\|_{\H}^2 +\frac{\varepsilon\beta}{4}\||\y_2+\y_e|^{\frac{r-1}{2}}\hat{\y}\|_{\H}^2 +\varrho_{\eps}\|\hat{\y}\|_{\H}^2,
	\end{align}
	where $\varrho_{\eps}:=\frac{r-3}{2\mu(r-1)}\left[\frac{4}{\varepsilon\beta\mu (r-1)}\right]^{\frac{2}{r-3}}.$
	From \eqref{C1}, we can write 
	\begin{align}\label{3.5}
		\beta\langle\wi{\mathcal{C}}_1(\y_1)-\wi{\mathcal{C}}_1(\y_2),\hat{\y}\rangle&\geq \frac{\beta}{2}\left(\||\y_1+\y_e|^{\frac{r-1}{2}}\hat{\y}\|_{\H}^2+\||\y_2+\y_e|^{\frac{r-1}{2}}\hat{\y}\|_{\H}^2\right)\nonumber\\&\geq\frac{\beta}{4}
		\left(\||\y_1+\y_e|^{\frac{r-1}{2}}\hat{\y}\|_{\H}^2+\||\y_2+\y_e|^{\frac{r-1}{2}}\hat{\y}\|_{\H}^2\right)+\frac{\beta}{2^{r}}\|\hat{\y}\|_{\wi\L^{r+1}}^{r+1}.
	\end{align}
	Using Taylor's formula and \eqref{C} (for $\mathcal{C}_2(\cdot)$), we calculate
	\begin{align}\label{c2.1}
		|\gamma\langle\wi{\mathcal{C}}_2(\y_1)-\wi{\mathcal{C}}_2(\y_2),\hat{\y}\rangle|\leq|\gamma|
		2^{q-2}q\int_{\mathbb{T}^d} (|\y_1(x)+\y_e(x)|^{q-1}+|\y_2(x)+\y_e(x)|^{q-1})|\hat{\y}(x)|^2 \d x. 
	\end{align}
	Using H\"older's and Young's inequalities, we obtain
	\begin{align}\label{c2.1.1}
		\int_{\mathbb{T}^d} |\y_1(x)+\y_e(x)|^{q-1}|\hat{\y}(x)|^2 \d x&=\int_{\mathbb{T}^d} |\y_1(x)+\y_e(x)|^{q-1}|\hat{\y}|^{2\left(\frac{q-1}{r-1}\right)}|\hat{\y}(x)|^{2\left(1-\frac{q-1}{r-1}\right)}\d x \nonumber\\&\leq 
		\frac{\wi{\varepsilon}\beta}{2^{q}q|\gamma|}\||\y_1+\y_e|^{\frac{r-1}{2}}\hat{\y}\|_{\H}^2+
		\rho_{\wi{\varepsilon}}\|\hat{\y}\|_{\H}^2,
	\end{align}
	where $\rho_{\wi{\varepsilon}}:=\frac{r-q}{r-1}\left[\frac{2^{q}q|\gamma|(q-1)}{\wi{\varepsilon}\beta(r-1)}\right]^{\frac{q-1}{r-q}}$ and $\wi{\varepsilon}>0$. Similarly, one can calculate 
	\begin{align}\label{c2.1.2}
		\int_{\mathbb{T}^d}|\y_2(x)+\y_e(x)|^{q-1}|\hat{\y}(x)|^2 \d x\leq
		\frac{\varepsilon\beta}{2^{q}q|\gamma|}\||\y_2+\y_e|^{\frac{r-1}{2}}\hat{\y}\|_{\H}^2+\rho_{\varepsilon}\|\hat{\y}\|_{\H}^2,
	\end{align}
	where $\rho_{\varepsilon}:=\frac{r-q}{r-1}\left[\frac{2^{q}q|\gamma|(q-1)}{\varepsilon\beta(r-1)}\right]^{\frac{q-1}{r-q}}$. Using \eqref{c2.1.1}-\eqref{c2.1.2} in \eqref{c2.1}, we write 
	\begin{align}\label{c2.2}
		|\gamma\langle\wi{\mathcal{C}}_2(\y_1)-\wi{\mathcal{C}}_2(\y_2),\hat{\y}\rangle|\leq
		\frac{\wi{\varepsilon}\beta}{4}\||\y_1+\y_e|^{\frac{r-1}{2}}\hat{\y}\|_{\H}^2+\frac{\varepsilon\beta}{4}\||\y_2+\y_e|^{\frac{r-1}{2}}\hat{\y}\|_{\H}^2+(\rho_{\wi{\varepsilon}}+\rho_{\varepsilon})\|\hat{\y}\|_{\H}^2.	
	\end{align}
	Combining \eqref{3.1}, \eqref{3.4}-\eqref{3.5} and \eqref{c2.2}, we conclude 
	\begin{align*}
		\langle(\mathcal{M}+\kappa\I)(\y_1)-(\mathcal{M}+\kappa\I)(\y_2),\hat{\y}\rangle&\geq\frac{\mu }{2}\|\nabla\hat{\y}\|_{\H}^2+(\alpha+\kappa-(\varrho_{\eps}+\rho_{\wi{\eps}}+\rho_{\eps}))\|\hat{\y}\|_{\H}^2+\frac{\beta}{2^{r}} \|\hat{\y}\|_{\wi\L^{r+1}}^{r+1}\nonumber\\&\quad+ \frac{\beta}{2}\left(\frac{1}{2}-\varepsilon\right)\||\y_2+\y_e|^{\frac{r-1}{2}}\hat{\y}\|_{\H}^2+\frac{\beta}{4}\left(1- \wi\varepsilon\right) \||\y_1+\y_e|^{\frac{r-1}{2}}\hat{\y}\|_{\H}^2.
	\end{align*}
	Choosing $0<\varepsilon\leq\frac{1}{2}$, $0<\wi{\varepsilon}\leq1$ and $\kappa\geq\varrho_{\eps}+\rho_{\wi{\eps}}+\rho_{\eps}$, we obtain 
	\begin{align*}
		\langle(\mathcal{M}+\kappa\I)(\y_1)-(\mathcal{M}+\kappa\I)(\y_2),\y_1-\y_2\rangle
		\geq0.
	\end{align*} 
	Thus the operator $\mathcal{M}+\kappa\I$ is monotone for any $\kappa\geq\varrho_{\eps}+\rho_{\wi{\eps}}+\rho_{\eps}$.
	\vskip 2mm
	\noindent
	\textbf{Step II:} \textsl{The operator $\mathcal{M}+\kappa\I$ is demicontinuous and coercive.} One can show that the operator $\mathcal{M}+\kappa\I:\V\cap\widetilde{\L}^{r+1}\to\V'+\widetilde{\L}^{\frac{r+1}{r}}$ is demicontinuous and coercive in a similar way as in the proof of \cite[Proposition 3.1]{sKM}.
	\vskip 2mm
	\noindent
	\textbf{Step IV:} \textsl{ $\D(\mathcal{M}+\kappa\I)=\D(\A)$ for sufficiently large $\kappa>0$}. 	Note that the space $\V\cap\widetilde{\L}^{r+1}$ is reflexive. Since the operator $\mathcal{M}+\kappa\I$ is monotone, hemicontinuous and coercive from $\V\cap\widetilde{\L}^{r+1}$ to $\V'+\widetilde{\L}^{\frac{r+1}{r}}$, then by an application of \cite[Example 2.3.7]{OPHB}, we obtain that the operator $\mathcal{M}+\kappa\I$ is maximal monotone in $\H$ with domain $\mathrm{D}(\mathcal{M}+\kappa\I)\supseteq\mathrm{D}(\mathrm{A})$ for any $\kappa\geq\varrho_{\eps}+\rho_{\wi{\eps}}+\rho_{\eps}$.
	
	Let us define
	\begin{align*}
		\mathcal{S}(\y):=\mathcal{M}(\y)+\kappa\y=\mu\A\y+\wi{\mathcal{B}}(\y)+\alpha\y+\beta\wi{\mathcal{C}}_1(\y)+\gamma\wi{\mathcal{C}}_2(\y)+\kappa\y,
	\end{align*}
	with $\D(\mathcal{S})=\{\y\in\V\cap\wi\L^{r+1}:\mu\A\y+\wi{\mathcal{B}}(\y)+(\alpha+\kappa)\y+\beta\wi{\mathcal{C}}_1(\y)+
	\gamma\wi{\mathcal{C}}_2(\y)\in\H\}$.  
	Let us take the inner product in $\mathcal{S}(\y)$ with $\A(\y+\y_e)$ to get
	\begin{align*}
		&\mu\|\A(\y+\y_e)\|_{\H}^2+\beta(\mathcal{C}_1(\y+\y_e),\A(\y+\y_e))+(\alpha+\kappa)\|\nabla(\y+\y_e)\|_{\H}^2\nonumber\\&=(\mathcal{S}(\y),\A(\y+\y_e))-(\mathcal{B}(\y+\y_e),\A(\y+\y_e))-\gamma(\mathcal{C}_2(\y+\y_e),\A(\y+\y_e))+(\mathscr{S}(\y_e),\A(\y+\y_e)),
	\end{align*}
	where $\mathscr{S}(\cdot):=\mu\A+\mathcal{B}(\cdot)+\beta\mathcal{C}_1(\cdot)+ \gamma\mathcal{C}_2(\cdot)+(\alpha+\kappa)\I.$  Note that 
	\begin{align*}
		\|\mathscr{S}(\y_e)\|_{\H}\leq C\|\A\y_e\|_{\H}\ \text{ for all }\ \y_e\in\D(\A). 
	\end{align*}
	A calculation similar to \eqref{3.4} gives
	\begin{align}\label{Ayc1}
		|(\B(\y+\y_e),\A(\y+\y_e))|\leq\frac{\mu}{2}\|\A(\y+\y_e)\|_{\H}^2+\frac{\beta}{4} \||\y+\y_e|^{\frac{r-1}{2}}\nabla(\y+\y_e)\|_{\H}^2+\eta_1\|\nabla(\y+\y_e)\|_{\H}^2,
	\end{align}
	where $\eta_1:=\frac{r-3}{2\mu(r-1)}\left[\frac{4}{\beta\mu(r-1)}\right]^{\frac{2}{r-3}}.$ We now calculate by using H\"older's inequality
	\begin{align}\label{Ayc3}
		&\int_{\mathbb{T}^d}|\y(x)+\y_e(x)|^{q-1}|\nabla(\y(x)+\y_e(x))|^2\d x\nonumber\\&\leq
		\left(\int_{\mathbb{T}^d}|\y(x)+\y_e(x)|^{r-1}|\nabla(\y(x)+\y_e(x))|^2\d x\right)^\frac{q-1}{r-1}
		\left(\int_{\mathbb{T}^d}|\nabla(\y(x)+\y_e(x))|^2 \d x\right)^\frac{r-q}{r-1}.
	\end{align}
	Thus, from \eqref{3}, \eqref{Ayc3}, the identity $\nabla|\y|^k=k \sum\limits_{j=1}^dy_j\nabla y_j|\y|^{k-2}$, and H\"older's and Young's inequalities, we calculate 
	\begin{align}\label{Ayc2}
		|\gamma(\mathcal{C}_2(\y+\y_e),\mathrm{A}(\y+\y_e))|&\leq
		|\gamma|\||\y+\y_e|^{\frac{q-1}{2}}\nabla(\y+\y_e)\|_{\H}^{2} +4|\gamma|\left[\frac{q-1}{(q+1)^2}\right]\|\nabla|\y+\y_e|^{\frac{q+1}{2}}\|_{\H}^{2}
		\nonumber\\&\leq q|\gamma|\||\y+\y_e|^{\frac{r-1}{2}}\nabla(\y+\y_e)\|_{\H}^{\frac{2(q-1)}{r-1}}\|\nabla(\y+\y_e)\|_{\H}^{\frac{2(r-q)}{r-1}}\nonumber\\&\leq 
		\frac{\beta}{4}\||\y+\y_e|^{\frac{r-1}{2}}\nabla(\y+\y_e)\|_{\H}^2+\eta_2\|\nabla(\y+\y_e)\|_{\H}^2,
	\end{align}
	where $\eta_2:=(q|\gamma|)^{\frac{r-1}{r-q}}\left[\frac{4}{\beta}\left(\frac{q-1}{r-1}\right)\right]^ {\frac{q-1}{r-q}}\left({\frac{r-q}{r-1}}\right).$ From \eqref{Ayc1}-\eqref{Ayc2}, we obtain 
	\begin{align*}
		&\frac{\mu}{2}\|\A(\y+\y_e)\|_{\H}^2+\frac{\beta}{2}\||\y+\y_e|^{\frac{r-1}{2}}\nabla(\y+\y_e)\|_{\H}
		^{2}+(\alpha+\kappa-(\eta_1+\eta_2))\|\nabla(\y+\y_e)\|_{\H}^2\nonumber\\&\leq (\mathcal{S}(\y),\A(\y+\y_e))+(\mathscr{S}(\y_e),\A(\y+\y_e)).
	\end{align*}
	For any $\kappa\geq\eta_1+\eta_2$, in view of the Cauchy-Schwarz inequality, we reach at
	\begin{align*}
		\frac{\mu}{2}\|\A(\y+\y_e)\|_{\H}\leq\|\mathcal{S}(\y)\|_{\H}+C\|\A\y_e\|_{\H},
	\end{align*}
	which provides  $\D(\mathcal{S})=\D(\mathcal{M}+\kappa\I)\subseteq\D(\A)$. 
	
	Now, choose $\kappa$ sufficiently large so that $\kappa\geq\max\{\varrho_{\eps}+\rho_{\wi{\eps}}+\rho_{\eps},\eta_1+\eta_2\}$. This completes the proof. 
\end{proof}
\begin{remark}\label{RK}
	For $d=r=3$, by the Sobolev embedding $\V\subset\wi{\L}^{4}$ so that $\V\cap\wi{\L}^{4}=\V$. We have from \eqref{C1} that 
	\begin{align}\label{rk1}
		\beta\langle\wi{\mathcal{C}}_1(\y_1)-\wi{\mathcal{C}}_1(\y_2),\y_1-\y_2\rangle\geq \frac{\beta}{2}\||\y_1+\y_e|(\y_1-\y_2)\|_{\H}^2+\frac{\beta}{2}\||\y_2+\y_e|(\y_1-\y_2)\|_{\H}^2.		
	\end{align}
	Moreover, we calculate
	\begin{align}\label{rk2}
		|\langle\wi{\mathcal{B}}(\y_1)-\wi{\mathcal{B}}(\y_2),\y_1-\y_2\rangle|\leq
		\mu\|\nabla(\y_1-\y_2)\|_{\H}^2+\frac{1}{4\mu}\||\y_2+\y_e|(\y_1-\y_2)\|_{\H}^2	.
	\end{align} 
	Modifying the calculation \eqref{c2.1.1} for $r=3$, we get  
	\begin{align*}
		\int_{\mathbb{T}^d} |\y_1(x)+\y_e(x)|^{q-1}|\y_1(x)-\y_2(x)|^2 \d x\leq \frac{1}{2^{q}q|\gamma|\mu}\||\y_1+\y_e|(\y_1-\y_2)\|_{\H}^2+
		\wi\rho_1\|\y_1-\y_2\|_{\H}^2,
	\end{align*}
	where $\wi\rho_1:=\rho_{\frac{1}{\mu\beta}}=\left(2^{q-1}q|\gamma|\mu(q-1)\right)^{\frac{q-1}{3-q}}\left(\frac{3-q}{2}\right)$. Similarly from \eqref{c2.1.2} for $r=3$, we infer
	\begin{align}\label{rk3}
		\int_{\mathbb{T}^d}|\y_2(x)+\y_e(x)|^{q-1}|\y_1(x)-\y_2(x)|^2 \d x\leq
		\frac{\left(\beta-\frac{1}{2\mu}\right)}{2^{q}q|\gamma|}\||\y_2+\y_e|(\y_1-\y_2)\|_{\H}^2+\wi\rho_2\|\y_1-\y_2\|_{\H}^2	,
	\end{align}
	where $\wi\rho_2:=\rho_{\frac{1}{\beta}\left(\beta-\frac{1}{2\mu}\right)}=\left(\frac{2^{q-1}q|\gamma|(q-1)}{\left(\beta-\frac{1}{2\mu}\right)}\right)^{\frac{q-1}{3-q}}\left(\frac{3-q}{2}\right)$. Thus, for $r=3$, from \eqref{rk1}-\eqref{rk3}, we conclude 
	\begin{align}\label{rk4}
		\langle(\mathcal{M}+\kappa\I)(\y_1)-(\mathcal{M}+\kappa\I)(\y_2),\y_1-\y_2\rangle
		&\geq
		(\kappa-(\wi\rho_1+\wi\rho_2))\|\y_1-\y_2\|_{\H}^2+\frac{1}{2}\left(\beta-\frac{1}{2\mu}\right)\||\y_1+\y_e|(\y_1-\y_2)\|_{\H}^2\nonumber\\&\quad+\frac{1}{4}
		\left(\beta-\frac{1}{2\mu}\right)\||\y_2+\y_e|(\y_1-\y_2)\|_{\H}^2
	\end{align}
	From \eqref{rk4}, it is clear that for $\kappa\geq\wi\rho_1+\wi\rho_2$ and $2\beta\mu>1,$ the nonlinear operator $\mathcal{M}(\cdot)+\kappa\I$ is globally monotone (in fact strongly). The other properties like demincontinuity, coercivity, etc. of the operator $\mathcal{M}(\cdot)+\kappa\I$ for $d=r=3$ can be established in a similar way as for $r>3$ in Proposition \ref{prop3.1}.
\end{remark}

\begin{proposition}\label{prop3.3}
	Let $\mathcal{K}$ be a subset of $\H$ which satisfies Hypothesis \ref{AssupK}. Let $\I_\mathcal{K}$ be the indicator function defined in \eqref{appl1.2}. Define the multivalued operator  $\mathfrak{G}:\mathrm{D}(\mathfrak{G})\multimap\H$ by 
	\begin{equation*}
		\mathfrak{G}(\cdot) = \mu\mathrm{A} +\wi{\mathcal{B}}(\cdot)+\alpha\I+\beta\wi{\mathcal{C}}_1(\cdot)+\gamma\wi{\mathcal{C}}_2(\cdot)+\partial\I_{\mathcal{K}}(\cdot)+\kappa\I,
	\end{equation*}
	with the domain $\mathrm{D}(\mathfrak{G})=\{\y\in\H: \varnothing \neq \mathfrak{G}(\y)\subset\H\}$, for $\kappa>0$. Then, for $r>3$ in $d=2,3$ and $d=r=3$ with $2\beta\mu>1$, $\mathfrak{G}$ is a maximal monotone operator in $\H\times\H$ for sufficiently large $\kappa>0$ with $\mathrm{D}(\mathfrak{G}) = \mathrm{D}(\mathrm{A})\cap\mathrm{D}(\partial\I_{\mathcal{K}})= \mathrm{D}(\mathrm{A})\cap\mathcal{K}$.  
\end{proposition}
\begin{proof}
	
	It has been shown in Proposition \ref{prop3.1} that  the operator $\mathcal{S}(\cdot)=\mu\mathrm{A} +\wi{\mathcal{B}}(\cdot)+ \alpha\I+\beta\wi{\mathcal{C}}_1(\cdot)+\gamma\wi{\mathcal{C}}_2(\cdot)+\kappa\mathrm{I}$ is maximal monotone with domain $\mathrm{D}(\mathcal{S})=\mathrm{D}(\mathrm{A})$ in $\H$, for sufficiently large $\kappa>0$. On the other hand, the subdifferential operator $\partial\I_{\mathcal{K}}(\cdot)$ is a maximal monotone operator in $\H\times\H$ (see \cite[Theorem 2.1, pp. 62]{VB2}).  Therefore $\mathfrak{G}$ is the monotone operator as it is the sum of two monotone operators $\mathcal{S}(\cdot)$ and $\partial\I_{\mathcal{K}}(\cdot)$. But the sum of two maximal monotone operator is not maximal monotone always. In order to prove $\mathfrak{G}$ is maximal monotone, we need to show that 
	\begin{align}\label{3.2.1}
		\mathrm{R}(\mathrm{I}+\mathfrak{G})=\H.
	\end{align} 
	Let us first consider the equation
	\begin{equation}\label{3.2.2}
		\y_{\lambda}+\mu\mathrm{A}\y_{\lambda}+\wi{\mathcal{B}}(\y_{\lambda})+\alpha\y_\lambda+\beta\wi{\mathcal{C}}_1 (\y_{\lambda})+\gamma\wi{\mathcal{C}}_2(\y_\lambda)+(\partial\I_{\mathcal{K}})_{\lambda}(\y_{\lambda})+\kappa\y_{\lambda}=\f,
	\end{equation}
	where $(\partial\I_{\mathcal{K}})_{\lambda}$  is the Yosida approximation of $\partial\I_{\mathcal{K}}$ (see \eqref{subdiff}). We next show that the equation \eqref{3.2.2} has at most one solution (uniqueness) for any given $\f\in\H$.
	\vskip 2mm
	\noindent
	\emph{Uniqueness.} Let $\f\in\H$ be arbitrary but fixed. 
	Let us set $\z_\lambda=\y_{\lambda}^1-\y_{\lambda}^2$, where $\y_{\lambda}^1$ and $\y_{\lambda}^2$ are any two solutions of \eqref{3.2.2}. Then we have 
	\begin{align}\label{3.2.2.}
		&(\alpha+\kappa+1)\z_\lambda+\mu\A\z_\lambda+(\partial\I_{\mathcal{K}})_{\lambda}(\y_{\lambda}^1)-(\partial\I_{\mathcal{K}})_{\lambda}(\y_{\lambda}^2)\nonumber\\&=
		-(\wi{\mathcal{B}}(\y_{\lambda}^1)-\wi{\mathcal{B}}(\y_{\lambda}^2))-
		\beta(\wi{\mathcal{C}}_1(\y_{\lambda}^1)-\wi{\mathcal{C}}_1(\y_{\lambda}^2))-
		\gamma(\wi{\mathcal{C}}_2(\y_{\lambda}^1)-\wi{\mathcal{C}}_2(\y_{\lambda}^2)).
	\end{align}
	Taking the inner product with $\z_{\lambda}$ in \eqref{3.2.2.} and using  \eqref{3.4}-\eqref{3.5} and \eqref{c2.2} for $\varepsilon=1, \ \wi\varepsilon=\frac{1}{2}$, we get 
	\begin{align*}
		\frac{\mu}{2}\|\nabla\z_\lambda\|_{\H}^2+(\alpha+\kappa+1-(\varrho_1+\rho_{\frac{1}{2}}+\rho_1))\|\z_\lambda\|_{\H}^2\leq0,
	\end{align*}
	where we have used the monotonicity of $(\partial\I_{\mathcal{K}})_{\lambda}(\cdot)$ also. Thus, for any $\kappa\geq\varrho_1+\rho_{\frac{1}{2}}+\rho_1$, we deduce $\y_{\lambda}^1=\y_{\lambda}^2$ and this completes the uniqueness of the Yosida approximated problem \eqref{3.2.2}.
	\vskip 2mm
	\noindent
	\emph{Uniform bounds of $\y_\lambda$.} 
	Taking the inner product with $\y_\lambda$ in \eqref{3.2.2}, we obtain
	\begin{align}\label{ey}
		&(\alpha+\kappa+1)\|\y_\lambda\|_{\H}^2+\mu\|\nabla\y_\lambda\|_{\H}^2+\beta(\mathcal{C}_1(\y_{\lambda}+\y_e)-\mathcal{C}_1(\y_e),\y_\lambda)+((\partial\I_{\mathcal{K}})_{\lambda}(\y_\lambda),\y_\lambda)\nonumber\\&=(\f,\y_\lambda)-(\mathcal{B}(\y_{\lambda}+\y_e)-\mathcal{B}(\y_e),\y_\lambda)-\gamma(\mathcal{C}_2(\y_\lambda+\y_e)-\mathcal{C}_2(\y_e),\y_\lambda) .
	\end{align}
	From \eqref{3.4}-\eqref{3.5} and \eqref{c2.2}, we have
	\begin{align}
		|(\mathcal{B}(\y_{\lambda}+\y_e)-\mathcal{B}(\y_e),\y_\lambda)|  &\leq
		\frac{\mu }{2}\|\nabla\y_\lambda\|_{\H}^2 +\frac{\beta}{8}\||\y_e|^{\frac{r-1}{2}}\y_\lambda\|_{\H}^2 +\varrho_{\frac{1}{2}}\|\y_\lambda\|_{\H}^2, \label{e1} \\
		|\gamma(\mathcal{C}_2(\y_\lambda+\y_e)-\mathcal{C}_2(\y_e),\y_\lambda)| & \leq
		\frac{\beta}{4}\||\y_\lambda+\y_e|^{\frac{r-1}{2}}\y_\lambda\|_{\H}^2+\frac{\beta}{8}\||\y_e|^{\frac{r-1}{2}}\y_\lambda\|_{\H}^2+(\rho_{\frac{1}{2}}+\rho_{1})\|\y_\lambda\|_{\H}^2, \label{e2}
	\end{align}
	and
	\begin{align}\label{e3}
		\beta(\mathcal{C}_1(\y_{\lambda}+\y_e)-\mathcal{C}_1(\y_e),\y_\lambda)&\geq\frac{\beta}{4}
		\left(\||\y_\lambda+\y_e|^{\frac{r-1}{2}}\y_\lambda\|_{\H}^2+\||\y_e|^{\frac{r-1}{2}}\y_\lambda\|_{\H}^2\right)+\frac{\beta}{2^{r}}\|\y_\lambda\|_{\wi\L^{r+1}}^{r+1}.
	\end{align}
	From the monotonicity of $(\partial\I_{\mathcal{K}})_{\lambda}(\cdot)$, we have
	\begin{align*}
		-((\partial\I_{\mathcal{K}})_{\lambda}(\y_\lambda),\y_\lambda)\leq-((\partial\I_{\mathcal{K}})_{\lambda}(\boldsymbol{0}),\y_\lambda)\leq\|(\partial\I_{\mathcal{K}})_{\lambda}(\boldsymbol{0})\|_{\H}\|\y_\lambda\|_{\H}\leq \|(\partial\I_{\mathcal{K}})(\boldsymbol{0})\|_{\H}\|\y_\lambda\|_{\H}=0.
	\end{align*}
	Therefore, using \eqref{e1}-\eqref{e3}, and Cauchy-Schwarz and Young's inequalities in \eqref{ey}, we conclude
	\begin{align}\label{eyy}
		\left(\alpha+\kappa+\frac{1}{2}-\varrho_{\frac{1}{2}}-\rho_{\frac{1}{2}}-\rho_{1}\right)\|\y_\lambda\|_{\H}^2+\frac{\mu}{2}\|\nabla\y_\lambda\|_{\H}^2+\frac{\beta}{2^{r}}\|\y_\lambda\|_{\wi\L^{r+1}}^{r+1}\leq\frac{1}{2}\|\f\|_{\H}^2,
	\end{align}
	for any $\kappa\geq\varrho_{\frac{1}{2}}+\rho_{\frac{1}{2}}+\rho_{1}$. We rewrite \eqref{3.2.2} as
	\begin{align}\label{eAy}
		&(\alpha+\kappa+1)(\y_{\lambda}+\y_e)+\mu\mathrm{A}(\y_{\lambda}+\y_e)+\mathcal{B}(\y_{\lambda}+\y_e)+\beta\mathcal{C}_1(\y_{\lambda}+\y_e)+\gamma\mathcal{C}_2(\y_\lambda+\y_e)+(\partial\I_{\mathcal{K}})_{\lambda}(\y_{\lambda})\nonumber\\&=\f+\y_e+\mathscr{S}(\y_e).
	\end{align}
	We now take the inner product in \eqref{eAy} with $\A\y_\lambda=\A(\y_\lambda+\y_e)-\A\y_e$ and rearrange the terms to obtain 
	\begin{align}\label{3.2.5}
		&(\alpha+\kappa+1)\|\nabla(\y_\lambda+\y_e)\|_{\H}^{2}+\mu\|\mathrm{A}(\y_\lambda+\y_e)\|_{\H}^{2} +(\mathcal{B}(\y_{\lambda}+\y_e),\mathrm{A}(\y_\lambda+\y_e))\nonumber\\& +\beta(\mathcal{C}_1(\y_\lambda+\y_e),\A(\y_\lambda+\y_e))+\gamma(\mathcal{C}_2(\y_\lambda+\y_e),\mathrm{A}(\y_\lambda+\y_e))+((\partial\I_{\mathcal{K}})_{\lambda}(\y_{\lambda}),\A\y_\lambda)\nonumber\\&= (\f,\A\y_\lambda)+(\y_e+\mathscr{S}(\y_e),\A\y_\lambda)+(\alpha+\kappa+1)((\y_\lambda+\y_e),\A\y_e)+\mu(\A(\y_\lambda+\y_e),\A\y_e)\nonumber\\&\quad+(\mathcal{B}(\y_{\lambda}+\y_e),\mathrm{A}\y_e)+\gamma(\mathcal{C}_2(\y_\lambda+\y_e),\mathrm{A}\y_e)+\beta(\mathcal{C}_1(\y_\lambda+\y_e),\mathrm{A}\y_e).
	\end{align}	

	We consider the cases $r>3$ for $d=2,3$ and $d=r=3$ with $2\beta\mu>1$ separately.
	\vskip 2mm
	\noindent
	\textbf{Case I:} \emph{$r>3$ for $d=2,3$.} The following estimates are calculated in a similar fashion as  \eqref{Ayc1}-\eqref{Ayc2}:
	\begin{align}
		|(\B(\y_{\lambda}+\y_e),\A(\y_{\lambda}+\y_e))|&\leq
		\frac{\mu}{2}\|\A(\y_{\lambda}+\y_e)\|_{\H}^2+\frac{\beta}{4} \||\y_\lambda+\y_e|^{\frac{r-1}{2}}\nabla(\y_\lambda+\y_e)\|_{\H}^2\eta_1\|\nabla(\y_\lambda+\y_e)\|_{\H}^2,\label{3.2.7.0}\\
		|(\B(\y_{\lambda}+\y_e),\A\y_e)|&\leq
		\frac{\mu}{2}\|\A\y_e\|_{\H}^2+\frac{\beta}{16} \||\y_\lambda+\y_e|^{\frac{r-1}{2}}\nabla(\y_\lambda+\y_e)\|_{\H}^2+\varrho_{\frac{1}{4}}\|\nabla(\y_\lambda+\y_e)\|_{\H}^2, \\
		|\gamma(\mathcal{C}_2(\y_\lambda+\y_e),\mathrm{A}(\y_\lambda+\y_e))|&\leq
		\frac{\beta}{16}\||\y_\lambda+\y_e|^{\frac{r-1}{2}}\nabla(\y_\lambda+\y_e)\|_{\H}^2+\eta_2\|\nabla(\y_\lambda+\y_e)\|_{\H}^2,\label{3.2.7.3}
	\end{align}
	where  $\varrho_{\frac{1}{4}}:=\frac{r-3}{2\mu(r-1)}\left[\frac{16}{\beta\mu(r-1)}\right]^{\frac{2}{r-3}}.$ By using interpolation inequality and Remark \ref{C2}, we calculate
	\begin{align}\label{e5}
		|(\mathcal{C}_1(\y_{\lambda}+\y_e),\A\y_e)|&\leq\|\mathcal{C}_1(\y_{\lambda}+\y_e)\|_{\H}
		\|\A\y_e\|_{\H}\leq\|\y_\lambda+\y_e\|_{\wi\L^{2r}}^{r}\|\A\y_e\|_{\H}\nonumber\\&\leq
		\|\y_\lambda+\y_e\|_{\wi\L^{3(r+1)}}^{\frac{3(r-1)}{4}}\|\y_\lambda+\y_e\|_{\wi\L^{r+1}}^{\frac{r+3}{4}}\|\A\y_e\|_{\H}\nonumber\\&\leq C\||\y_\lambda+\y_e|^{\frac{r-1}{2}}\nabla(\y_\lambda+\y_e)\|_{\H}^{\frac{3(r-1)}{2(r+1)}}\|\y_\lambda+\y_e\|_{\wi\L^{r+1}}^{\frac{r+3}{4}}\|\A\y_e\|_{\H}\nonumber\\&\leq\frac{\beta}{16}\||\y_\lambda+\y_e|^{\frac{r-1}{2}}\nabla(\y_\lambda+\y_e)\|_{\H}^2+C\left(\|\y_\lambda+\y_e\|_{\wi\L^{r+1}}^{r+1}\right)^{\frac{r+3}{r+7}}.
	\end{align}
	Similarly, applying interpolation inequality for $2\leq2q\leq3(r+1)$, we find
	\begin{align}\label{e6}
		|(\mathcal{C}_2(\y_{\lambda}+\y_e),\A\y_e)|&\leq\|\mathcal{C}_2(\y_{\lambda}+\y_e)\|_{\H}
		\|\A\y_e\|_{\H}\leq\|\y_\lambda+\y_e\|_{\wi\L^{2q}}^{q}\|\A\y_e\|_{\H}\nonumber\\&\leq\frac{\beta}{16}\||\y_\lambda+\y_e|^{\frac{r-1}{2}}\nabla(\y_\lambda+\y_e)\|_{\H}^2+C\|\y_\lambda+\y_e\|_{\H}^{\frac{3(r+1)-2q}{3(r+1)-3q+1}}.
	\end{align}
	Moreover, from \cite[Proposition 1.1, part (iv), pp. 183]{VB1}, we infer that  
	\begin{align}\label{e7}
		((\partial\I_{\mathcal{K}})_{\lambda}(\y_{\lambda}),\mathrm{A}\y_{\lambda})\geq0.
	\end{align}
	Combining \eqref{3.2.7.0}-\eqref{e7} and \eqref{eyy}, and using the Cauchy-Schwarz and Young's  inequalities, we conclude from \eqref{3.2.5} that 
	\begin{align}\label{r3}
		&\left(\alpha+\kappa+1-\varrho_{\frac{1}{4}}-\eta_1-\eta_2\right)\|\nabla(\y_\lambda+\y_e)\|_{\H}^{2}+\frac{\mu}{2}\|\mathrm{A}(\y_\lambda+\y_e)\|_{\H}^{2} +\frac{\beta}{2}\||\y_\lambda+\y_e|^{\frac{r-1}{2}}\nabla(\y_\lambda+\y_e)\|_{\H}^2 \nonumber\\&\leq\frac{2}{\mu}\|\f\|_{\H}^2+\frac{2}{\mu}\|\mathscr{S}(\y_e)+\y_e\|_{\H}^2+\mu\|\A\y_e\|_{\H}^2\nonumber\\&\quad+C\left(\|\y_\lambda+\y_e\|_{\H}^2+\|\y_\lambda+\y_e\|_{\H}^{\frac{3(r+1)-2q}{3(r+1)-3q+1}}+\left(\|\y_\lambda+\y_e\|_{\wi\L^{r+1}}^{r+1}\right)^{\frac{r+3}{r+7}}\right)\leq C, 
	\end{align}	
	where $C=C(\mu,|\gamma|,\beta,\|\f\|_{\H},\|\A\y_e\|_{\H})>0$ is a constant.
	\vskip 2mm
	\noindent
	\textbf{Case II:} \emph{$d=r=3$ with $2\beta\mu>1.$} Let $0<\theta<1.$
	By using the Cauchy-Schwarz and Young's inequalities
	\begin{align}\label{3.2.8.1}
		|(\B(\y_{\lambda}+\y_e),\A(\y_{\lambda}+\y_e))|&\leq\||\y_{\lambda}+\y_e|\nabla(\y_{\lambda}+\y_e)\|_{\H}\|\A(\y_{\lambda}+\y_e)\|_{\H}\nonumber\\&\leq\theta\mu\|\A(\y_{\lambda}+\y_e)\|_{\H}^2+\frac{1}{4\theta\mu}\||\y_{\lambda}+\y_e|\nabla(\y_{\lambda}+\y_e)\|_{\H}^2.
	\end{align}
	Similarly, we can write
	\begin{align}\label{3.2.8.11}
		|(\mathcal{B}(\y_{\lambda}+\y_e),\mathrm{A}\y_e)|&\leq4\theta\mu\|\A\y_e\|_{\H}^2+\frac{1}{16\theta\mu}\||\y_{\lambda}+\y_e|\nabla(\y_{\lambda}+\y_e)\|_{\H}^2.
	\end{align}
	For $r=3$, we write \eqref{3} as
	\begin{align}
		(\mathcal{C}_1(\y_{\lambda}+\y_e),\mathrm{A}(\y_{\lambda}+\y_e))&=\||\y_{\lambda}+\y_e|\nabla (\y_{\lambda}+\y_e)\|_{\H}^{2}+\frac{1}{2}\|\nabla|(\y_{\lambda}+\y_e)|^2\|_{\H}^{2}.\label{3.2.8.1.1}
	\end{align}
	Similarly for $r=3$ from \eqref{3.2.7.3}, we have
	\begin{align}\label{3.2.8.1.2}
		|\gamma(\mathcal{C}_2(\y_\lambda+\y_e),\mathrm{A}(\y_\lambda+\y_e))|\leq\frac{1}{16\theta\mu} \|(\y_\lambda+\y_e)\nabla(\y_\lambda+\y_e)\|_{\H}^2+\wi\eta_1\|\nabla(\y_\lambda+\y_e)\|_{\H}^2,
	\end{align}
	where $\wi\eta_1:=\left(8\theta\mu|\gamma|q(q-1)\right)^{\frac{3-q}{q-1}}\left({\frac{3-q}{2}}\right).$ Also, from \eqref{e5}-\eqref{e6}, we have
	\begin{align*}
		|(\mathcal{C}_1(\y_{\lambda}+\y_e),\A\y_e)|&\leq\frac{1}{16\theta\mu}\||\y_\lambda+\y_e|\nabla(\y_\lambda+\y_e)\|_{\H}^2+C\|\y_\lambda+\y_e\|_{\wi\L^{4}}^{\frac{12}{5}},\\
		|(\mathcal{C}_2(\y_{\lambda}+\y_e),\A\y_e)|&\leq\frac{1}{16\theta\mu}\||\y_\lambda+\y_e|\nabla(\y_\lambda+\y_e)\|_{\H}^2+C\|\y_\lambda+\y_e\|_{\H}^{\frac{12-2q}{13-3q}}.
	\end{align*}
	and
	\begin{align}\label{3.2.8.2}
		|(\f+\mathscr{S}(\y_e)+\y_e,\A(\y_{\lambda}+\y_e))|&\leq\|\f+\mathscr{S}(\y_e)+\y_e\|_{\H}\|\A(\y_{\lambda}+\y_e)\|_{\H}\nonumber\\&\leq\frac{(1-\theta)\mu}{2}\|\A(\y_{\lambda}+\y_e)\|_{\H}^2+\frac{1}{2(1-\theta)\mu}\|\f+\mathscr{S}(\y_e)+\y_e\|_{\H}^2,
	\end{align}
	where $0<\theta<1$ is chosen the same as in \eqref{3.2.8.1}.
	Using  the estimates \eqref{3.2.8.1}-\eqref{3.2.8.2} in \eqref{3.2.5} and the fact that 
	$((\partial\I_{\mathcal{K}})_{\lambda}(\y_{\lambda}),\mathrm{A}\y_{\lambda})\geq0,$ we obtain 
	\begin{align}\label{3.2.8.3}
		&\frac{\mu(1-\theta)}{2}\|\A(\y_{\lambda}+\y_e)\|_{\H}^2+\left(\beta-\frac{1}{2\theta\mu}\right)\||(\y_{\lambda}+\y_e)|\nabla(\y_{\lambda}+\y_e)\|_{\H}^2+\frac{\beta}{2}\|\nabla|(\y_{\lambda}+\y_e)|^2\|_{\H}^{2}\nonumber\\&\quad+(\alpha+\kappa+1-\wi\eta_1)\|\nabla(\y_\lambda+\y_e)\|_{\H}^{2}\leq C,
	\end{align}
	where  $C=C(\mu,|\gamma|,\beta,\|\f\|_{\H},\|\A\y_e\|_{\H})>0$ is a constant. 
	
	As $\y_e\in\D(\A),$ from \eqref{r3} (for $r>3$ with $\kappa\geq\varrho_{\frac{1}{4}}+\eta_1+\eta_2$) and \eqref{3.2.8.3} (for $r=3$ with $\kappa\geq\widetilde{\eta}_1$ and $2\beta\mu>1$), we conclude that the sequence $\{\y_\lambda\}_{\lambda>0}$ is a bounded sequence in $\D(\A)$. For sufficiently large $\kappa>0$, using H\"older's and Agmon's inequalities (see \cite[Lemma 13.2]{Agm}), \eqref{r3} (for $r>3$) and \eqref{3.2.8.3} (for $r=3$), we calculate
	\begin{align}\label{3.2.9}
		\|\mathcal{B}(\y_\lambda+\y_e)\|_{\H}&\leq\|((\y_\lambda+\y_e)\cdot\nabla)(\y_\lambda+\y_e)\|_{\H}\leq\|\y_\lambda+\y_e\|_{\wi\L^{\infty}}\|\nabla(\y_\lambda+\y_e)\|_{\H}
		\nonumber\\&\leq\|\nabla(\y_\lambda+\y_e)\|_{\H}
		\|\y_\lambda+\y_e\|_{\H}^{1-\frac{d}{4}}\|\y_\lambda+\y_e\|_{\H^2_{\mathrm{p}}}^{\frac{d}{4}}\leq C.
	\end{align}
	Moreover, using interpolation inequality, \eqref{eyy}, \eqref{r3} (for $r>3$), \eqref{3.2.8.3} (for $r=3$) and Remark \ref{C2}, we estimate 
	\begin{align}\label{3.2.10}
		\|\mathcal{C}_1(\y_\lambda+\y_e)\|_{\H}\leq\|\y_\lambda+\y_e\|_{\widetilde{\L}^{2r}}^{r}\leq
		\|\y_\lambda+\y_e\|_{\widetilde{\L}^{r+1}}^{\frac{r+3}{4}}\|\y_\lambda+\y_e\|_{\widetilde{\L}^{3(r+1)}}^{\frac{3(r-1)}{4}}\leq C. 
	\end{align}
	Similarly, using interpolation inequality with $2\leq2q\leq3(r+1)$, one can obtain
	\begin{align}\label{3.2.10.0}
		\|\mathcal{C}_2(\y_\lambda+\y_e)\|_{\H}\leq\|\y_\lambda+\y_e\|_{\widetilde{\L}^{2q}}^{q}&\leq
		\|\y_\lambda+\y_e\|_{\H}^{\frac{3(r+1)-2q}{3(r+1)-2}} \|\y_\lambda+\y_e\|_{\widetilde{\L}^{3(r+1)}}^{\frac{3(r+1)(q-1)}{3(r+1)-2}}\leq C.
	\end{align}
	Let us rewrite \eqref{3.2.2} as 
	\begin{equation}\label{3.2.11}
		\y_{\lambda}+\mathcal{S}(\y_{\lambda})+(\partial\I_{\mathcal{K}})_{\lambda}(\y_{\lambda})=\f.
	\end{equation} 
	Using \eqref{r3} (for $r>3$), \eqref{3.2.8.3} (for $r=3$), \eqref{3.2.9}-\eqref{3.2.10.0} in \eqref{3.2.11}, we deduce 
	\begin{align}\label{3.2.12}
		\|\mathcal{S}(\y_{\lambda})\|_{\H}\leq C \  \mbox{ and } \ \|(\partial\I_{\mathcal{K}})_{\lambda}(\y_{\lambda})\|_{\H}\leq C,
	\end{align}
	where $C=C(\mu,|\gamma|, \beta,\|\f\|_{\H},\|\A\y_e\|_{\H})>0$ is a constant.
	\vskip 2mm
	\noindent
	\emph{Passing limit as $\lambda\to0.$} By making the use of the Banach-Alaoglu theorem (see \cite{semi}) and denoting the subsequences again by $\y_\lambda$, we have the following convergences from \eqref{r3} (for $r>3$), \eqref{3.2.8.3} (for $r=3$) and \eqref{3.2.12}:
	
	\begin{equation}\label{3.2.13}
		\left\{
		\begin{aligned}
			\y_{\lambda}&\xrightharpoonup{w} \y, \ \text{ in } \  \V, \\
			\A\y_{\lambda}&\xrightharpoonup{w} \A\y, \ \text{ in }\ \H , 
		\end{aligned}	
		\right. \ \
		\left\{
		\begin{aligned}
			(\partial\I_{\mathcal{K}})_{\lambda}(\y_{\lambda})&\xrightharpoonup{w} \f_{1}, \ \text{ in }\ \H ,  \\ 
			\mathcal{S}(\y_{\lambda})&\xrightharpoonup{w} \f_{2}, \ \text{ in }\ \H .
		\end{aligned}	
		\right.
	\end{equation}
	Moreover, by using the compact embedding $\D(\A)\hookrightarrow\V$, we have $\y_\lambda\to\y$  in  $\V.$
	
	Finally, we pass the weak limit in \eqref{3.2.11} as $\lambda\to0$ to get $\y+\f_1+\f_2=\f$ in $\H$.
	Let us write \eqref{3.2.11} for $\lambda$ and $\hat{\lambda}$, subtract and then take the inner product with $\y_\lambda-\y_{\hat{\lambda}}$ to find 
	\begin{align}\label{3.2.14}
		((\partial\I_{\mathcal{K}})_\lambda(\y_{\lambda})-(\partial\I_{\mathcal{K}})_{\hat{\lambda}}(\y_{\hat{\lambda}}),\y_{\lambda}-\y_{\hat{\lambda}})+
		((\mathcal{S}+\mathrm{I})(\y_{\lambda})-(\mathcal{S}+\mathrm{I})(\y_{\hat{\lambda}}),\y_{\lambda}-\y_{\hat{\lambda}})=0,
	\end{align}
	for all $\lambda, \hat{\lambda}>0.$ By the monotonicity of $\wi{\mathcal{F}}(\cdot)+\I$, we obtain
	\begin{align}\label{v1}
		((\partial\I_{\mathcal{K}})_\lambda(\y_{\lambda})-(\partial\I_{\mathcal{K}})_{\hat{\lambda}}(\y_{\hat{\lambda}}),\y_{\lambda}-\y_{\hat{\lambda}})\leq 0.	
	\end{align}
	Then by an application of \cite[Proposition 1.3, part (iv), pp. 49]{VB2}, we conclude from \eqref{3.2.14}-\eqref{v1} that  
	\begin{align}\label{v2}
		(\y,\f_{1})\in\partial\I_{\mathcal{K}} \  \text{and} \ 	\lim\limits_{\lambda,{\hat{\lambda}}\to 0} ((\partial\I_{\mathcal{K}})_\lambda(\y_{\lambda})-(\partial\I_{\mathcal{K}})_{\hat{\lambda}}(\y_{\hat{\lambda}}),\y_{\lambda}-\y_{\hat{\lambda}})=0.
	\end{align}
	Thus from \eqref{3.2.14}, we deduce 
	\begin{align}\label{3.2.15}
		\lim\limits_{\lambda,{\hat{\lambda}}\to 0} 	((\mathcal{S}+\mathrm{I})(\y_{\lambda})-(\mathcal{S}+\mathrm{I})(\y_{\hat{\lambda}}),\y_{\lambda}-\y_{\hat{\lambda}})=0.	
	\end{align}
	Since the operator $\mathcal{S}(\cdot)+\I$ is maximal monotone,  \eqref{3.2.13}, \eqref{3.2.15} and \cite[Lemma 1.3, pp. 49]{VB2} yield
	\begin{align}\label{3.2.16}
		(\y,\y+\f_{2})\in\mathcal{S}+\mathrm{I},
	\end{align}
	and it gives $\mathcal{S}(\y)=\f_{2}$ as $\mathcal{S}+\mathrm{I}$ is a single-valued operator.  This proves the surjectivity condition \eqref{3.2.1}.
	
	Now our aim is to show that $\mathrm{D}(\mathfrak{G}) = \mathrm{D}(\mathrm{A})\cap\mathrm{D} (\partial\I_{\mathcal{K}}).$ 
	Note that  $\mathfrak{G}=\mathcal{S}+\partial\I_{\mathcal{K}}$ implies   $\mathrm{D}(\mathrm{A})\cap\mathrm{D}(\partial\I_{\mathcal{K}})=\mathrm{D}(\mathrm{S})\cap\mathrm{D}(\partial\I_{\mathcal{K}})\subseteq\mathrm{D}(\mathfrak{G})$. It follows from the surjectivity condition \eqref{3.2.1} that for any $\f\in\H$ there exists a element $\y\in\D(\mathfrak{G})$ such that $\y+\mathfrak{G}(\y)\ni\f$. As we have seen in the previous step that $\y$ can be approximated by a sequence of solutions $\{\y_{\lambda}\}_{\lambda>0}$ of equation \eqref{3.2.2}. From \eqref{v2} and \eqref{3.2.16}, we infer that $\y\in\mathrm{D}(\mathcal{S})\cap\mathrm{D}(\partial\I_{\mathcal{K}})=\mathrm{D}(\mathrm{A}) \cap\mathrm{D} (\partial\I_{\mathcal{K}})$ and hence $\mathrm{D}(\mathfrak{G}) = \mathrm{D}(\mathrm{A})\cap\mathrm{D} (\partial\I_{\mathcal{K}}).$ 
\end{proof}


	\section{Stationary problem}\label{SP}
	\numberwithin{equation}{section}
Let us consider the abstract formulation of \eqref{stability}, by taking the Leray projection $\mathcal{P}$, which is given by
\begin{align}\label{1ee}
	\mu\A\y_e+\mathcal{B}(\y_e)+\alpha\y_e+\beta\mathcal{C}_1(\y_e)+\gamma\mathcal{C}_2(\y_e)= \f_e \  \text{ in } \ \mathbb{T}^d,
\end{align}
where $\f_e:=\mathcal{P}\g_e\in\H$ is an external forcing. 
\vskip 2mm
\noindent

\begin{theorem}\label{STT}
For every $\f_{e}\in\H$,	there exists at least one strong  solution for the stationary problem \eqref{1e}. For
 \begin{align*}
	\min\left\{\mu,\alpha\right\}\geq2\mathfrak{K}_2+C\left(\frac{1}{\beta\mu}\|\f_e\|_{\H}^2+\frac{\mathfrak{K}_1}{\beta}|\mathbb{T}^d|\right)^{\frac{1}{2}},
\end{align*}
where  $\mathfrak{K}_1:=|\gamma|^{\frac{r+1}{r-q}}\left(\frac{2(q+1)}{\beta(r+1)}\right)^{\frac{r-q}{q+1}}\left(\frac{r-q}{r+1}\right)$ and $\mathfrak{K}_2:=\left(\frac{2^{q-1}q|\gamma|(q-1)}{\beta(r-1)}\right)^{\frac{q-1}{r-q}}\left(\frac{r-q}{r-1}\right)$, the solution is unique. 
\end{theorem}
\begin{proof} 
\textbf{\emph{Existence}:} 
From Proposition \ref{prop3.1}, it is straightforward that the operator $\hat{\mathcal{M}}(\cdot):=\mu\A+\mathcal{B}(\cdot)+\alpha\I+\beta\mathcal{C}_1(\cdot)+\gamma\mathcal{C}_2(\cdot)$ is $m$-accretive in $\H$ for $r>3$ in $d=2,3$ and $d=r=3$ with $2\beta\mu>1$. Thus, from \cite{VB1}, we have
\begin{align*}
	\mathrm{R}(\hat{\mathcal{M}}(\cdot)+\I)=\H,
\end{align*}
which says that for every $\f_e\in\H$, there exists a $\y_{e}\in\D(\A)=\D(\hat{\mathcal{M}})$ (Step IV in the proof of Proposition \ref{prop3.1}) such that \eqref{1ee} is satisfied. Therefore, $\y_{e}\in\D(\A)$ is a  strong solution to \eqref{1ee}. We now derive some stationary energy estimates for the solution $\y_e$ of equation \eqref{1ee}. Taking the inner product with $\y_e$ in \eqref{1ee}, and then using H\"older's and Young's inequalities, we obtain
\begin{align}\label{eg1}
	\mu	\|\nabla\y_e\|_{\H}^2+\frac{\alpha}{2}\|\y_e\|_{\H}^2+\frac{\beta}{2}\|\y_e\|_{\wi{\L}^{r+1}}^{r+1}\leq\frac{1}{2\alpha}\|\f_e\|_{\H}^2+\mathfrak{K}_1|\mathbb{T}^d|,
\end{align}
which implies that
\begin{align}\label{eg}
	\min\left\{\mu, \frac{\alpha}{2}\right\}\|\y_e\|_{\H_{\mathrm{p}}^1}^2+\frac{\beta}{2}\|\y_e\|_{\wi{\L}^{r+1}}^{r+1}\leq\frac{1}{2\mu}\|\f_e\|_{\H}^2+\mathfrak{K}_1|\mathbb{T}^d|,
\end{align}
where the constant $\mathfrak{K}_1:=|\gamma|^{\frac{r+1}{r-q}}\left(\frac{2(q+1)}{\beta(r+1)}\right)^{\frac{r-q}{q+1}}\left(\frac{r-q}{r+1}\right)$. 
\vskip 2mm
\noindent
\textbf{\emph{Uniqueness}:} 
Let us now establish the uniqueness of \eqref{1ee}. Let $\y_e, \z_e\in\D(\A)$ be two steady state solutions of \eqref{1ee} and define $\w_e:=\y_e-\z_e$. Then, from \eqref{1ee}, we write
\begin{align}\label{engs1}
	\mu\A\w_e+\mathcal{B}(\y_e)-\mathcal{B}(\z_e)+\alpha\w_e+\beta(\mathcal{C}_1(\y_e)-\mathcal{C}_1(\z_e))+\gamma(\mathcal{C}_2(\y_e)-\mathcal{C}_2(\z_e))=\boldsymbol{0}.
\end{align}
By taking the inner product with $\w_e$ in \eqref{engs1} and using \eqref{C1}, we obtain
\begin{align}\label{engs2}
	\mu\|\nabla\w_e\|_{\H}^2+\alpha\|\w_e\|_{\H}^2+\beta(\mathcal{C}_1(\y_e)-\mathcal{C}_1(\z_e),\w_e)\leq-(\mathcal{B}(\y_e)-\mathcal{B}(\z_e),\w_e)-
	\gamma(\mathcal{C}_2(\y_e)-\mathcal{C}_2(\z_e),\w_e).
\end{align}
From Proposition \ref{prop3.1} (see \eqref{3.5} and \eqref{c2.2}), we have  the following estimates:
\begin{align}\label{engs3}
	|\gamma(\mathcal{C}_2(\y_e)-\mathcal{C}_2(\z_e),\w_e)|&\leq
	\frac{\beta}{2}\||\y_e|^{\frac{r-1}{2}}\w_e\|_{\H}^2+\frac{\beta}{2}\||\z_e|^{\frac{r-1}{2}}\w_e\|_{\H}^2+2\mathfrak{K}_5\|\w_e\|_{\H}^2,\nonumber\\
	\beta(\mathcal{C}_1(\y_e)-\mathcal{C}_1(\z_e),\w_e)&\geq\frac{\beta}{2}
	\left(\||\y_e|^{\frac{r-1}{2}}\w_e\|_{\H}^2+\||\z_e|^{\frac{r-1}{2}}\w_e\|_{\H}^2\right),
\end{align}
where $\mathfrak{K}_2:=\left(\frac{2^{q-1}q|\gamma|(q-1)}{\beta(r-1)}\right)^{\frac{q-1}{r-q}}\left(\frac{r-q}{r-1}\right)$. 
We calculate the bilinear term by using Gagliardo-Nirenberg interpolation inequality as
\begin{align}\label{engs4}
	|(\mathcal{B}(\y_e)-\mathcal{B}(\z_e),\w_e)|=|(\mathcal{B}(\w_e,\z_e),\w_e)|
	\leq\|\w_e\|_{\wi\L^4}^2\|\nabla\z_e\|_{\H}\leq C\|\w_e\|_{\V}^2\|\nabla\z_e\|_{\H}.
\end{align}
Using \eqref{engs3}-\eqref{engs4} in \eqref{engs2}, we find
\begin{align*}
	&\min\left\{\mu,\alpha\right\}\|\w_e\|_{\V}^2\leq2\mathfrak{K}_2\|\w_e\|_{\V}^2+C\|\w_e\|_{\V}^2\|\nabla\z_e\|_{\H},
\end{align*}
and it implies that
\begin{align}\label{engs5}
	\left(\min\left\{\mu,\alpha\right\}-2\mathfrak{K}_2-C\|\nabla\z_e\|_{\H}\right)\|\w_e\|_{\V}^2\leq0.
\end{align}
From \eqref{eg1}, we calculate
\begin{align*}
	\|\nabla\z_e\|_{\H}\leq\left(\frac{1}{2\alpha\mu}\|\f_e\|_{\H}^2+\frac{\mathfrak{K}_1}{\mu}|\mathbb{T}^d|\right)^{\frac{1}{2}}.
\end{align*}
Thus, from \eqref{engs5},  we conclude that $\w_e=0$ and hence $\y_e=\z_e$. This completes the proof of uniqueness.
\end{proof}
\begin{remark} (1) We can also find
	\begin{align*}
		|(\mathcal{B}(\y_e)-\mathcal{B}(\z_e),\w_e)|&\leq\mu\|\nabla\w_e\|_{\H}^2+\frac{1}{4\mu}\int_{\mathbb{T}^d}|\z_e(x)|^2|\y_e(x)-\z_e(x)|^2\d x \nonumber\\&= \mu\|\nabla\w_e\|_{\H}^2+\frac{1}{4\mu} \int_{\mathbb{T}^d}|\w_e(x)|^2\left(|\z_e(x)|^{r-1}+1\right)\frac{|\z_e(x)|^2}{|\z_e(x)|^{r-1}+1}\d x\nonumber\\&\leq\mu\|\nabla\w_e\|_{\H}^2+ \frac{1}{4\mu}\int_{\mathbb{T}^d}|\z_e(x)|^{r-1}|\w_e(x)|^2\d x
		+\frac{1}{4\mu}\int_{\mathbb{T}^d}|\w_e(x)|^2\d x,
	\end{align*}
	where we used the fact that $\left\|\frac{|\z_e|^2}{|\z_e|^{r-1}+1}\right\|_{\widetilde{\L}^{\infty}}<1$ for $r\geq 3$. 
	By using Taylor's formula, H\"older's and Young's inequality, we have
	\begin{align*}
		|\gamma(\mathcal{C}_2(\y_e)-\mathcal{C}_2(\z_e),\w_e)|&\leq2^{q-2}q|\gamma|\int_{\mathbb{T}^d} (|\y_e(x)|^{q-1}+|\z_e(x)|^{q-1})|\w_e(x)|^2 \d x\nonumber\\&\leq
		\frac{1}{2}\left(\frac{\beta}{2}-\frac{1}{4\mu}\right)\||\z_e|^{\frac{r-1}{2}}\w_e\|_{\H}^2+\mathfrak{K}_3\|\w_e\|_{\H}^2+\frac{1}{4\mu}\||\y_e|^{\frac{r-1}{2}}\w_e\|_{\H}^2+\mathfrak{K}_4\|\w_e\|_{\H}^2,
	\end{align*}
	where $\mathfrak{K}_3:=\left(2^{q-2}q|\gamma|\right)^{\frac{r-1}{r-q}}\left(\frac{r-q}{r-1}\right)\left(\frac{2(q-1)}{(r-1)\left(\frac{\beta}{2}-\frac{1}{4\mu}\right)}\right)^{\frac{q-1}{r-q}}$ and $\mathfrak{K}_4:=\left(2^{q-2}q|\gamma|\right)^{\frac{r-1}{r-q}}\left(\frac{r-q}{r-1}\right)\left(\frac{4\mu(q-1)}{r-1}\right)^{\frac{q-1}{r-q}}$. Thus, from \eqref{engs2}, we obtain
	\begin{align*}
		(\alpha-\mathfrak{K}_3-\mathfrak{K}_4)\|\w_e\|_{\H}^2+\frac{1}{2}\left(\frac{\beta}{2}-\frac{1}{4\mu}\right)\||\z_e|^{\frac{r-1}{2}}\w_e\|_{\H}^2+\left(\frac{\beta}{2}-\frac{1}{4\mu}\right)\||\y_e|^{\frac{r-1}{2}}\w_e\|_{\H}^2\leq0.
	\end{align*}
	Therefore, for $\alpha>\mathfrak{K}_3+\mathfrak{K}_4$ and $\beta>\frac{1}{2\mu}$, we obtain $\y_e=\z_e$.
	
	\vskip 2mm
	\noindent 
	(2) We further calculate
	\begin{align*}
		|(\mathcal{B}(\y_e)-\mathcal{B}(\z_e),\w_e)|\leq
		\frac{\mu}{2}\|\nabla\w_e\|_{\H}^2 +\frac{\beta}{4}\||\z_e|^{\frac{r-1}{2}}\w_e\|_{\H}^2 +\varrho_1\|\w_e\|_{\H}^2,
	\end{align*}
	where $\varrho_1:=\frac{r-3}{2\mu(r-1)}\left[\frac{4}{\beta\mu (r-1)}\right]^{\frac{2}{r-3}}.$ Furthermore, by using H\"older's and Young's inequalities, we estimate 
	\begin{align*}
		\int_{\mathbb{T}^d} |\y_e(x)|^{q-1}|\w_e(x)|^2 \d x&=\int_{\mathbb{T}^d} |\y_e(x)|^{q-1}|\w_e(x)|^{2\left(\frac{q-1}{r-1}\right)}|\w_e(x)|^{2\left(1-\frac{q-1}{r-1}\right)}\d x \nonumber\\&\leq 
		\frac{\beta}{2^{q-1}q|\gamma|}\||\y_e|^{\frac{r-1}{2}}\w_e\|_{\H}^2+
		\rho_2\|\w_e\|_{\H}^2, 
	\end{align*}
	where $\rho_2:=\left(\frac{2^{q-1}q|\gamma|(q-1)}{\beta(r-1)}\right)^{\frac{q-1}{r-q}}\left(\frac{r-q}{r-1}\right)$ . Similarly, one can obtain 
	\begin{align*}
		\int_{\mathbb{T}^d}|\z_e(x)|^{q-1}|\w_e(x)|^2 \d x\leq
		\frac{\beta}{2^{q}q|\gamma|}\||\z_e|^{\frac{r-1}{2}}\w_e\|_{\H}^2+\rho_1\|\w_e\|_{\H}^2,
	\end{align*}
	where $\rho_1:=\left(\frac{2^{q}q|\gamma|(q-1)}{\beta(r-1)}\right)^{\frac{q-1}{r-q}}\left(\frac{r-q}{r-1}\right)$. Using the above results, we find   
	\begin{align*}
		|\gamma\langle\mathcal{C}_2(\y_e)-\mathcal{C}_2(\z_e),\y_e-\z_e\rangle|&\leq2^{q-2}q|\gamma|\int_{\mathbb{T}^d} (|\y_e(x)|^{q-1}+|\z_e(x)|^{q-1})|\w_e(x)|^2 \d x
		\nonumber\\&\leq
		\frac{\beta}{2}\||\y_e|^{\frac{r-1}{2}}\w_e\|_{\H}^2+\frac{\beta}{4}\||\z_e|^{\frac{r-1}{2}}\w_e\|_{\H}^2+2^{q-2}q|\gamma|(\rho_1+\rho_2)\|\w_e\|_{\H}^2.	
	\end{align*}
	Thus from \eqref{engs2}, we have
	\begin{align*}
		&\frac{\mu}{2}\|\nabla\w_e(x)\|_{\H}^2+\alpha\|\w_e(x)\|_{\H}^2\leq2^{q-2}q|\gamma|(\rho_1+\rho_2)\|\w_e\|_{\H}^2+\varrho_1\|\w_e\|_{\H}^2.	
	\end{align*}
	It implies that
	\begin{align*}
		&\left(\min\left\{\frac{\mu}{2},\alpha\right\}-2^{q-2}q|\gamma|(\rho_1+\rho_2)-\varrho_1\right)\|\w_e\|_{\V}^2\leq0.
	\end{align*}
	Therefore  for $\min\left\{\frac{\mu}{2},\alpha\right\}>2^{q-2}q|\gamma|(\rho_1+\rho_2)+\varrho_1$, we get the uniqueness $\y_e=\z_e$.
	
\end{remark}

\begin{remark}
	In all of the above cases,  we observe that the linear damping term $\alpha\y$ is helping us to obtain the existence and uniqueness result. For the non-zero average condition in the torus, without this term, the uniqueness of the stationary CBFeD model is an open problem.   Thus the linear damping term  plays a crucial role in order to obtain the solvability of the stationary system as well as the stability results. 
\end{remark}

	\section{Proportional Controllers}\label{PC}
	\numberwithin{equation}{section}
We are now going to discuss the stabilization via \emph{proportional controllers} where the control is localized in an open subset $\Omega$ of a torus having smooth boundary $\partial\Omega$. By a \emph{proportional control}, we mean a linear feedback control $\u$ which is proportional to the difference between the desired state (eqilibrium solution $\y_e$) and actual state (solution $\y$). Consider the following controlled CBFeD system:
\begin{equation}\label{prop1.1}
	\left\{
	\begin{aligned}
		\frac{\partial\y}{\partial t}-\mu\Delta\y+(\y\cdot\nabla)\y+\alpha\y+\beta|\y|^{r-1}\y +\gamma|\y|^{q-1}\y+\nabla p&=\f_e+m\u(t), \  &&\text{ in } \ \mathbb{T}^d\times(0,\infty), \\ \nabla\cdot\y&=0, \  &&\text{ in } \ \mathbb{T}^d\times(0,\infty), \\
		\y(0)&=\y_0, \ && \text{ in } \ \mathbb{T}^d,
	\end{aligned}
	\right.
\end{equation}	
where $\f_e\in\L^2(\mathbb{T}^d)$, $\y_0\in\H$ and $m$ is the characteristic function on $\Omega.$ Moreover, here $\u=-k(\y-\y_e), \ k\in\N$ is a proportional control having support in $\Omega$. Let $(\y_e,p_e)\in\H_{\mathrm{p}}^2(\mathbb{T}^d)\times (\mathrm{H}^1_{\mathrm{p}}(\mathbb{T}^d)\cap\mathrm{L}^2_0(\mathbb{T}^d))$ be the the steady state solution of the system \eqref{prop1.1} satisfying 
\begin{equation}\label{prop1.2}
	-\mu \Delta\y_e+(\y\cdot\nabla)\y_e+\alpha\y_e+\beta|\y_e|^{r-1}\y_e +\gamma|\y_e|^{q-1}\y_e+\nabla p_e=\f_e, \ \text{ in } \ \mathbb{T}^d. 
\end{equation}
Let $\wi{\mathrm{Q}}:=\mathbb{T}^d\setminus\ \overline{\Omega}$ and define  $\wi{\mathcal{V}}:=\{\y\in\C_{\mathrm{p}}^{\infty}(\wi\Q;\R^d):\nabla\cdot\y=0\}.$ We also define the Hilbert spaces $\wi\H$ and $\wi\V$ as the closures of $\wi{\mathcal{V}}$ in the Lebesgue space $\mathrm{L}^2(\wi\Q;\R^d)$ and the Sobolev space $\mathrm{H}^1(\wi\Q;\R^d)$, respectively. We identify the dual of $\wi\H$ with itself and we define the inner product $(\cdot,\cdot)_{\wi{\mathrm{Q}}}$ in $\wi\H$ as below
\begin{align*}
	(\y,\z)_{\wi{\mathrm{Q}}}:=\int_{\wi{\mathrm{Q}}} \y(x)\z(x)\d x.
\end{align*}  
Let us now define the operator $\mathcal{A}:=-\mu\mathcal{P}\Delta\y+\alpha\y$ on $\wi{\mathrm{Q}}$ with the domain $\D(\mathcal{A}):=\mathrm{H}^2(\wi{\mathrm{Q}};\R^d)\cap\wi\V.$ Let us denote the first eigenvalue of an operator $\mathcal{A}$ by $\lambda_1^*(\wi{\mathrm{Q}})$. Then using the concept of Rayleigh quotient, we know that 
\begin{align}\label{prop1.3}
	\lambda_1^*(\wi{\mathrm{Q}}):=\inf\{\mu\|\nabla\wi\y\|_{\wi\H}^2+\alpha\|\wi\y\|_{\wi\H}^2: \ \wi\y\in\wi\V, \ \|\wi\y\|_{\wi\H}=1\}.
\end{align}
It is clear that $\lambda_1^*(\wi{\mathrm{Q}})>0$. We now give an auxiliary  result which will be used in the sequel.
\begin{proposition}\label{stbl3}
	Let us define the operator  
	\begin{align*}
		\A_k\y:=\mu\A\y+\alpha\y+k\mathcal{P}(m\y), \ \text{ for all }  \ \y\in\V,
	\end{align*}
	with $\D(\A_k)=\D(\A)$, where $k$ is any positive real number. Then for each $\varepsilon>0,$ there exists $k_1=k_1(\varepsilon)$ such that 
	\begin{align*}
		(\mu\A\y+\alpha\y+k\mathcal{P}(m\y),\y)\geq(\mu\lambda_1^*(\wi{\mathrm{Q}})-\varepsilon)\|\y\|_{\H}^2,
	\end{align*}
	for all $\y\in\V$ and $k\geq k_1$, where $\lambda_1^*(\wi{\mathrm{Q}})$ is given by \eqref{prop1.3}.
\end{proposition}
\begin{remark}
	Let $\nu_k^1$ be the first eigenvalue of $\A_k$. Then, we have 
		\begin{align}\label{prop1.4}
			\nu_k^1:=\inf\left\{\mu\|\nabla\y\|_{\H}^2+\alpha\|\y\|_{\H}^2+k\int_{\Omega} |\y(x)|^2 \d x; \ \y\in\V, \ \|\y\|_{\H}=1\right\}.
		\end{align}
		 Let us define the following subsets of $\mathbb{R}$: \begin{align*}
		 	\mathrm{G_1}&=\{\mu\|\nabla\wi\y\|_{\wi\H}^2+\alpha\|\wi\y\|_{\wi\H}^2:\wi\y\in\wi\V, \ \|\wi\y\|_{\wi\H}=1\}\ \text{ and }\\
		 	\mathrm{G_2}&=\left\{\mu\|\nabla\y\|_{\H}^2+\alpha\|\y\|_{\H}^2+k\int_{\Omega} |\y(x)|^2 \d x : \ \y\in\V, \ \|\y\|_{\H}=1\right\}.
		 \end{align*}  Let us consider $\wi\y\in\wi\V$ and extend by zero across the boundary $\partial\Omega$ to $\V$. That is, we define 
	 	\begin{align*}
	 	\y=
	 	\begin{cases}
	 		\wi\y \  &\text{ in } \ \wi{\mathrm{Q}},\\
	 		\boldsymbol{0} &\text{ in }\  \Omega.
	 	\end{cases}
	 \end{align*} 
 Therefore, it is immediate that $\mathrm{G_1}\subset\mathrm{G_2},$ since  $\|\y\|_{\H}=\|\wi\y\|_{\wi\H}=1$ and 
		\begin{align*}
	\mu\|\nabla\y\|_{\H}^2+\alpha\|\y\|_{\H}^2+k\int_{\Omega}|\y(x)|^2\d x&=	\mu\|\nabla\wi\y\|_{\wi\H}^2+\alpha\|\wi\y\|_{\wi\H}^2.
		\end{align*}
		Hence, we conclude 
		\begin{align}\label{prop1.5}
			\nu_k^1\leq\lambda_1^*(\wi{\mathrm{Q}}).
	\end{align}
\end{remark}
\vskip 2mm
\noindent

\begin{proof}[Proof of Proposition \ref{stbl3}]
	Let $\y_k^1\in\D(\A)$ be an  eigenfunction of the operator $\A_k$ corresponding to the eigenvalue $\nu_k^1.$  Then we write
	\begin{align}\label{prop1.6}
		\mu\A\y_k^1+\alpha\y_k^1+k\mathcal{P}(m\y_k^1)=\nu_k^1\y_k^1.
	\end{align}
	We may assume that the $\y_k^1$ is the normalized eigenfunction, that is, $\|\y_k^1\|_{\H}=1.$ Then using \eqref{prop1.5}, we get
	\begin{align}\label{prop1.7}
		\mu\|\nabla\y_k^1\|_{\H}^2+\alpha\|\y_k^1\|_{\H}^2+k\int_{\Omega}|\y_k^1(x)|^2 \d x= \nu_k^1\leq\lambda_1^*(\wi{\mathrm{Q}}).
	\end{align}
	By using the Banach-Alaoglu theorem, we deduce  from above that 
		\begin{align}\label{conv}
			\begin{cases}
				\y_k^1&\xrightharpoonup{w}\y^1 \   \text{ in } \ \V,\\
				\y_k^1&\to\y^1 \  \text{ in } \ \H,
			\end{cases}
	\end{align}
	where the final convergence in \eqref{conv} is along a subsequence  (denoted by the same notation) and due to the compact embedding $\V\hookrightarrow\hookrightarrow\H$. Moreover, we have  a convergent subsequence of positive reals $\nu_k^1\to\nu^*$. Thus from \eqref{prop1.7}, we have 
		\begin{align*}
			\int_{\Omega}|\y_k^1(x)|^2 \d x\to0  \ \text{ as } \ k\to\infty.
	\end{align*}
	By using the fact that $\|\cdot\|_{\wi\V}\leq\|\cdot\|_{\V}$, we get a subsequence $\{\y_{k_l}^1\}_{l\in\N}\subset\{\y_k^1\}_{k\in\N}$ such that 
	\begin{align}\label{conv1}
		\y_{k_l}^1\xrightharpoonup{w}\y^1_* \   \text{ in } \ \wi\V,
	\end{align}
	as $l\to\infty$. Since $\{\y_{k_l}^1\}_{l\in\N}$ is a subsequence of $\{\y_k^1\}_{k\in\N}$, therefore from \eqref{conv}, we have
		\begin{align}\label{conv2}
		\y_{k_l}^1\xrightharpoonup{w}\y^1 \   \text{ in } \ \V,
	   \end{align}
The convergences \eqref{conv1},\eqref{conv2} and the fact $\|\cdot\|_{\wi\V}\leq\|\cdot\|_{\V}$ yield $\y^1_*=\y^1$ in $\wi\V$. From the strong convergence in $\H$, we conclude that $\|\y^1\|_{\H}=1$. Now, we multiply \eqref{prop1.6} by $\psi\in\wi\V$ and integrate over $\wi{\mathrm{Q}}$ to get 
	\begin{align*}
		\mu\int_{\wi{\mathrm{Q}}} \nabla\y_k^1(x):\nabla\psi(x)\d x +\alpha\int_{\wi{\mathrm{Q}}} \y_k^1(x)\psi(x)\d x= \nu_k^1\int_{\wi{\mathrm{Q}}}\y^1_k(x)\psi(x)\d x,
	\end{align*}
	where we have used the fact that $m(x)=0$ for $x\in\wi{\mathrm{Q}}$. Using the convergences \eqref{conv}-\eqref{conv1}, we get $\mathcal{A}\y^1=\nu^*\y^1$ in $\wi\V'$, so that $\nu^*$ is an eigenvalue of $\mathcal{A}$. But from \eqref{prop1.5}, we have  $\nu^*\leq\lambda_1^*(\wi{\mathrm{Q}}),$ and $\lambda_1^*(\wi{\mathrm{Q}})$ is the first eigen value of $\mathcal{A}$. Therefore, it is immediate that $\nu^*=\lambda_1^*(\wi{\mathrm{Q}}),$  and 
	\begin{align*}
		\lim\limits_{k\to\infty}\nu_k^1=\lambda_1^*(\wi{\mathrm{Q}}). 
	\end{align*}
	This also says that for each $\varepsilon>0$, there exists a $k_1=k_1(\varepsilon)>0$ such that 
	\begin{align*}
		|\nu_k^1-\lambda_1^*(\wi{\mathrm{Q}})|<\varepsilon \  \text{ for } \ k\geq k_1.
	\end{align*}
	Thus by \eqref{prop1.4} and using the definition of $\A_k$, we obtain 
	\begin{align*}
		(\mu\A\y+\alpha\y+k\mathcal{P}(m\y),\y)\geq\nu^1_k\|\y\|_{\H}^2\geq(\lambda_1^*(\wi{\mathrm{Q}})-\varepsilon)\|\y\|_{\H}^2,
	\end{align*}
	which  completes the proof. 
\end{proof}

\begin{remark}\label{rk5}
	\textbf{1.} It is known from the classical Rayleigh-Faber-Krahn perimetric inequality (see \cite[Section 5.4]{VB7}) in dimension $d\geq2$ that
	\begin{align}\label{RFK}
		\lambda_1^*(\wi{\mathrm{Q}})\geq\left(\frac{\omega_d}{|\wi{\mathrm{Q}}|}\right)^{\frac{2}{d}}J_{\frac{d}{2}-1,1},
	\end{align}
	where $|\wi{\mathrm{Q}}|$ is the volume of $\wi{\mathrm{Q}}$, $\omega_d$ is the volume of the unit ball in $\R^d$ and $J_{m,1}$ is the first positive zero of the Bessel function $I_m(r)$ (see \cite{AoM}). \\
	
	\textbf{2.} If  $|\wi{\mathrm{Q}}|$ is sufficiently small (or $\wi{\mathrm{Q}}$ is sufficiently ``thin"), then  $\lambda_1^*(\wi{\mathrm{Q}})$ becomes larger and larger. 
\end{remark}

\begin{theorem}\label{propfd}
Let $\f_e\in\H$ and $\y_e\in\D(\A)$ be the solution of the equation \eqref{prop1.2}. Also, let us set $\z(\cdot)=\y(\cdot)-\y_e.$ Then, for $\z(0)\in\H$, the following problem  $\text{for a.e.} \ t>0$ 
	\begin{equation}\label{appl6.1}
		\left\{
		\begin{aligned}
			\frac{\d \z(t)}{\d t}&+\mu\A\z(t)+\wi{\mathcal{B}}(\z(t))+\alpha\z(t)+ \beta\wi{\mathcal{C}}_1(\z(t))+\gamma\wi{\mathcal{C}}_2(\z(t))+k\mathcal{P}(m\z(t))=\boldsymbol{0},\ \\ 
			\z(0)&=\y_0-\y_e,
		\end{aligned}
		\right.
	\end{equation}	
	has unique weak solution for $r>3$ and $r=3$ with $2\beta\mu>1$, satisfying
	\begin{align*}
		\z\in\mathrm{C}([0,T];\H)\cap\mathrm{L}^2(0,T;\V)\cap\mathrm{L}^{r+1}(0,T;\wi\L^{r+1}) \ \text{ with }\ \frac{\d\z}{\d t}\in\mathrm{L}^{\frac{r+1}{r}}(0,T;\H).
	\end{align*}
	 Furthermore, there exists a $\delta>0$ such that the following global exponential stability condition holds:
	\begin{align*}
		\|\z(t)\|_{\H}\leq e^{-\delta t}\|\z(0)\|_{\H}, \text{ for all } t\in[0,T].
	\end{align*}
\end{theorem}

\begin{proof}
	Similar to the proof Proposition \ref{prop3.3}, one can show that the operator $\mu\A+\wi{\mathcal{B}}(\cdot)+ \beta\wi{\mathcal{C}}_1(\cdot)+\alpha\wi{\mathcal{C}}_2(\cdot)+k\mathcal{P}(m\cdot)$  is $m$-accretive in $\H\times\H.$ Moreover, using the similar steps as we have done in the proof of Theorem \ref{mainT}, one can prove that the system \eqref{appl6.1} has a unique weak solution
	\begin{align*}
		\z\in\C([0,T];\H)\cap\mathrm{L}^2(0,T;\V)\cap\mathrm{L}^{r+1}(0,T;\wi\L^{r+1}) \ \text{ with }\ \frac{\d\z}{\d t}\in\mathrm{L}^{\frac{r+1}{r}}(0,T;\H),
	\end{align*} 
	for $r>3$ in $d=2,3$ and $d=r=3$ with $2\beta\mu>1$. We will prove only the stability part. Let us take the inner product with $\z(\cdot)$ in \eqref{appl6.1} to get 
	\begin{align*}
		&\frac{1}{2}\frac{\d}{\d t}\|\z(t)\|_{\H}^2 +\mu\|\nabla\z(t)\|_{\H}^2+\alpha\|\z(t)\|_{\H}^2+\beta(\wi{\mathcal{C}}_1(\z(t)),\z(t)) +k(\mathcal{P}(m\z(t)),\z(t))\nonumber\\&=-(\wi{\mathcal{B}}(\z(t)),\z(t))-\gamma(\wi{\mathcal{C}}_2(\z(t)),\z(t)),
	\end{align*}
	for a.e. $t\in[0,T].$ Using the similar calculations as we have performed  in Proposition \ref{prop3.1}, we obtain 
	\begin{align}\label{3.4*}
		|\langle\mathcal{B}(\z+\y_e)-\mathcal{B}(\y_e),\z\rangle|\leq
		\frac{\mu }{2}\|\nabla\z(t)\|_{\H}^2+\frac{\beta}{4}\||\y_e|^{\frac{r-1}{2}}\z\|_{\H}^2 +\varrho^*\|\z\|_{\H}^2,
	\end{align}
	where $\varrho^*=\frac{r-3}{2\mu(r-1)}\left[\frac{4}{\beta\mu(r-1)}\right]^{\frac{2}{r-3}}$. Furthermore, we have 
	\begin{align}\label{c2.2*}
		|\gamma\langle\mathcal{C}_2(\z+\y_e)-\mathcal{C}_2(\y_e),\z\rangle|\leq
		\frac{\beta}{2}\||\z+\y_e|^{\frac{r-1}{2}}\z\|_{\H}^2+\frac{\beta}{4}\||\y_e|^{\frac{r-1}{2}}\z\|_{\H}^2+(\varrho_1^*+\varrho_2^*)\|\z\|_{\H}^2,
	\end{align}
	where $\varrho_1^*=\left(\frac{2^{q}q|\gamma|(q-1)}{2\beta(r-1)}\right)^{\frac{q-1}{r-q}}\left(\frac{r-q}{r-1}\right)$ and  $\varrho_2^*=\left(\frac{2^{q}q|\gamma|(q-1)}{\beta(r-1)}\right)^{\frac{q-1}{r-q}}\left(\frac{r-q}{r-1}\right)$. From \eqref{C1}, one can deduce 
	\begin{align}\label{3.5*}
		\beta\langle\mathcal{C}_1(\z+\y_e)-\mathcal{C}_1(\y_e),\y_e\rangle\geq \frac{\beta}{2}\left(\||\z+\y_e|^{\frac{r-1}{2}}\z\|_{\H}^2+\||\y_e|^{\frac{r-1}{2}}\z\|_{\H}^2\right).
	\end{align}
	Combining \eqref{3.4*}-\eqref{3.5*}, we obtain from \eqref{appl6.1}
	\begin{align*}
		\frac{1}{2}\frac{\d}{\d t}\|\z(t)\|_{\H}^2+\left(\frac{\mu}{2}\A\z(t)+\alpha\z(t)+k\mathcal{P}(m\z(t)),\z(t)\right) \leq(\varrho^*+\varrho_1^*+\varrho_2^*)\|\z(t)\|_{\H}^2,
	\end{align*}
for a.e. $t\geq0.$	Now,  Proposition \ref{stbl3} and the above inequality imply
	\begin{align*}
		\frac{1}{2}\frac{\d}{\d t}\|\z(t)\|_{\H}^2+\left(\lambda_1^*(\wi{\mathrm{Q}})-\varepsilon\right) \|\z(t)\|_{\H}^2\leq(\varrho^*+\varrho_1^*+\varrho_2^*)\|\z(t)\|_{\H}^2
	\end{align*}
	or
	\begin{align}\label{3.6*}
		\frac{1}{2}\frac{\d}{\d t}\|\z(t)\|_{\H}^2\leq-\left(\lambda_1^*(\wi{\mathrm{Q}}) -\varepsilon-\varrho^*-\varrho_1^*-\varrho_2^*\right)\|\z(t)\|_{\H}^2,
	\end{align}
for a.e. $t\geq0$.	From Remark \ref{rk5}, we see that if $|\wi\Q|$ sufficiently small, then one can make the constant 
	\begin{align*}
		\delta:=\lambda_1^*(\wi{\mathrm{Q}}) -\varepsilon-\varrho^*-\varrho_1^*-\varrho_2^*>0. 
	\end{align*} 
 Thus from \eqref{3.6*}, it is immediate that 
	\begin{align*}
		\|\z(t)\|_{\H}^2\leq e^{-2\delta t}\|\z(0)\|_{\H}^2,
	\end{align*}
	for all $t\geq0.$ This proves that the feedback control $\u(\cdot)$ exponentially stabilizes the equilibrium solution  $\y_e$ of the system \eqref{prop1.1}.
\end{proof}

\end{appendix}

\medskip\noindent\textbf{Acknowledgments:} The first author would like to thank Ministry of Education, Government of India-MHRD for financial assistance. K. Kinra would like to thank the Council of Scientific $\&$ Industrial Research (CSIR), India for financial assistance (File No. 09/143(0938)/2019-EMR-I).

%
%
%
%
%
%

\end{document}